\documentclass[aos]{imsart}
\usepackage[english]{babel}
\usepackage[utf8]{inputenc}
\usepackage{amsmath,amsthm, amssymb, latexsym, color,mathtools}
\usepackage{enumerate}
\usepackage{verbatim}
\usepackage{comment}
\usepackage{bbm}
\usepackage{mathrsfs}
\usepackage{tikz}
\usetikzlibrary{cd}
\usetikzlibrary{shapes,arrows,cd}
\usepackage{tikz-cd}
\usepackage[all,2cell]{xy}
\usepackage{abstract}
\usepackage{hyperref}
\usepackage[textsize=tiny]{todonotes}
\usepackage{standalone}
\usepackage{accents}
\usepackage{microtype}

\RequirePackage[colorlinks,citecolor=blue,urlcolor=blue]{hyperref}

\startlocaldefs
\definecolor{org}{rgb}{1,0.53,0.0}

\definecolor{ssw}{rgb}{0.1,0.45,0.1}

\newtheorem{thm}{Theorem}
     
\newtheorem{lem}[thm]{Lemma}

\newtheorem{rem}[thm]{Remark}
\newtheorem{prop}[thm]{Proposition}

\newtheorem{ass}[thm]{Assumption}
\numberwithin{thm}{section}
\numberwithin{equation}{section}

\newcommand{\indicator}{\mathbf{1}}		%indicator function
\newcommand{\identity}{\mathbb{I}}  %identity
\renewcommand{\underbar}[1]{\underaccent{\bar}{#1}}

\newcommand{\data}{\mathbf Y_N}
\newcommand{\D}{\mathrm{d}}

\newcommand{\ProjOne}{\mathcal{P}}
\newcommand{\ProjTwo}{\widetilde{\mathcal{P}}}
\newcommand{\ProjThree}{\mathcal{P}}

\NewDocumentCommand{\op}{o}{%
  \IfNoValueTF{#1}%
    {\mathcal{L}}%
    {\mathcal{L}\!\left[#1\right]}%
}   %operator macro
\NewDocumentCommand{\Iop}{m o}{%
  \IfNoValueTF{#2}%
    {\mathcal{I}_{#1}}%
    {\mathcal{I}_{#1}\!\left[#2\right]}%
}   %Information operator macro
\NewDocumentCommand{\opDeriv}{m o}{%
  \IfNoValueTF{#2}%
    {\mathcal{DL}_{#1}}%
    {\mathcal{DL}_{#1}\![#2]}%
}   %Frechet derivative operator macro
\NewDocumentCommand{\opLin}{o}{%
  \IfNoValueTF{#1}%
    {\tilde{\mathcal{L}}}%
    {\tilde{\mathcal{L}}#1}%
}   %Linearised operator macro
\NewDocumentCommand{\Law}{o m}{
\IfNoValueTF{#1}
{\operatorname{Law}\!\left(#2\right)}
{\operatorname{Law}_{#1}\!\left(#2\right)}
}
\newcommand{\domain}{\mathcal{X}}    %domain macro
\newcommand{\gaussian}{\mathcal{N}_{u_0}} %limiting Gaussian for BvM

\renewcommand{\epsilon}{\varepsilon}

\NewDocumentCommand{\SG}{m}{\textcolor{orange}{[#1]}}

\DeclarePairedDelimiterXPP{\iprodWrapper}[4]
  {\ifblank{#1}{}{{}_{#1}}} % left index
  {\langle}
  {\rangle}
  {\ifblank{#2}{}{_{#2}}}  % right index
  {%
    \ifblank{#3}{\MTemptyplaceholder}{#3},
    \ifblank{#4}{\MTemptyplaceholder}{#4}%
  }

\NewDocumentCommand{\iprod}{s O{} O{} m m}{%
  \IfBooleanTF{#1}
    {\iprodWrapper*{#2}{#3}{#4}{#5}}
    {\iprodWrapper {#2}{#3}{#4}{#5}}%
}
\DeclarePairedDelimiterXPP\normWrapper[2]{}\lVert\rVert{#1}{\ifblank{#2}{\MTemptyplaceholder}{#2}}
\NewDocumentCommand\norm{ s o m }{
	\IfBooleanTF {#1}
	{ \normWrapper*{\IfNoValueF{#2}{_{#2}}}{#3}}
	{ \normWrapper {\IfNoValueF{#2}{_{#2}}}{#3}}
}

\DeclarePairedDelimiterXPP\EVWrapper[2]{#1}[]{}{
	\renewcommand\given{\mathrel{}\mathclose{}\delimsize\vert\mathopen{}\mathrel{}}
	#2
}
\providecommand\given{}
\NewDocumentCommand\EV{ s O{} O{} O{} m }{
	\ifblank {#5}{\mathrm{E}_{#2}^{#3}}
		{\IfBooleanTF {#1}
			{ \EVWrapper*{\mathrm{E}_{#2}^{#3}}{#5} }
			{ \IfNoValueTF{#4}
				{
					\EVWrapper{\mathrm{E}_{#2}^{#3}}{#5}
				}
				{
					\EVWrapper[#4]{\mathrm{E}_{#2}^{#3}}{#5}
				}
			}
		}
}

\DeclarePairedDelimiter\abs\lvert\rvert
\reDeclarePairedDelimiterInnerWrapper\abs{star}{%
	\mathopen{#1\vphantom{\MTkillspecial{#2}}\kern-\nulldelimiterspace\right.}%
	\ifblank{#2}{\MTemptyplaceholder}{#2}%
	\mathclose{\left.\kern-\nulldelimiterspace\vphantom{\MTkillspecial{#2}}#3}%
}

\DeclarePairedDelimiterXPP\ProbWrapper[2]{#1}(){}{
	\renewcommand\given{\nonscript\:\delimsize\vert\nonscript\:\mathopen{}}
	#2
}

\NewDocumentCommand\prob{ s O{} O{} O{} m }{
	\ifblank {#5}{\mathrm{P}_{#2}^{#3}}
		{\IfBooleanTF {#1}
			{ \ProbWrapper*{\mathrm{P}_{#2}^{#3}}{#5} }
			{ \IfNoValueTF{#4}
				{
					\ProbWrapper{\mathrm{P}_{#2}^{#3}}{#5}
				}
				{
					\ProbWrapper[#4]{\mathrm{P}_{#2}^{#3}}{#5}
				}
			}
		}
}

\newcommand{\MTemptyplaceholder}{\:\cdot\:}
\newcommand{\smallo}{
	  \mathchoice
	    {{\scriptstyle\mathcal{O}}}% \displaystyle
	    {{\scriptstyle\mathcal{O}}}% \textstyle
	    {{\scriptscriptstyle\mathcal{O}}}% \scriptstyle
	    {\scalebox{.7}{$\scriptscriptstyle\mathcal{O}$}}%\scriptscriptstyle
	  }	%kleines mathcal o

\endlocaldefs

\begin{document}
\begin{frontmatter}
	\title{Bernstein--von Mises theorems for Bayesian probabilistic numerics}

\begin{aug}
\author[A]{%
  \fnms{Sascha}~\snm{Gaudlitz}%
  \ead[label=e1]{sascha.gaudlitz@epfl.ch}%
}

\author[B]{%
  \fnms{and Sven}~\snm{Wang}%
  \ead[label=e2]{sven.wang@epfl.ch}%
}

\address[B]{%
  Institute of Mathematics,
  EPFL%
  \printead[presep={;\ }]{e2}%
}
\address[A]{%
  Institute of Mathematics,
  \'Humboldt University Berlin%
  \printead[presep={;\ }]{e1}%
}

\end{aug}
    
    \begin{abstract}
    We study probabilistic numerical methods for solving nonlinear PDEs from a Bayesian nonparametric perspective. Given noisy evaluations at random collocation points, we place a truncated Gaussian series prior on the unknown solution and establish contraction at the minimax nonparametric rate, up to a logarithmic factor. Our main results give Gaussian approximations of the posterior in positive-order Sobolev spaces and, under suitable conditions, in the uniform topology. This contrasts with classical ill-posed inverse problems, where Bernstein--von Mises theorems typically require substantially weaker topologies. Here, the observation operator is differential rather than smoothing, and inversion of its linearisation gains regularity, making these strong-topology results possible. The posterior may be centred at either the posterior mean or the posterior mode. We further prove that the Gaussian Laplace approximation is asymptotically equivalent to the true posterior at a $\sqrt{N}$-scale.
\end{abstract}

\end{frontmatter}

\section{Introduction}
Many numerical procedures for partial differential equations naturally involve uncertainty. Forcing terms or boundary data may be observed with statistical errors, the differential equation may be enforced only at finitely many locations, and its solution must be represented in a finite-dimensional approximation space. Probabilistic numerical methods seek to represent these sources of uncertainty through a Bayesian posterior distribution on the unknown solution, rather than through a single point approximation. This perspective goes back at least to Diaconis's formulation of Bayesian probabilistic numerics \cite{diaconis1988bayesian}, see \cite{hennig2015probabilistic,cockayneBayesianProbabilisticNumerical2019, owhadi2019statistical} for more recent accounts and \cite{owhadi2019operator} for a historical review.

Our paper is motivated in particular by Gaussian-process (GP) methods for solving nonlinear PDEs as proposed by \cite{chen2021solving}. These methods model the unknown solution $u$ by a prior GP, which is conditioned to satisfy the PDE constraint at finitely many collocation points, and gives rise to a Bayesian posterior distribution on the solution space of the PDE. In \cite{chen2021solving}, the PDE solution is approximated by the maximum a posteriori estimator (MAP). Subsequent work developed sparse implementations \cite{mengSparseGaussianProcesses2023, chen2025sparse} and proved error estimates for the MAP approximation  \cite{batlleErrorAnalysisKernel2025, chen2025gaussian}; see also \cite{owhadi2019operator} for the connection of GP methods to optimal recovery. Related noisy-collocation methods based on physics-informed neural networks were introduced in \cite{raissiPhysicsinformedNeuralNetworks2019}, and posterior contraction for Bayesian neural-network variants has recently been studied in \cite{sun2024estimationratebayesianpinn,zhao2026posterior}. These contributions concern either deterministic or MAP approximation, or posterior contraction under neural-network priors. The asymptotic shape and frequentist calibration of the full posterior distribution under noisy collocation remain largely open.

The main goal of this paper is to develop rigorous theoretical guarantees for GP-based PDE solution methods by studying them through the lens of Bayesian nonparametric statistics and the infinite-dimensional Bernstein--von Mises (BvM) phenomenon \cite{CN13,castilloBernsteinMisesTheorem2015,GV17, N22,nicklBernsteinvonMisesTheorems2025}. Posterior contraction quantifies the rates at which the posterior distribution concentrates around the true solution \cite{ghosalConvergenceRatesPosterior2000a,shenRatesConvergencePosterior2001a,ghosalConvergenceRatesPosterior2007a,vandervaartRatesContractionPosterior2008, GV17}, but does not describe its asymptotic shape. In finite-dimensional regular models, the BvM theorem states that, after appropriate centring and rescaling, the posterior is asymptotically Gaussian with covariance given by the inverse Fisher information \cite{Vaart1998}. Function-space versions are substantially more delicate: for direct observations, the canonical Gaussian limit is generally tight only in sufficiently weak spaces \cite{freedman1999wald, CN13, CN14}, while smoothing forward maps in ill-posed inverse problems entail a further loss of information \cite{N22}. The mechanism in the present setting is different. Whereas smoothing forward operators induce an information loss and lead to weaker BvM topologies, the operator $\op$ is differential. Inversion of its linearisation gains regularity and allows for a Gaussian approximation of the full posterior in positive-order Sobolev spaces, including  $L^2$ and, in suitable regimes, $L^\infty$. The resulting strong-topology phenomenon complements \cite{nicklBernsteinvonMisesTheorems2025}, where the corresponding regularity instead arises from parabolic smoothing at positive times.

\subsection{Main contributions}

\begin{figure}[h]
\centering
\includegraphics[width =0.7 \textwidth]{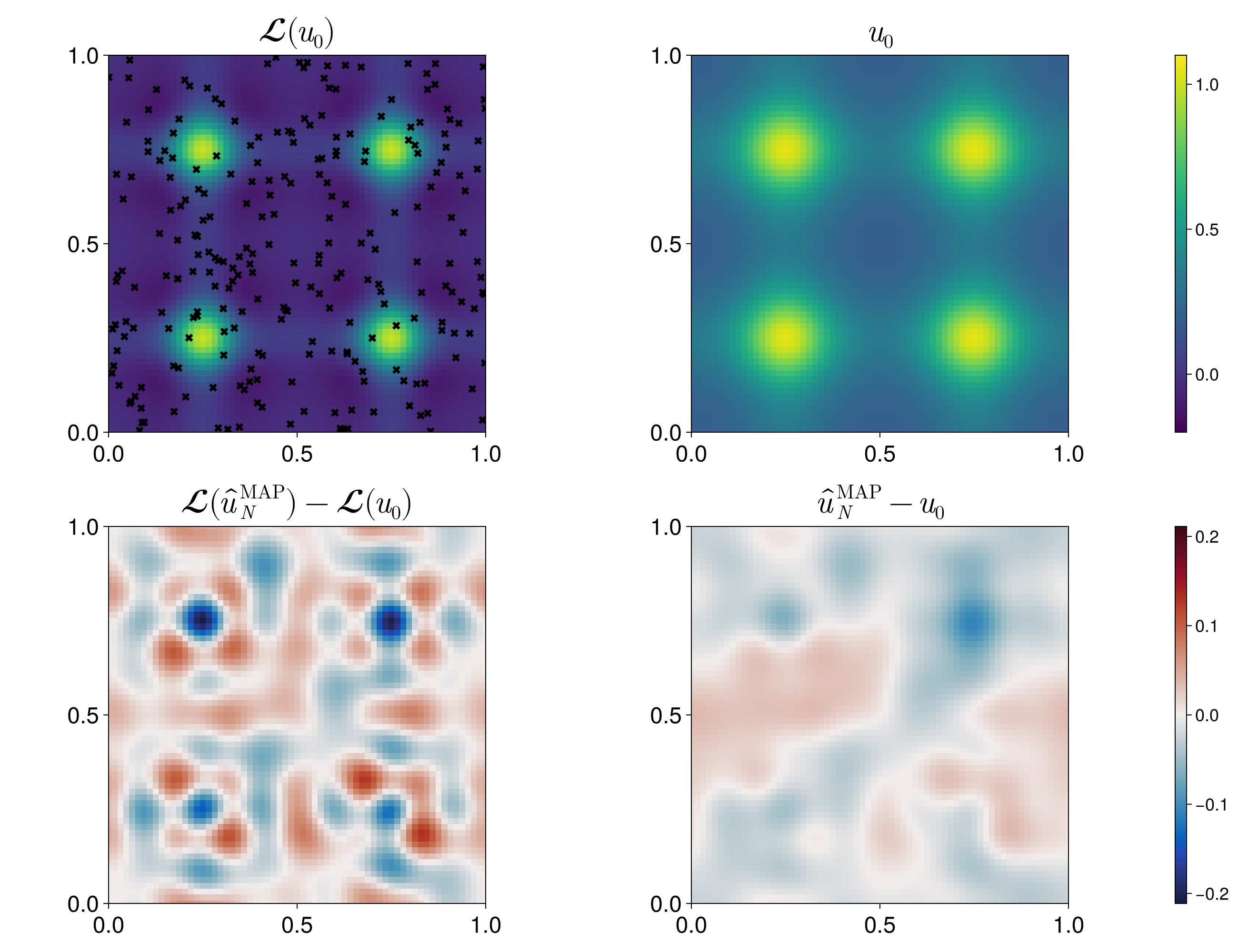}
\caption{Illustration of the reconstruction of the truth $u_0$ and the forward field $\op[u_0]$ for a semilinear elliptic operator $\op$. Black crosses correspond to observed design points $X_i$, $i=1,\dots, N$, $N=250$. See Section \ref{sec:semilienarElliptic} for details.}
\label{fig:Reconstruction}
\end{figure}
In this paper, we prove convergence rates and function-space BvM theorems for a class of Bayesian probabilistic numerical methods for possibly nonlinear PDEs. Let $\domain\subset\mathbb R^d$ be a smooth bounded domain and let $\op\colon \mathcal{V}\to L^2(\domain)$ be a differential operator whose domain $\mathcal{V}$ is a closed subspace of $H^m(\domain)$ for some $m>0$, e.g.\, incorporating prescribed boundary conditions. We consider noisy random evaluations
\begin{equation}\label{eq:intro-observation-model}
    Y_i=\op[u_0](X_i)+\sigma\epsilon_i,
    \qquad i=1,\ldots,N,
\end{equation}
where the design points $(X_i\colon i=1,\dots, N)$ are sampled independently from a density $p$ on $\domain$, bounded above and away from zero, and the variables $\epsilon_i$ are independent standard Gaussians. The goal is to recover the unknown `true' solution $u_0\in\mathcal{V}$ of the differential equation $ \op[u_0]=f$. 
One may interpret \eqref{eq:intro-observation-model} either as noisy random collocation or as a nonlinear statistical inverse problem with forward operator $\op$. It differs, however, from classical \textit{ill-posed} inverse problems \cite{ KNS08, S10, N22, NVW18} where $\op$ is smoothing. In the present setting, we assume a stability estimate of the form
\begin{equation}\label{eq:intro_Stability}
    \norm{u-v}_{H^m(\domain)}\sim \norm{\op[u]-\op[v]}_{L^2(\domain)},\quad u,v\in\mathcal{V}.
\end{equation}
Thus, inversion of $\op$ gains $m$ derivatives, permitting statistical control of the solution in the stronger $H^m(\domain)$-norm.
We verify such estimates for several PDEs in Section \ref{sec:examples}.

We model the unknown solution $u_0$ on general nested approximation spaces $ V_0\subset V_1\subset\cdots\subset\mathcal{V}$, assuming that these spaces admit suitable Jackson and Bernstein estimates standard in multiscale numerical analysis  (see \cite{cohen2003numerical,cohenMultiscaleDecompositionsBounded2000,monasseOrthonormalWaveletBases1998}). For a suitable truncation level  $J=J_N$ we consider a truncated Gaussian series prior with reproducing kernel Hilbert space (RKHS) $\mathbb{H}_N$, which serves simultaneously as a numerical discretisation and as a statistical regularisation.

Figure \ref{fig:Reconstruction} illustrates the procedure for the semilinear elliptic operator $\op[u] = -\kappa\Delta u+\tau(u)$ considered in Section \ref{sec:semilienarElliptic}. The left column displays the true forward field $\op[u_0]$ and the error $\op[\hat{u}_N^{\operatorname{MAP}}]$, while the right column displays $u_0$ and the reconstruction error $\hat{u}_N^{\operatorname{MAP}}-u_0$. The black crosses indicate the sampled collocation points. The stability estimate \ref{eq:intro_Stability} explains that the reconstruction of $u_0$ is visibly better than of the forward field.
%\begin{equation}\label{eq:intro-prior}
%    u=\sum_{|\mu|\leq J}2^{-(m+\beta_0)|\mu|}Z_\mu\psi_\mu,
%    \qquad
%    Z_\mu\stackrel{\mathrm{i.i.d.}}{\sim}N(0,1),
%\end{equation}
%on $V_J$, 

%If $u_0\in H^{m+\beta}(\domain)$, we choose
%\[    2^J\sim N^{1/(2\beta+d)}. \]
%This scaling balances the approximation error of $V_J$ against the statistical complexity of its $\dim(V_J)\sim 2^{Jd}$ active degrees of freedom.

Our first result establishes posterior contraction around $u_0$ in the $H^m(\domain)$-Sobolev norm, see Theorem \ref{thm:Contraction}. For $u_0\in H^{m+\beta}(\domain)$, the contraction rate in $H^m(\domain)$-topology is $\epsilon_N=N^{-\beta/(2\beta+d)}\log(N)$, which is, up to logarithmic factors, the minimax rate for estimating $\beta$-smooth functions in $d$ dimensions \cite{T09}. The posterior mean and the MAP attain the same convergence rate. The proof first establishes contraction in the statistical `information distance' $\norm{\op[u]- \op[u_0]}_{L^2_p(\domain)}$, where $L_p^2(\domain)$ denotes the $p$-weighted $L^2$-space over $\domain$, and then transfers this result to $H^m(\domain)$ by the stability \eqref{eq:intro_Stability}. %, which requires control of an $L^\infty$ envelope of $\mathcal L[u]-\mathcal L[u_0]$ which are achieved for the truncated prior \eqref{eq:intro-prior} via Bernstein estimates on the subspaces $V_J$.

The main results describe the finer local asymptotic shape of the posterior distribution in BvM theorems, see Section \ref{sec:BvM}. The linearisation $\opDeriv{u_0}$ of $\op$ at $u_0$ induces the `LAN inner product' (cf.~\cite{castilloBernsteinMisesTheorem2015, N22})
\begin{equation}\label{eq:intro-lan-inner-product}
    \langle h_1,h_2\rangle_{\mathrm{LAN}}=\frac{1}{\sigma^2}\iprod{\opDeriv{u_0}[h_1]}{\opDeriv{u_0}[h_2]
    }_{L^2_p(\domain)},\quad h_1,h_2\in\mathcal{V},
\end{equation}
characterising the asymptotic covariance of the limiting Gaussian law. Under natural `gradient stability' conditions for $\opDeriv{u_0}$, the LAN norm is equivalent to the $H^m(\domain)$-norm.
%and the associated information operator
%\[
%    \mathcal{I}_{u_0}\colon \mathcal{V}\to\mathcal{V}^\ast,
%    \qquad
%    (\mathcal{I}_{u_0}h_1)(h_2)
%    =
%    \iprod{h_1}{h_2}_{\operatorname{LAN}},\quad h_1,h_2\in\mathcal{V},
%\]
%thus forms an isomorphism.
 %In this sense, the limiting geometry is supplied by the linearised differential equation, rather than by the particular multiscale prior.
The canonical Gaussian process associated with \eqref{eq:intro-lan-inner-product} is generally not tight in the natural $H^m(\domain)$-topology, but becomes a Gaussian random element in the weaker $H^\gamma(\domain)$-topology whenever $0<\gamma<m-d/2$.
%The loss of $d/2$ derivatives is a sufficient condition for the embedding of the LAN Hilbert space into $H^\gamma(\domain)$ to be Hilbert--Schmidt.
Assuming suitable local differentiability of $\op$ and sufficiently high regularity of $u_0$, we prove (cf.~Theorem \ref{thm:BvM_RandomDesign}) the following BvM result in the Wasserstein-$1$ metric with respect to the $H^\gamma(\domain)$-norm:
\begin{equation}\label{eq:intro-bvm}
    \mathcal{W}_{1,H^\gamma(\domain)}\left(\Pi_N\bigl(\sqrt N(\cdot -\bar{u}_N)\mid \data\bigr),\gaussian\right)\xrightarrow{\prob[u_0][N]{}} 0,
\end{equation}
as $N\to \infty$,
where $\bar{u}_N$ denotes the posterior mean and $\Pi_N(\sqrt N(\cdot -\bar{u}_N)\mid \data)$ denotes the rescaled and recentred posterior distribution of $\sqrt N(u -\bar{u}_N)$. Hence the joint posterior fluctuations of the entire function, instead of finitely many coordinates, are asymptotically Gaussian at the parametric scale $N^{-1/2}$. In particular, the result implies the convergence for the $L^2(\domain)$-topology when $m>d/2$ and for $L^\infty(\domain)$ when $d/2<m-d/2$. This strong-topology result does not contradict the earlier impossibility results for direct-observation models \cite{freedman1999wald, CN13}. In the direct model one has $\op=\identity$, corresponding to $m=0$, whereas here the differential operator $\op$ strengthens the information norm by $m$ derivatives.
We further establish a Wasserstein-1 functional `Delta method', see Theorem \ref{thm:nonlinear_wasserstein_delta_method}. For suitably regular Fréchet-differentiable maps $\Phi\colon\mathbb{V}_\gamma\to V$ into a separable Banach space $V$, the posterior distribution of $\sqrt{N}(\Phi(u)-\Phi(\bar{u}_N))$ converges to a Gaussian law in Wasserstein-1 distance. % It therefore transfers the function-space BvM theorem to smooth quantities of interest derived from the numerical solution.

For computational purposes, it is desirable to centre the posterior at the MAP, which is given by the high-dimensional regularised least-squares problem
\begin{equation*}
    \hat{u}_N^{\operatorname{MAP}}\in \operatorname*{argmax}_{u\in V_J}\Big\{-\frac{1}{2\sigma^2}\sum_{i=1}^N (Y_i-\op[u](X_i))^2 - \frac{1}{2}\norm{u}_{\mathbb{H}_N}^2\Big\}.
\end{equation*}
Under second-order regularity assumptions, we prove that the BvM theorem remains valid
with $\bar{u}_N$ replaced by $\hat{u}_N^{\operatorname{MAP}}$, see
Theorem \ref{thm:BvM_MAP}. Moreover, we show that
\begin{equation*}
    \sqrt{N}\mathcal{W}_{1,H^\gamma(\domain)}\left(\Pi_N(\MTemptyplaceholder|\data),\hat{\Pi}_N^{\operatorname{Lap}}(\MTemptyplaceholder|\data)\right)\xrightarrow{\prob[u_0][N]{}}0,
\end{equation*}
as $N\to\infty$, where $\hat{\Pi}_N^{\operatorname{Lap}}(\MTemptyplaceholder|\data)$ is the Gaussian Laplace approximation centred at the MAP, with covariance given by the inverse observed posterior Hessian. Related Gaussian approximation results for finite-dimensional target parameters or growing-dimensional Euclidean models are given in
\cite{spokoinyBernsteinMisesTheorem2014,panovFiniteSampleBernstein2015,katsevichImprovedDimensionDependence2025}. By contrast, Theorems \ref{thm:BvM_MAP} and \ref{thm:LaplaceApproximation} control the full posterior as a probability measure on the fixed Sobolev space $H^\gamma(\domain)$, although asymptotically as $N\to\infty$.

%Theorem \ref{thm:BvM_MAP} is a statistical approximation result and does not provide algorithmic guarantees for locating a global MAP or for sampling from the nonlinear posterior.
Polynomial-time guarantees for optimisation and sampling in related nonlinear inverse problems have been studied in \cite{NW20,bandeira2023free,wangGlobalPolynomialtimeEstimation2026} (see also \cite{nicklBernsteinvonMisesTheorems2025}). Extending these techniques to the present probabilistic numerical setting is an interesting direction for future research.

As key intermediate statements for our BvM theorems, we obtain asymptotic-normality results at the $\sqrt{N}$-scale for both the posterior mean and the MAP estimator, which may be of independent interest. Under an additional approximation condition on the LAN-orthogonal projection of $u_0$ we show that 
\begin{equation*}
    \sqrt N(\bar{u}_N-u_0)\xrightarrow{d} \gaussian,\quad \sqrt N
    \bigl(\hat{u}_N^{\operatorname{MAP}}-u_0
    \bigr)\xrightarrow{d} \gaussian
\end{equation*}
as $N\to\infty$, see Lemmas \ref{lem:Convergence_PosteriorMean}, \ref{lem:Convergence_MAP_Asymptotic_Normality} and Remark \ref{rmk:Bias_Control}.

The assumptions for the main theorems in Section \ref{sec:mainresults} are stated abstractly in order to cover PDEs of different orders and with different nonlinearities. In Section \ref{sec:examples} we verify them for four representative classes: semilinear second-order elliptic equations, linear fourth-order elliptic equations, a one-dimensional stationary transport--reaction equation, and elliptic equations with nonlinear gradient dependence. These examples illustrate the role of regularity estimates for the different PDEs in establishing (global and linearised) stability estimates. They also demonstrate how the order $m$ of the differential operator determines the strength of the topology in which $\sqrt N$-rescaled Gaussian limits exist. In particular, when $d/2<\gamma<m-d/2$, the Sobolev embedding allows us to derive $\sqrt N$-posterior asymptotics in the uniform $L^\infty(\domain)$-norm.

Our proofs are mainly based on tools from Bayesian nonparametrics, in particular those provided by  \cite{GV17, N22,castilloBayesianNonparametricStatistics2024, nicklBernsteinvonMisesTheorems2025}. Posterior contraction is proved using small-ball asymptotics and concentration estimates for Gaussian priors, as well as the global stability of $\op$. For the BvM theorems, we first localise the posterior to an $H^m(\domain)$-ball around the ground truth. Following techniques from \cite{N22,castilloBayesianNonparametricStatistics2024, nicklBernsteinvonMisesTheorems2025} we prove a uniform local asymptotic normality (LAN) expansion for the likelihood, whose remainders can be controlled using empirical process techniques and the Jackson and Bernstein properties of the approximation spaces. A change-of-measure argument then yields convergence of posterior Laplace transforms for finite-dimensional distributions, which can be translated to convergence in Wasserstein distance on $H^\gamma(\domain)$. A distinctive feature of our BvM results is that the posterior can be centred at the optimisation-based MAP. Proving this requires a separate convergence analysis of the MAP and uniform control of the empirical Hessian. The same Hessian estimates show that the negative posterior Hessian at the MAP converges to the information operator. On combination with the MAP-centred BvM theorem, this yields the asymptotic validity of the Gaussian Laplace approximation.

\subsection{Outline} 
The paper is organised as follows. Section \ref{sec:model} introduces the statistical model, the multiscale approximation spaces, and the truncated Gaussian prior. Section \ref{sec:mainresults} establishes posterior contraction and convergence of the posterior mean and MAP, followed by the function-space BvM theorems, the nonlinear delta method, and MAP centring. Section \ref{sec:examples} verifies the assumptions for four classes of differential equations and presents a numerical illustration. The proofs and supporting technical results are collected in the appendices.

\pagebreak

\section{Statistical model, prior and posterior}\label{sec:model}

\subsection{Notation}

For a metric space $(X,d)$ we denote by $B(a,d,L)$ the closed ball in $(X,d)$ with centre $a\in X$ and radius $L>0$. We denote by $\identity$ the identity operator on $X$. For a measure $\mu$ on $X$ and a measurable map $f\colon X\to Y$ for another metric space $(Y,d_Y)$, we denote by $f_{\#}\mu$ the pushforward of $\mu$ under $f$ on $Y$. For a linear bounded operator $A$ between two Banach spaces $A_1$ and $A_2$ we denote by $\norm{A}_{\mathcal{L}(A_1,A_2)}$ its operator norm.
We write $a\lesssim b$ (or $b\gtrsim a$) if there exists a constant $0<C<\infty$ depending only on nonasymptotic quantities such that $a\le Cb$ and $a\sim b$ if $a\lesssim b$ and $a\gtrsim b$.
With $\xrightarrow{d}$ we denote convergence in distribution and with $\xrightarrow{\prob{}}$ convergence in $\prob{}$-probability.
We use the notation $\mathcal{O}_{\prob{}}$ for boundedness in $\prob{}$-probability and $\smallo_{\prob{}}$ for convergence in $\prob{}$-probability. For functions $f,g$ on $\domain$ and $a\in\mathbb{R}^N$ we use the shorthand notations
\begin{equation*}
    \iprod{f}{g}_N\coloneqq\frac{1}{N}\sum_{i=1}^N f(X_i)g(X_i),\quad \iprod{a}{f}_N\coloneqq \frac{1}{N}\sum_{i=1}^N a_if(X_i)
\end{equation*}
for the empirical inner products. We also introduce $(x)_+\coloneqq \max(0,x)$ for $x\in\mathbb{R}$.

\subsection{Statistical setting}

Let $\domain\subset\mathbb{R}^d$ be a bounded domain with smooth boundary. We write $H^s(\domain)$, $s\in\mathbb{R}$, for the usual $L^2$-based Sobolev spaces, $W^{k,\infty}(\domain)$, $k\in\mathbb{N}_0$, for the integer-order Sobolev spaces based on $L^\infty$, and $C(\overline{\domain})$ for the space of continuous functions on the closure of $\domain$. Fix $m\in\mathbb{N}$ and let $\mathcal{V}$ be a closed linear subspace of $H^m(\domain)$, endowed with the inherited $H^m(\domain)$-norm. We consider a possibly nonlinear differential operator
\[\op\colon \mathcal{V}\to L^2(\domain).\]
The space $\mathcal{V}$ encodes the prescribed homogeneous boundary conditions; for example, one may take $\mathcal{V}=H^2(\domain)\cap H_0^1(\domain)$ for homogeneous Dirichlet conditions or $\mathcal{V}
=\bigl\{u\in H^2(\domain):\partial_\nu u=0
\text{ on }\partial\domain\bigr\}$ for homogeneous Neumann conditions. In particular, the approximation spaces introduced below are subspaces of $\mathcal{V}$ and therefore inherit the boundary conditions. Known and sufficiently smooth inhomogeneous boundary conditions can be reduced to this setting by a standard `lifting' argument. Throughout the paper we also assume, that for some $d/2<\kappa_0<\beta$, the map $\op\colon\mathcal{V}\cap H^{m+\kappa_0}(\domain)\to H^{\kappa_0}(\domain)$ is locally bounded and continuous, such that $\op[u]\in C(\domain)$ for any $u\in \mathcal{V}\cap H^{m+\kappa_0}(\domain)$ by the Sobolev embedding.

For some $u_0\in \mathcal{V}$ with $\op[u_0]\in C(\domain)$ we are given observations $\data= ((X_i,Y_i)\colon i=1,\dots, N)$
with
\begin{equation*}
    Y_i = \op[u_0](X_i) + \sigma\epsilon_i,\quad i=1,\dots, N,
\end{equation*}
where $\epsilon_i\overset{i.i.d.}{\sim}N(0,1)$, and $X_i\overset{i.i.d.}{\sim} p(x)\,\D x$ for a Lebesgue-density $p$ satisfying $0<\underbar{p}\le p(x)\le \bar{p}<\infty$, $x\in\domain$ for some constants $\underbar{p},\bar{p}$. We assume throughout that $(\epsilon_i\colon i=1,\dots, N)$ and $(X_i\colon i=1,\dots, N)$ are mutually independent measurement errors and design points.

Our goal is to reconstruct $u_0\in \mathcal{V}\cap H^{m+\beta}(\domain)$ when $N\to\infty$ for some regularity index $\beta> d/2$. The log-likelihood $\ell_N$ is given by
\begin{equation}
    \ell_N(u) =-\frac{1}{2\sigma^2}\sum_{i=1}^N (Y_i-\op[u](X_i))^2\label{eq:Likelihood}
\end{equation}
for any $u\in\mathcal{V}$ such that $\op[u]\in C(\domain)$. For any such $u$ we define 
\[\prob[u][]{}\coloneqq \operatorname{Law}((X_1,\op[u](X_1)+\sigma\epsilon_1)) =  \operatorname{Law}((X_1,Y_1)),\] with corresponding product measure $\prob[u][N]{}$, $N\in\mathbb{N}$. We denote by 
\begin{equation*}
    \iprod{u}{v}_{L_p^2(\domain)}\coloneqq \int_{\domain} u(x)v(x)p(x)\,\D x,\quad u,v\in L^2(\domain),
\end{equation*}
the weighted $L^2$-inner product. Since $0<\underbar{p}\le p(x)\le\bar{p}<\infty$, we have the norm equivalence $\sqrt{\underbar{p}}\norm{}_{L^2(\domain)}\le \norm{}_{L_p^2(\domain)}\le \sqrt{\bar{p}}\norm{}_{L^2(\domain)}$.

\subsection{A truncated Gaussian series prior and the posterior distribution}

To construct the prior for $u$, we first introduce a sequence of finite-dimensional approximation spaces compatible with the Sobolev spaces on $\mathcal{V}$ and its prescribed boundary conditions.
\begin{ass}\label{assump:ApproximatonSets}
    Let $V_0\subset V_1\subset \cdots \subset \mathcal{V}$ be finite-dimensional approximation spaces whose union is dense in $(\mathcal{V},\norm{}_{H^m(\domain)})$ and which satisfy $\operatorname{dim}(V_J)\sim 2^{Jd}$, uniformly in $J\in\mathbb{N}$. Moreover, the $\iprod[][L^2(\domain)]{}{}$-orthogonal projection $\ProjOne_J\colon L^2(\domain)\to V_J$, $J\in\mathbb{N}$, satisfies the following $\norm{}_{H^\kappa(\domain)}$-stability for  for any $0\le \kappa<S$,
    \begin{equation}
        \sup_{J\in\mathbb{N}}\norm{\ProjOne_J u}_{H^\kappa(\domain)}\lesssim \norm{u}_{H^\kappa(\domain)},\quad u\in \mathcal{V}\cap H^\kappa(\domain).\label{eq:H^m_stability}
    \end{equation}
    Furthermore, the Jackson and Bernstein estimates hold: For some $S>m+\beta+d/2$ and any $0\le s< t<  S$ we have
    \begin{align}
        \norm{(\identity-\ProjOne_J)u}_{H^s(\domain)}&\lesssim 2^{-J(t-s)}\norm{u}_{H^t(\domain)},\quad u\in\mathcal{V}\cap H^t(\domain),\label{eq:Jackson}\\
        \norm{u}_{H^t(\domain)}&\lesssim 2^{J(t-s)}\norm{u}_{H^s(\domain)},\quad u\in V_J.\label{eq:Bernstein}
    \end{align}
    Finally, we assume that $\mathcal V$ admits an $\iprod{}{}_{L^2(\domain)}$-orthonormal `multiscale' basis  $(\psi_\mu)_{\mu\in \Lambda}$ with $\Lambda=\{(j,k), j\in\mathbb{N}, k = 0,\dots, \max(1,2^{jd}-1)\}$  such that: 
    
    \begin{itemize}
        \item[(i)] Each approximation space is given by $V_J = \operatorname{span}(\psi_\mu: \mu \in \Lambda~\text{with}~|\mu|\le J)$, where we write $|\mu|=j$ for $\mu=(j,k)\in\Lambda$.
        \item[(ii)] For any $0\le s<S$ and $0<\epsilon<S-s$ we have 
    \begin{equation}
        c\norm{u}_{H^s(\domain)}^2\le  \sum_{\mu\in\Lambda}2^{2\abs{\mu}s}\iprod{u}{\psi_\mu}_{L^2(\domain)}^2\le C\norm{u}_{H^{s+\epsilon}(\domain)}^2,\quad u\in \mathcal{V}\cap H^{s+\epsilon}(\domain),\label{eq:CharacterisationSobolevSpaces} 
    \end{equation}
    where the constant $C$ may depend on $\epsilon$.
    \end{itemize}
\end{ass}
We refer to \cite{monasseOrthonormalWaveletBases1998,cohenMultiscaleDecompositionsBounded2000,cohen2003numerical} for constructions of such multiscale decompositions.
Having defined the approximation spaces $(V_J\colon J\in\mathbb{N})$ and the basis $(\psi_\mu\colon \mu\in\Lambda)$, we define the prior on $u$ as the law of the following truncated Gaussian series.
For smoothness parameter $\beta_0>d/2$ let
\begin{equation}
    \pi_N = \Law{\sum_{\abs{\mu}\le J}2^{-(m+\beta_0)\abs{\mu}}Z_\mu\psi_\mu},\quad Z_\mu\overset{i.i.d}{\sim}N(0,1),\label{eq:Prior}
\end{equation}
where the cut-off $J=J_N$ depends on $N$.
The characterisation of the Sobolev norm from \eqref{eq:CharacterisationSobolevSpaces}  implies that the Cameron--Martin space of the finite-dimensional Gaussian measure $\pi_N$ is given by $V_J$ with norm
\begin{equation}
    \norm{u}_{\mathbb{H}_N}^2= \sum_{\abs{\mu}\le J}2^{2(m+\beta_0)\abs{\mu}}\iprod{u}{\psi_\mu}_{L^2(\domain)}^2 \le  C\norm{u}_{H^{m+\beta_0+\epsilon}(\domain)}^2,\quad u\in V_J,\label{eq:RKHS_Norm}
\end{equation}
for $0<\epsilon<S-m-\beta_0$ and some constant $C$ depending on $\epsilon$.

Given the Gaussian prior and the likelihood \eqref{eq:Likelihood}, the posterior distribution for any Borel set $A\subset V_J$ is given by
\begin{equation*}
    \Pi_N(A|\data)\coloneqq \frac{\int_A e^{\ell_N(u)}\,\D \pi_N(u)}{\int_{V_J} e^{\ell_N(u)}\,\D \pi_N(u)} = \frac{\int_A e^{-\frac{1}{2\sigma^2}\sum_{i=1}^N(Y_i-\op[u](X_i))^2-\frac{1}{2}\norm{u}_{\mathbb{H}_N}^2}\,\D u}{\int_{V_J} e^{-\frac{1}{2\sigma^2}\sum_{i=1}^N(Y_i-\op[u](X_i))^2-\frac{1}{2}\norm{u}_{\mathbb{H}_N}^2}\,\D u},
\end{equation*}
where $\D u$ denotes the coordinate Lebesgue-measure on $V_J$.
We denote the expectation operator associated to the posterior distribution $\Pi_N(\MTemptyplaceholder|\data)$ by $\EV[\Pi_N]{\MTemptyplaceholder\given\data}$.

\begin{rem}
%Instead of defining $\ProjOne_J\colon \mathcal{V}\to V_J$ as the $\iprod[][L^2(\domain)]{}{}$-orthogonal projection and imposing the uniform $H^m(\domain)$-stability condition \eqref{eq:H^m_stability}, one may define $\ProjOne_J$ as the $\iprod[][H^m(\domain)]{}{}$-orthogonal projection onto $V_J$. In this case, the uniform $H^m(\domain)$-stability holds automatically, with stability constant equal to one.\todo{Is this comment useful to the reader? If not, let's take it out}
The orthonormality requirement in Assumption \ref{assump:ApproximatonSets} can be relaxed to a biorthogonal Riesz basis, see \cite{cohenBiorthogonalBasesCompactly1992,dahmenMultiscaleWaveletMethods2003} for details on the construction of such bases. More precisely, let $(\psi_\mu\colon \mu\in\Lambda)$ and $(\tilde\psi_\mu)_{\mu\in\Lambda}$ be a biorthogonal system satisfying $\iprod{\psi_\mu}{\tilde{\psi}_\nu}_{L^2(\domain)}=\delta_{\mu\nu}$ for any $\mu,\nu\in\Lambda$,
and define
\begin{equation*}
    V_J\coloneqq \operatorname{span}(\psi_\mu\colon \abs{\mu}\le J),\quad \ProjOne_Ju\coloneqq \sum_{\abs{\mu}\le J}\iprod{u}{\tilde{\psi}_\mu}\psi_\mu,\quad u\in\mathcal{V}.
\end{equation*}
Then $\ProjOne_J$ is a projection onto $V_J$, although in general no longer $\iprod{}{}_{L^2(\domain)}$-orthogonal.
It is sufficient to assume that $\ProjOne_J$ satisfies the stability, Jackson and Bernstein estimates from Assumption \ref{assump:ApproximatonSets} and that the Sobolev-norm characterisation \eqref{eq:CharacterisationSobolevSpaces} takes the form 
    \begin{equation*}
        c\norm{u}_{H^s(\domain)}^2\le  \sum_{\mu\in\Lambda}2^{2\abs{\mu}s}\iprod{u}{\tilde{\psi}_\mu}_{L^2(\domain)}^2\le C\norm{u}_{H^{s+\epsilon}(\domain)}^2,\quad u\in \mathcal{V}\cap H^{s+\epsilon}(\domain).
    \end{equation*}
    for the same $s,\epsilon$ as in Assumption \ref{assump:ApproximatonSets}.
The prior may then be defined via the law of $\sum_{\abs{\mu}\le J}2^{-(m+\beta_0)\abs{\mu}}Z_\mu\psi_\mu$, $Z_\mu\overset{i.i.d.}{\sim}N(0,1)$. If, in addition,
\begin{equation*}
    \sup_{\mu\in\Lambda}\norm{2^{(m-\epsilon_0)\abs{\mu}}\tilde{\psi}_\mu}_{H^{-m}(\domain)}<\infty
\end{equation*}
for some $\epsilon_0>0$, our results generalise to this setting.
\end{rem}

\section{Main results}\label{sec:mainresults}

%\begin{ass}[Assumption on the Wavelet system]\label{assump:WaveletSystem}
%    We need that 
%    \begin{equation*}
%        \norm{f}_{H^s(\domain)}^2 \sim \sum_{(k,j)\in V}2^{2ks}\iprod{f}{\psi_{(k,j)}}_{L^2(\domain)}^2,\quad f\in H^s(\domain), s\ge 0
%    \end{equation*}
%    for a suitable index set $V$ and the fractional Sobolev spaces $H^{s}(\domain)$ are defined via the spectral powers of the Dirichlet Laplace operator. The constant in $\sim$ is not allowed to depend on $f$ or $s$.
%\end{ass}
%Assumption \ref{assump:WaveletSystem} ensures the following conditions
%\SG{Write this sufficiently general}

%\begin{rem}
%    For the statistical model with pointwise and noisy observations from \eqref{eq:Observationmodel_discrete}, the small ball estimate from Theorem \ref{thm:GeneralContraction} requires us to control
%    \begin{equation*}
%\pi_N(B_N)\ge \Pi\left(\norm{\op[u_0]-\op[u]}_{L^2(\domain)}^2\le \frac{\sigma^2}{2}\epsilon_N^2,  \norm{\op[u_0]-\op[u]}_{L^4(\domain)}^4\le 2\sigma^4\epsilon_N^2\right),
%\end{equation*}
%compare Lemma \ref{lem:RandomDesign_Hellinger/KL}. This would require to control $\norm{\op[u]-\op[u_0]}_{L^4(\domain)}$. \SG{Stability results applicable?-> Ja! (theorem 9.19 im Trundiger)->Contraction in Hellinger, but how to go to $L^2$? Or use a rescaled prior?}
%\end{rem}

This section contains the frequentist guarantees for the posterior distribution. We start with posterior contraction and convergence rates for the posterior mean and the MAP. % 
Subsequently, we state our BvM theorems, whose proofs require the contraction and convergence results.

\subsection{Posterior contraction rates}\label{sec:Contraction}

The following assumption on the operator $\op\colon\mathcal{V}\to L^2(\domain)$ combines a global Lipschitz estimate with a quantification of the 'injectivity' of $\op$. Not only does it ensure that the statistical model is identifiable, but it also allows us to transfer contraction rates for recovering $\op[u_0]$ to those for $u_0$. It is commonly imposed for Bayesian inference for inverse problems, compare \cite{NVW18} and Condition 2.1.1 of \cite{N22}.

\begin{ass}[Global stability]\label{assump:Operator_stability}
    There exist constants $0<c\le C<\infty$ such that
        \begin{equation}
            c\norm{u-v}_{H^m(\domain)} \le \norm{\op[u]-\op[v]}_{L_p^2(\domain)}\le C\norm{u-v}_{H^m(\domain)},\quad u,v\in  \mathcal{V}.\label{eq:Stability_Lipschitz_Bounds}
        \end{equation}
        Moreover, we have
        \begin{equation}
            \norm{\op[u]-\op[v]}_{L^\infty(\domain)}\le C \norm{u-v}_{W^{m,\infty}(\domain)},\quad u,v\in\mathcal{V}\cap W^{m,\infty}(\domain).\label{eq:Lipschitz_Bounds_infty}
        \end{equation}
\end{ass}

\begin{thm}[Posterior contraction in $H^m(\domain)$]\label{thm:Contraction}
    Grant Assumptions \ref{assump:ApproximatonSets} and \ref{assump:Operator_stability}. Assume that $u_0\in\mathcal{V}\cap H^{m+\beta}(\domain)$ with $\beta>d/2$ and $d/2<\beta_0<\beta+d/2$.
     Consider the prior $\pi_N$ from \eqref{eq:Prior} with cut-off $2^J\sim N^{1/(2\beta+d)}$. Then there exists a constant $0<D<\infty$ such that for any $L$ large enough we have
    \begin{equation}
        \Pi_N\left(u\colon \norm{u-u_0}_{H^m(\domain)}\le  L \epsilon_N \,\vert\, \data\right)=1 -\mathcal{O}_{\prob[u_0][N]{}}(e^{-D N \epsilon_N^2})\label{eq:Contraction_RandomDesign}        
    \end{equation}
    for $\epsilon_N = N^{-\beta/(2\beta+d)}\log(N)$  as $N\to\infty$.
\end{thm}
\begin{proof}
    See Section \ref{sec:Proofs_Contraction}.
\end{proof}
Theorem \ref{thm:Contraction} shows that the posterior distribution $\Pi_N(\MTemptyplaceholder|\data)$ concentrates around the truth $u_0$ at rate $\epsilon_N$ in $H^m(\domain)$-norm as the sample size $N$ tends to infinity. The global stability of $\mathcal{L}$ provided by Assumption \ref{assump:Operator_stability} is important here, since it transfers a statistical contraction statement for the observed field $\op[u]$ in $L^2(\domain)$-norm into the recovery of the solution $u_0$ in $H^m(\domain)$-norm.
Next we show that the posterior mean $\bar{u}_N\coloneqq \EV[\Pi_N]{u\given\data}$ converges to $u_0$ in $H^m(\domain)$-norm at the minimax-optimal rate (up to $\log$-factors).

\begin{lem}[Nonparametric convergence of the posterior mean]\label{lem:Convergence_PosteriorMean_Nonparametric}
    Grant Assumptions \ref{assump:ApproximatonSets} and \ref{assump:Operator_stability}. Assume that $u_0\in\mathcal{V}\cap H^{m+\beta}(\domain)$ with $\beta>d/2$ and $d/2<\beta_0<\beta+d/2$.
     Consider the prior $\pi_N$ from \eqref{eq:Prior} with cut-off $2^J\sim N^{1/(2\beta+d)}$. Then
     \begin{equation*}
         \norm{\bar{u}_N - u_0}_{H^m(\domain)}=\mathcal{O}_{\prob[u_0][N]{}}(\epsilon_N).
     \end{equation*}
\end{lem}
\begin{proof}
    See Section \ref{sec:Proofs_Contraction}.
\end{proof}
The next result concerns the convergence of the MAP estimator.
For $u\in V_J$ we denote by
\begin{equation*}
    \ell_N^{\operatorname{reg}}(u)\coloneqq \ell_N(u) -\frac{1}{2}\norm{u}_{\mathbb{H}_N}^2
\end{equation*}
the regularised loglikelihood.
The MAP $\hat{u}_N^{\operatorname{MAP}}$ is an element of 
\begin{equation}
    \operatorname*{argmax}_{u\in V_J}\{\ell_N^{\operatorname{reg}}(u)\}=\operatorname*{argmax}_{u\in V_J}\Big\{-\frac{1}{2\sigma^2}\sum_{i=1}^N (Y_i-\op[u](X_i))^2 - \frac{1}{2}\norm{u}_{\mathbb{H}_N}^2\Big\}.\label{eq:MAP_Definition}
\end{equation}
Note that $\ell_N^{\operatorname{reg}}$ is continuous on the finite-dimensional space $V_J$. Since $\ell_N^{\operatorname{reg}}(u)\le -\norm{u}_{\mathbb{H}_N}^2/2\to-\infty$ for $\norm{u}_{\mathbb{H}_N}\to\infty$, a standard localization argument implies that there exists at least one global maximiser. By well-known measurable selection theorems (see \cite{brownMeasurableSelectionsExtrema1973} or \cite{potscher1997dynamic}) we may assume that $\hat{u}_N^{\operatorname{MAP}}$ is a measurable choice of the data $\data$.

The principal computational motivation for considering the MAP estimator $\hat{u}_N^{\operatorname{MAP}}$ is its characterisation in \eqref{eq:MAP_Definition} as the solution of a finite-dimensional (non-convex) optimisation problem. Unlike the posterior mean $\bar{u}_N$, its computation does not require integration with respect to the non-log-concave posterior density. %For nonlinear forward maps, however, the objective is generally non-convex.\todo{Wondering whether we should take out this sentence?} %The results below concern a global optimiser and do not provide guarantees for its numerical computation.

\begin{lem}[Nonparametric convergence of the MAP]\label{lem:Convergence_MAP_Nonparametric}
Grant Assumptions \ref{assump:ApproximatonSets} and \ref{assump:Operator_stability}.
Assume that $u_0\in\mathcal{V}\cap H^{m+\beta}(\domain)$ with $\beta>d/2$ and $d/2<\beta_0<\beta+d/2$.
     Consider the prior $\pi_N$ from \eqref{eq:Prior} with cut-off $2^J\sim N^{1/(2\beta+d)}$. Then the MAP estimator satisfies
    \begin{equation*}
        \norm{\hat{u}_N^{\operatorname{MAP}}-u_0}_{H^m(\domain)}= \mathcal{O}_{\prob[u_0][N]{}}(\epsilon_N).
    \end{equation*}
\end{lem}
\begin{proof}
    Follows from Lemma \ref{lem:MAP_Convergence}.
\end{proof}

\subsection{Nonparametric BvM theorems}\label{sec:BvM}

The posterior contraction of the previous section describes the region in which most posterior mass lies, but not the shape of the posterior inside that region. The BvM theorems below identify this local shape: after centring at a suitable estimator and rescaling by $\sqrt{N}$, the full posterior converges to a Gaussian measure whose covariance is determined by the linearisation of $\mathcal{L}$ at $u_0$. We first impose the first-order regularity conditions needed for a uniform local asymptotic normality (LAN) expansion.
\begin{ass}[Local first-order regularity]\label{assump:operator_linearisation}
        There exists an open neighbourhood $\mathcal{U}$ of $u_0$ in $(\mathcal{V},\norm{}_{H^m(\domain)})$ such that $\op\colon\mathcal{V}\to L^2(\domain)$ is Fréchet differentiable on $\mathcal{U}$ with bounded derivative $\opDeriv{u}\in\mathcal{L}(\mathcal{V},L^2(\domain))$ for any $u\in\mathcal{U}$. Moreover, there exist constants $0<c\le C<\infty$ such that the following conditions hold:
        \begin{enumerate}[(i)]
        \item Gradient stability: For every $u\in\mathcal{V}$ we have 
        \begin{equation}
            c\norm{u}_{H^m(\domain)}\le \norm{\opDeriv{u_0}[u]}_{L_p^2(\domain)}\le  C\norm{u}_{H^m(\domain)}.\label{eq:Graphnormequivalence}
        \end{equation}
        \item Quadratic remainder: For all $u,v\in\mathcal{U}$ we have
        \begin{equation}
            \norm{\op[u]-\op[v]-\opDeriv{v}[u-v]}_{L^2(\domain)} \le C \norm{u-v}_{H^m(\domain)}^2.\label{eq:linearisation_2}
        \end{equation}
        If, additionally, $u,v\in \mathcal{U}\cap W^{m,\infty}(\domain)$, then $\opDeriv{u}$ restricts to a bounded operator from $\mathcal{V}\cap W^{m,\infty}(\domain)$ into $L^\infty(\domain)$ and
        \begin{equation}
            \norm{\op[u]-\op[v]-\opDeriv{v}[u-v]}_{L^\infty(\domain)} \le C \norm{u-v}_{W^{m,\infty}(\domain)}^2.\label{eq:linearisation_infty}
        \end{equation}
        \item Lipschitz continuity of the derivative at $u_0$: For any $h_1,h_2\in\mathcal{V}$ with $u_0+h_1\in\mathcal{U}$ we have
                \begin{align}
        \begin{split}
            \norm{(\opDeriv{u_0}-\opDeriv{u_0+h_1})[h_2]}_{L^2(\domain)}&\le C\norm{h_1}_{H^m(\domain)}\norm{h_2}_{H^m(\domain)},\\
            \norm{(\opDeriv{u_0}-\opDeriv{u_0+h_1})[h_2]}_{L^\infty(\domain)}&\le C\norm{h_1}_{W^{m,\infty}(\domain)}\norm{h_2}_{W^{m,\infty}(\domain)},
            \end{split}\label{eq:linearisation2}
        \end{align}
        where the second inequality is required when $h_1,h_2\in \mathcal{V}\cap W^{m,\infty}(\domain)$ and $u_0+h_1\in \mathcal{U}\cap W^{m,\infty}(\domain)$.
        \end{enumerate}
\end{ass}
Apart from the gradient-stability condition \eqref{eq:Graphnormequivalence}, the remaining conditions in Assumption \ref{assump:operator_linearisation} are implied if $\op$ is continuously Fréchet differentiable near $\op$ with locally Lipschitz derivative.
We refer to \cite{castreGradientStabilityNonlinear2026} for an in-depth discussion of the gradient stability property \eqref{eq:Graphnormequivalence} which was first used in \cite{NW20} to show local curvature of the posterior density in inverse problems.

We define the LAN inner product on $\mathcal{V}$ by
\begin{equation*}
    \iprod[][\operatorname{LAN}]{u}{v}\coloneqq \frac{1}{\sigma^2}\iprod[][L_p^2(\domain)]{\opDeriv{u_0}[u]}{\opDeriv{u_0}[v]},\quad u,v\in \mathcal{V}.
\end{equation*}
In view of \eqref{eq:Graphnormequivalence} we immediately deduce the norm equivalence
\begin{equation}
    c\norm{u}_{H^m(\domain)}\le \norm{u}_{\operatorname{LAN}}\le  C\norm{u}_{H^m(\domain)},\quad u\in\mathcal{V}\label{eq:Normequivalence}
\end{equation}
for constants $0<c\le C<\infty$ depending on $\sigma$ and $\op$. 
Let $\mathcal{V}^\ast$ be the topological dual space to $\mathcal{V}$ and define the linear operator $\mathcal{I}_{u_0}\colon \mathcal{V}\to\mathcal{V}^\ast$ via the identity
    \begin{equation*}
        (\mathcal{I}_{u_0}u)(v) \coloneqq \iprod{u}{v}_{\operatorname{LAN}}= \frac{1}{\sigma^2}\iprod{\opDeriv{u_0}[u]}{\opDeriv{u_0}[v]}_{L_p^2(\domain)},\quad u,v\in\mathcal{V}.
    \end{equation*}
    Then $\mathcal{I}_{u_0}\colon \mathcal{V}\to\mathcal{V}^\ast$ is a topological isomorphism by \eqref{eq:Graphnormequivalence} and the Lax-Milgram Theorem (Theorem 9.14 of \cite{renardyIntroductionPartialDifferential2004a}).

Next we define the limiting Gaussian measure. Note that $(\mathcal{V},\iprod{}{}_{\operatorname{LAN}})$ is a separable Hilbert space by \eqref{eq:Normequivalence} and \eqref{eq:Graphnormequivalence}.
Let $\gaussian$ be a Gaussian isonormal process indexed by $(\mathcal{V},\iprod{}{}_{\operatorname{LAN}})$, i.e. a centred Gaussian process indexed by $\mathcal{V}$ with covariance
\begin{equation*}
    \EV{\gaussian(\phi)\gaussian(\psi)} = \iprod{\phi}{\psi}_{\operatorname{LAN}},\quad \phi,\psi\in\mathcal{V}.
\end{equation*}
Note that $\gaussian$ cannot be realised as a (tight) Gaussian measure on $\mathcal{V}$, since its identity covariance operator is not trace class with respect to $\iprod{}{}_{\operatorname{LAN}}$. However, in the present setting we may alternatively view $\gaussian$ as a tight Gaussian measure supported on a larger `ambient' Sobolev space with weaker topology $H^\gamma(\domain)$ for $\gamma<m-d/2$. In slight abuse of notation, we shall use the same notation $\gaussian$ for both those interpretations of the Gaussian limiting measure. We elaborate this technical point in the following remark.
\begin{rem}\label{rmk:Tightness_Limiting_Measure}
For $m>d/2$ and $0<\gamma<m-d/2$ denote by $\mathbb{V}_\gamma$ the closure of $\mathcal{V}$ with respect to $\norm{}_{H^\gamma(\domain)}$. Since $\norm{}_{\operatorname{LAN}}$ and $\norm{}_{H^m(\domain)}$ are equivalent norms by \eqref{eq:Normequivalence} and recalling the Hilbert-Schmidt properties of Sobolev embeddings, we deduce that the embedding $\iota\colon (\mathcal{V},\norm{}_{\operatorname{LAN}})\hookrightarrow(\mathbb{V}_\gamma, \norm{}_{H^\gamma(\domain)})$ is a Hilbert-Schmidt operator. Consequently, the sum $X\coloneqq \sum_{k\in\mathbb{N}} \gaussian(e_k)\iota [e_k]$ for an orthonormal basis of $(\mathcal{V},\iprod{}{}_{\operatorname{LAN}})$ converges in $L^2(\Omega,(\mathbb{V}_\gamma,\norm{}_{H^\gamma(\domain)}))$ and defines a Gaussian random variable in $(\mathbb{V}_\gamma,\norm{}_{H^\gamma(\domain)})$. Since for any $\psi\in \mathbb{V}_\gamma$ we have $\iprod{X}{\psi}_{H^\gamma(\domain)} = \gaussian(\iota^\ast\psi)$ in $L^2(\Omega)$, it follows that
\begin{equation*}
    \EV{\iprod{X}{\phi}_{H^\gamma(\domain)}\iprod{X}{\psi}_{H^\gamma(\domain)}}=\EV{\gaussian(\iota^\ast \phi)\gaussian(\iota^\ast \psi)}= \iprod{\iota^\ast \phi}{\iota^\ast\psi}_{\operatorname{LAN}} = \iprod{\iota\iota^\ast\phi}{\psi}_{H^\gamma(\domain)}
\end{equation*}
for any $\phi,\psi\in \mathbb{V}_\gamma$. Consequently, the law 
of $X$ on $(\mathbb{V}_\gamma,\norm{}_{H^\gamma(\domain)})$ has trace-class covariance operator $\iota\iota^\ast$. Fernique's Theorem (Theorem 2.7 of \cite{DaPrato2014}) implies that $\EV[X\sim\gaussian]{\norm{X}_{H^{\gamma}(\domain)}^p}<\infty$ for any $p\ge 0$.
\end{rem}

Following Definition 6.1 of \cite{villaniOptimalTransport2009}, for a Polish space $(\domain,d)$ we denote by $\mathcal{W}_{1,\domain}$ the Wasserstein-1 distance
\begin{equation*}
    \mathcal{W}_{1,\domain}(\mu,\nu) = \inf\left\{ \EV{d(X,Y)}\colon \operatorname{Law}(X)=\mu, \operatorname{Law}(Y)=\nu\right\}
\end{equation*}
between two probability measures $\mu,\nu$ on $\domain$ with finite first moments. Recall that $\Pi_N(\MTemptyplaceholder|\data)$ denotes the posterior distribution and that $\bar{u}_N=\EV[\Pi_N]{u\given\data}$ denotes the posterior mean. For $J\in\mathbb{N}$ let $\ProjTwo_{J}\colon \mathcal{V}\to V_J$ be the $\iprod{}{}_{\operatorname{LAN}}$-orthogonal projection onto the approximation space $V_J$, see Lemma \ref{lem:Properties_LAN_Norm} for details.

\begin{thm}[BvM theorem centred at the posterior mean]\label{thm:BvM_RandomDesign}
    Grant Assumptions \ref{assump:ApproximatonSets}, \ref{assump:Operator_stability} and \ref{assump:operator_linearisation}.
    Let $m>d/2$, $0<\gamma<m-d/2$, consider the prior $\pi_N$ from \eqref{eq:Prior} with cut-off $2^J\sim N^{1/(2\beta+d)}$ and assume that $u_0\in \mathcal{V}\cap H^{m+\beta}(\domain)$ for $\beta>2d$, $d/2<\beta_0<\beta+d/4$. Define the (random) map $\tau_N(u)\coloneqq \sqrt{N}(u-\bar{u}_N)$ for $u\in V_J$. Then
    \begin{equation*}
        \mathcal{W}_{1,H^{\gamma}(\domain)}\left((\tau_N)_{\#}\Pi_N(\MTemptyplaceholder|\data), \gaussian\right) \xrightarrow{\prob[u_0][N]{}}0, \quad N\to\infty.
    \end{equation*}
\end{thm}
\begin{proof}
    See Section \ref{sec:Proofs_BvM}.
\end{proof}

Theorem \ref{thm:BvM_RandomDesign} is an infinite-dimensional statement describing the joint posterior fluctuations for the whole conditional random function $u|\data$, rather than for individual marginals. The loss of a little more than $d/2$-Sobolev regularity is required to make the limiting Gaussian measure $\gaussian$ tight, see Remark \ref{rmk:Tightness_Limiting_Measure}. The regularising effect of order $m$ of the inverse of $\opDeriv{u_0}$ enables measuring the posterior fluctuations in a positive Sobolev norm up to order $m-d/2>0$. If moreover $d/2<\gamma<m-d/2$, we obtain a BvM in the uniform $\norm{}_{L^\infty(\domain)}$ topology, allowing for uncertainty quantification for pointwise evaluations of $u$.

From Theorem \ref{thm:BvM_RandomDesign}, one can deduce asymptotic normality of the posterior mean.
\begin{lem}[Convergence of posterior mean]\label{lem:Convergence_PosteriorMean}
    Grant Assumptions \ref{assump:ApproximatonSets}, \ref{assump:Operator_stability} and \ref{assump:operator_linearisation}. Let $m>d/2$, $0<\gamma<m-d/2$, consider the prior $\pi_N$ from \eqref{eq:Prior} with cut-off $2^J\sim N^{1/(2\beta+d)}$ and assume that $u_0\in \mathcal{V}\cap H^{m+\beta}(\domain)$ for $\beta>2d$, $d/2<\beta_0<\beta+d/4$.
   Then
    \begin{equation*}
        \sqrt{N}(\bar{u}_N-\ProjTwo_J u_0)\xrightarrow{d}\gaussian
    \end{equation*}
    under $\prob[u_0][N]{}$ as $N\to\infty$. If, additionally, $\sqrt{N}\norm{(\identity - \ProjTwo_J)u_0}_{H^\gamma(\domain)}\to 0$ as $N\to\infty$, then
    \begin{equation*}
        \sqrt{N}(\bar{u}_N-u_0)\xrightarrow{d}\gaussian
    \end{equation*}
    under $\prob[u_0][N]{}$ as $N\to\infty$.
\end{lem}
\begin{proof}
    See Section \ref{sec:Proofs_BvM}.
\end{proof}
\begin{rem}\label{rmk:Bias_Control}
   The condition $\sqrt{N}\norm{(\identity - \ProjTwo_J)u_0}_{H^\gamma(\domain)}\to 0$ can be verified similarly to Lemma \ref{lem:Properties_LAN_Norm} \eqref{num:ProjectionJackson}, provided the $H^\gamma(\domain)\to H^\gamma(\domain)$-stability of $\ProjTwo_J$
\begin{equation}
\sup_{J\in\mathbb{N}}\norm{\ProjTwo_J}_{\mathcal{L}(H^\gamma(\domain),H^\gamma(\domain))}<\infty\label{eq:Hgammastability}
\end{equation}
holds. In this case, we obtain a parametric rate $N^{-1/2}$ for the estimation of the function $u_0$ in $H^\gamma(\domain)$-norm.
\end{rem}

The next theorem shows how a BvM theorem on $H^\gamma(\domain)$ transfers under Fréchet-differentiable transformations.

\begin{thm}\label{thm:nonlinear_wasserstein_delta_method}
    Grant Assumptions \ref{assump:ApproximatonSets}, \ref{assump:Operator_stability} and \ref{assump:operator_linearisation}. Let $m>d/2$, $0<\gamma<m-d/2$, consider the prior $\pi_N$ from \eqref{eq:Prior} with cut-off $2^J\sim N^{1/(2\beta+d)}$ and assume that $u_0\in \mathcal{V}\cap H^{m+\beta}(\domain)$ for $\beta>2d$, $d/2<\beta_0<\beta+d/4$. Let $(V,\norm{}_V)$ be a separable Banach space and $\Phi\colon (\mathbb{V}_\gamma,\norm{}_{H^\gamma(\domain)})\to (V,\norm{}_V)$ be Borel-measurable and continuously Fréchet differentiable on an open neighbourhood $\mathscr{U}\subset\mathbb{V}_\gamma$ of $u_0\in \mathbb{V}_\gamma$ with linear and bounded derivative $\mathcal{D}\Phi_{u_0}\colon (\mathbb{V}_\gamma,\norm{}_{H^\gamma(\domain)})\to (V,\norm{}_V)$.
    Assume, moreover, that $\Phi$ has polynomial growth: there exist constants $q\ge 1$ and $0<C<\infty$ such that
    \begin{equation*}
        \norm{\Phi(u)}_V\le C\left(1+\norm{u}_{H^\gamma(\domain)}^q\right),\quad u\in\mathbb{V}_\gamma.
    \end{equation*}    
    For $u\in\mathbb{V}_\gamma$ define the transformation
    %\begin{equation*}
    %    S_{N,1}(u)\coloneqq \sqrt{N}(\Phi(u)-\Phi(u_0) +\mathcal{D}\Phi_{u_0}(u_0-\bar{u}_N)),\quad S_{N,2}(u)\coloneqq \sqrt{N}(\Phi(u)-\Phi(\bar{u}_N)).
    %\end{equation*}
    \begin{equation*}
        S_{N}(u)\coloneqq \sqrt{N}(\Phi(u)-\Phi(\bar{u}_N)).
    \end{equation*}
     If
        \begin{equation*}
            \norm{\mathcal{D}\Phi_{u_0} - \mathcal{D}\Phi_{u}}_{\mathcal{L}(\mathbb{V}_\gamma,V)}\le C \norm{u-u_0}_{H^\gamma(\domain)},\quad u\in\mathscr{U},
        \end{equation*}
        for some constant $0<C<\infty$, then
        \begin{equation*}
        \mathcal{W}_{1,V}\left((S_{N})_{\#}\Pi_N(\MTemptyplaceholder|\data), (\mathcal{D}\Phi_{u_0})_{\#}\gaussian\right)\xrightarrow{\prob[u_0][N]{}}0
    \end{equation*}
    as $N\to\infty$.
\end{thm}
\begin{proof}
    See Section \ref{sec:Proofs_BvM}.
\end{proof}

\begin{rem}~
\begin{enumerate}[(i)]
    \item Note that $\Phi(\bar{u}_N)$ is a plug-in centring and differs in general from the posterior mean of $\Phi(u)$ when $\Phi$ is nonlinear.
    \item Theorem \ref{thm:nonlinear_wasserstein_delta_method} is a Wasserstein-1 version of the well-known functional delta method, see Chapter 20 of \cite{Vaart1998} or Theorem 3.10.4 of \cite{vandervaartWeakConvergenceEmpirical2023}.
    \item Typical examples for the transformation $\Phi$ in Theorem \ref{thm:nonlinear_wasserstein_delta_method} include bounded linear projections, Nemytskii operators or nonlinear integral functionals such as $u\mapsto \int_{\domain}g(u(x))w(x)\,\D x$.
\end{enumerate}
     
\end{rem}

Next we show that the BvM theorem can be extended to centring at the MAP instead of at the posterior mean. Since the MAP is obtained from a finite-dimensional optimisation problem and does not require posterior integration as the posterior mean, this yields computational advantages.

\begin{ass}[Local second-order regularity]\label{assump:operator_linearisation_2}
        The operator $\op$ is twice continuously Fréchet differentiable on the open neighbourhood $\mathcal{U}$ from Assumption \ref{assump:operator_linearisation}. Denoting its second derivative at $u\in\mathcal{U}$ by $\mathcal{D}^2\op_u$, there exists a constant $0<C<\infty$ such that for $u,v\in\mathcal{U}$ and $h_1,h_2\in\mathcal{V}$ we have
        \begin{align*}
            \norm{\mathcal{D}^2\op_u[h_1,h_2]}_{L^2(\domain)}&\le C\norm{h_1}_{H^m(\domain)}\norm{h_2}_{H^m(\domain)},\\
            \norm{(\mathcal{D}^2 \op_u - \mathcal{D}^2 \op_v)[h_1,h_2]}_{L^2(\domain)}&\le C\norm{u-v}_{H^m(\domain)}\norm{h_1}_{H^m(\domain)}\norm{h_2}_{H^m(\domain)}.
            \end{align*}
            If, additionally, $u,v\in \mathcal{U}\cap W^{m,\infty}(\domain)$ and $h_1,h_2 \in \mathcal{V}\cap W^{m,\infty}(\domain)$, then
            \begin{align*}
            \norm{\mathcal{D}^2\op_u[h_1,h_2]}_{L^\infty(\domain)}&\le C\norm{h_1}_{W^{m,\infty}(\domain)}\norm{h_2}_{W^{m,\infty}(\domain)},\\
            \norm{(\mathcal{D}^2 \op_u - \mathcal{D}^2 \op_v)[h_1,h_2]}_{L^\infty(\domain)}&\le C\norm{u-v}_{W^{m,\infty}(\domain)}\norm{h_1}_{W^{m,\infty}(\domain)}\norm{h_2}_{W^{m,\infty}(\domain)}.
        \end{align*}
\end{ass}
Assumption \ref{assump:operator_linearisation_2} is implied if $\op$ is twice locally Fréchet differentiable around $u_0$ with Lipschitz-continuous second Fréchet derivative.

\begin{thm}[BvM centred at the MAP]\label{thm:BvM_MAP}
Grant Assumptions \ref{assump:ApproximatonSets}, \ref{assump:Operator_stability}, \ref{assump:operator_linearisation} and \ref{assump:operator_linearisation_2}.
Let $m>d/2$, $0<\gamma<m-d/2$, consider the prior $\pi_N$ from \eqref{eq:Prior} with cut-off $2^J\sim N^{1/(2\beta+d)}$ and assume that $u_0\in \mathcal{V}\cap H^{m+\beta}(\domain)$ for $\beta>2d$, $d/2<\beta_0<\beta+d/4$. Define the map $\tau_N(u)\coloneqq \sqrt{N}(u-\hat{u}_N^{\operatorname{MAP}})$ for $u\in V_J$, then
    \begin{equation*}
        \mathcal{W}_{1,H^{\gamma}(\domain)}\left((\tau_N)_{\#}\Pi_N(\MTemptyplaceholder|\data), \gaussian\right) \xrightarrow{\prob[u_0][N]{}}0, \quad N\to\infty.
    \end{equation*}
\end{thm}
\begin{proof}
    See Section \ref{sec:Proofs_BvM}.
\end{proof}

\begin{rem}~
\begin{enumerate}[(i)]
    \item The posterior mean coincides with the MAP estimator in conjugate Gaussian models. This identity generally fails in the present nonlinear setting. Our proof expands the regularised score between $\ProjOne_J u_0$ and $\hat{u}_N^{\operatorname{MAP}}$ and controls the empirical Hessian using the second-order regularity provided by Assumption \ref{assump:operator_linearisation_2}. Consequently, we establish that the MAP estimator is an efficient estimator and the BvM centred at the MAP follows.
    \item For Theorem \ref{thm:BvM_MAP} alone, Assumption \ref{assump:operator_linearisation_2} could potentially be weakened to first-order regularity of $\op$ by treating the MAP as an Z-estimator and building upon established theory (Theorem 19.26 of \cite{Vaart1998}, Theorem 3.3.1 of \cite{vandervaartWeakConvergenceEmpirical2023}). We retain the stronger Assumption \ref{assump:operator_linearisation_2} to streamline the proofs since controlling the empirical score Hessian is required for establishing the convergence of the Laplace approximation in Theorem \ref{thm:LaplaceApproximation} (below).
    \item Related results in \cite{panovFiniteSampleBernstein2015,katsevichImprovedDimensionDependence2025} concern Gaussian approximations in semiparametric or growing-dimensional Euclidean models: \cite{panovFiniteSampleBernstein2015} principally treats a low-dimensional target marginal, while \cite{katsevichImprovedDimensionDependence2025} derives a mode-centred Laplace approximation with explicit finite-sample bounds. By contrast, Theorem \ref{thm:BvM_MAP} concerns the full posterior as a probability measure on a fixed Sobolev space $H^\gamma(\domain)$, while $\operatorname{dim}(V_J)\to\infty$.
    \end{enumerate}
\end{rem}

\begin{lem}[Convergence of the MAP]\label{lem:Convergence_MAP_Asymptotic_Normality}
    Grant Assumptions \ref{assump:ApproximatonSets}, \ref{assump:Operator_stability}, \ref{assump:operator_linearisation} and \ref{assump:operator_linearisation_2}. Let $m>d/2$, $0<\gamma<m-d/2$, consider the prior $\pi_N$ from \eqref{eq:Prior} with cut-off $2^J\sim N^{1/(2\beta+d)}$ and assume that $u_0\in \mathcal{V}\cap H^{m+\beta}(\domain)$ for $\beta>2d$, $d/2<\beta_0<\beta+d/4$.
   Then
    \begin{equation*}
        \sqrt{N}(\hat{u}_N^{\operatorname{MAP}}-\ProjTwo_J u_0)\xrightarrow{d}\gaussian
    \end{equation*}
    under $\prob[u_0][N]{}$ as $N\to\infty$.  If, additionally, $\sqrt{N}\norm{(\identity - \ProjTwo_J)u_0}_{H^\gamma(\domain)}\to 0$ as $N\to\infty$, then
    \begin{equation*}
        \sqrt{N}(\hat{u}_N^{\operatorname{MAP}}-u_0)\xrightarrow{d}\gaussian
    \end{equation*}
    under $\prob[u_0][N]{}$ as $N\to\infty$.
\end{lem}
\begin{proof}
    See Section \ref{sec:Proofs_BvM}.
\end{proof}
As discussed in Remark \ref{rmk:Bias_Control}, the projection of $u_0$ in Lemma \ref{lem:Convergence_MAP_Asymptotic_Normality} can thus be avoided if the projections $\ProjTwo_J$ are $H^\gamma(\domain)$-stable uniformly in $J\in\mathbb{N}$.

Finally, we also obtain the following convergence result for the Laplace approximation. Let $\mathcal{J}_N(v)\coloneqq -\mathcal{D}^2\ell_N^{\operatorname{reg}}(v)/N$ be the scaled negative Hessian of the log-posterior at $v\in V_J$. On the event where $\mathcal{J}_N(\hat{u}_N^{\operatorname{MAP}})$ is invertible, define the Laplace approximation
\begin{equation*}
    \hat{\Pi}_N^{\operatorname{Lap}}(\MTemptyplaceholder|\data)\coloneqq N\Big(\hat{u}_N^{\operatorname{MAP}}, \frac{1}{N}\mathcal{J}_N(\hat{u}_N^{\operatorname{MAP}})^{-1}\Big)
\end{equation*}
as a Gaussian measure on $V_J$, conditionally on the data $\data$.
\begin{thm}[Laplace approximation]\label{thm:LaplaceApproximation}
    Grant Assumptions \ref{assump:ApproximatonSets}, \ref{assump:Operator_stability}, \ref{assump:operator_linearisation} and \ref{assump:operator_linearisation_2}. Let $m>d/2$, $0<\gamma<m-d/2$, consider the prior $\pi_N$ from \eqref{eq:Prior} with cut-off $2^J\sim N^{1/(2\beta+d)}$ and assume that $u_0\in \mathcal{V}\cap H^{m+\beta}(\domain)$ for $\beta>2d$, $d/2<\beta_0<\beta+d/4$.
   Then $\prob[u_0][N]{\mathcal{J}_N(\hat{u}_N^{\operatorname{MAP}})\text{ is invertible}}\to 1$ and
    \begin{equation*}
        \sqrt{N}\mathcal{W}_{1,H^{\gamma}(\domain)}\left(\Pi_N(\MTemptyplaceholder|\data), \hat{\Pi}_N^{\operatorname{Lap}}(\MTemptyplaceholder|\data)\right) \xrightarrow{\prob[u_0][N]{}}0, \quad N\to\infty.
    \end{equation*}
    as $N\to\infty$.
\end{thm}
\begin{proof}
    See Section \ref{sec:Proofs_BvM}.
\end{proof}

\begin{comment}
\begin{rem}
    Let $\mu_N$ be the posterior measure and $r_N=\sqrt{N}$, then the condition \eqref{eq:condition_remainder_1} can be checked if the Fréchet remainder can be controlled quantitatively: If locally around $\theta_0$ we have
    \begin{equation*}
        \norm{\Phi(\theta)-\Phi(\theta_0)- \mathcal{D}\Phi_{\theta_0}(\theta-\theta_0)}_V\le \norm{\theta-\theta_0}_V^{1+\gamma}
    \end{equation*}
    for some $\gamma>0$, then \eqref{eq:condition_remainder_1} reduces to
    \begin{equation*}
         \EV[\Pi_N]{\sqrt{N}\norm{\theta-\theta_0}_U^{1+\gamma}|\data }\xrightarrow{\prob[\theta_0]{}}0,
    \end{equation*}
    which can be verified using the ideas from the proof of Lemma \ref{lem:Localisationisfine}.
    If $\Phi$ is Fréchet differentiable in an open neighbourhood of $\theta_0$ and $\theta\mapsto \mathcal{D}\Phi_{\theta_0}$ is Lipschitz-continuous, then the condition \eqref{eq:condition_remainder_2} simplifies to
    \begin{equation*}
        \EV[\Pi_N]{\sqrt{N}\norm{\theta-a_N}_U(\norm{\theta_0-\theta}_U + \norm{\theta_0-a_N}_U)\given\data }\xrightarrow{\prob[\theta_0]{}}0,
    \end{equation*}
    which can be verified using ideas from the proofs of Lemmas \ref{lem:Convergence_PosteriorMean}, \ref{lem:Localisationisfine}, if higher posterior moments of $\norm{\theta-a_N}_U$ can be controlled.\SG{rewrite remark?}
\end{rem}
\end{comment}

\section{Examples}\label{sec:examples}

We illustrate the previous assumptions and results for four classes of PDEs. The first two examples are elliptic equations of second and fourth order, respectively. The third example is a one-dimensional first-order transport-reaction equation, while the final example contains nonlinear gradient dependence. Except for the one-dimensional transport example, let $\domain\subset \mathbb{R}^d$ be a smooth bounded domain, and let the domain $\mathcal{V}\subset H^m(\domain)$ be chosen according to homogeneous boundary conditions. 
The construction of boundary-adapted multiscale decompositions satisfying Assumption \ref{assump:ApproximatonSets} is well-studied in this setting, see e.g. \cite{monasseOrthonormalWaveletBases1998,cohen2003numerical,cohenMultiscaleDecompositionsBounded2000}.

The examples further illustrate how the order $m$ the affects the limiting Gaussian measure and thus the asymptotic posterior covariance. Since the limiting process is tight in $H^\gamma(\domain)$ for $\gamma<m-d/2$, higher-order differential operators allow for BvM limits in stronger topologies. The fourth-order elliptic example allows uniform-norm conclusions in dimensions $d\le 3$, the second-order elliptic examples do so in dimension $d=1$, whereas the first-order transport equation in dimension one yields $H^\gamma(\domain)$-limits only for $\gamma<1/2$.

\subsection{Semilinear elliptic equation}\label{sec:semilienarElliptic}

Consider the operator 
\begin{equation*}
    \op[u]=-\Delta u + \lambda u +\tau(u),\quad\lambda>0,
\end{equation*}
where $\tau\in C^\infty(\mathbb{R})$ satisfies for some $K>0$ the condition $0\le \tau'(z)\le K$ for all $z\in\mathbb{R}$ and its higher derivatives are bounded. We take $m=2$. For homogeneous Dirichlet boundary conditions we consider $\mathcal{V}=H^2(\domain)\cap H_0^1(\domain)$, whereas for homogeneous Neumann boundary conditions we take $\mathcal{V} = \{u\in H^2(\domain)\colon \partial_\nu u=0\text{ on }\partial\domain\}$, where $\partial_\nu$ denotes the outward normal derivative. Periodic boundary conditions are covered by taking $\domain=\mathbb{T}^d=\mathbb{R}^d/\mathbb{Z}^d$, the $d$-dimensional torus and $\mathcal{V}= H^2(\mathbb{T}^d)$.
We refer to \cite{GT98} for standard elliptic well-posedness and regularity properties, and to \cite{runstSobolevSpacesFractional1996} for the required properties of Nemytskii operators on Sobolev spaces. 

First note that $u\mapsto \tau(u)\colon H^{2+\kappa_0}(\domain)\to H^{\kappa_0}(\domain)$ is locally bounded and continuous for any $\kappa_0>d/2$. Next we verify the global stability Assumption \ref{assump:Operator_stability}. For $u,v\in\mathcal{V}$ set $w=u-v$ and
\begin{equation*}
    q_{u,v}(x) = \lambda + \int_0^1\tau'(v(x)+tw(x))\,\D t\in[\lambda,\lambda+K],\quad x\in\domain.
\end{equation*}
Then $\op[u]-\op[v]=-\Delta w + q_{u,v}w$, coercivity of $-\Delta + q_{u,v}$ and elliptic $H^2(\domain)$-regularity imply the global stability \eqref{eq:Stability_Lipschitz_Bounds}. The Lipschitz bound \eqref{eq:Lipschitz_Bounds_infty} follows from the global Lipschitz continuity of $\tau$.
Consequently, the nonparametric contraction and convergence results of both the posterior mean and the MAP of Section \ref{sec:Contraction} apply with $d\ge 1$ and $m=2$.

We proceed with the stronger regularity conditions of Assumptions \ref{assump:operator_linearisation} and \ref{assump:operator_linearisation_2} required for the BvM results for $d\le 3$. The first two Fréchet derivatives are given by
\begin{equation*}
    \opDeriv{u}[h] = -\Delta h+ \lambda h+\tau'(u)h,\quad \mathcal{D}^2\op_u [h_1,h_2] = \tau''(u)h_1h_2,\quad u,h,h_1,h_2\in \mathcal{V}.
\end{equation*}
Since $\lambda+ \tau'(u_0)\ge\lambda$, the previous coercivity and elliptic-regularity argument applies and yields the gradient stability \eqref{eq:Graphnormequivalence}. The remaining conditions follow from Taylor's theorem and Sobolev embeddings with $d\le 3$. Consequently, the BvM results of Section \ref{sec:BvM} apply with $d\le 3$ and $0<\gamma<2-d/2$. In particular, when $d=1$, one may choose $1/2<\gamma<3/2$, and the Sobolev embedding shows that the posterior fluctuations around their centre are of parametric order $N^{-1/2}$ in $L^\infty(\domain)$. If the projection bias condition in Remark \ref{rmk:Bias_Control} holds, then we obtain a parametric convergence rate of both the posterior mean and the MAP in $\norm{}_{L^\infty(\domain)}$.

\section*{Numerical illustration}

To illustrate our results we perform a numerical simulation with the forward operator
\begin{equation*}
    \op[u]=0.04(- 0.1 \Delta u + \tau(u)),\quad \tau(z) = \frac{z}{10} + \tanh(5(z-0.35))+\tanh(1.75),\quad z\in\mathbb{R},
\end{equation*}
on the two-dimensional torus $\mathbb{T}^2$ with $X_i\overset{i.i.d.}{\sim}\operatorname{Unif(\mathbb{T}^2)}$. The true solution is the superposition of functions as shown in Figure \ref{fig:Reconstruction}
and the measurement noise level is $\sigma = 0.03$.
Let $\mathbb{Z}_+^2 = \{\mu=(\mu_1,\mu_2)\in\mathbb{Z}^2\colon\mu_1>0\}\cup\{ (0,\mu_2)\colon\mu_2>0\}$ and define the $L^2(\mathbb{T}^2)$-orthonormal Fourier basis by 
\begin{equation*}
    \psi_\mu = \begin{cases}
    1,&\mu = 0,\\
    \sqrt{2}\cos(2\pi\iprod{\mu}{x}_{\mathbb{R}^2}),&\mu\in\mathbb{Z}_+^2,\\
    \sqrt{2}\sin(2\pi\iprod{-\mu}{x}_{\mathbb{R}^2}),&-\mu\in\mathbb{Z}_+^2,
    \end{cases}
\end{equation*}
for $\mu\in\mathbb{Z}^2$. We let $\abs{\mu} = 0$ if $\max(\abs{\mu_1},\abs{\mu_2})\le 1$ and $\abs{\mu}=\lceil \log_2(\max(\abs{\mu_1},\abs{\mu_2}))\rceil$ for all other $\mu\in\mathbb{Z}^2$. We take $\beta=4.5$, $\beta_0=2$ and scale the prior by a factor of 5.

In Figure \ref{fig:Reconstruction} we compare the recovery of $u_0$ with the recovery of the forward field $\op[u_0]$ when using the MAP $\hat{u}_N^{\operatorname{MAP}}$ for $N=250$ sample points. We can observe that the discrepancy between reconstruction and truth of $u_0$ is visibly better than the reconstruction of the forward field. This illustrates the effect of the stability Assumption \ref{assump:Operator_stability}: The prediction error for $\op[u]$ in $\norm{}_{L^2(\domain)}$ controls the error for $u$ in the stronger norm $\norm{}_{H^2(\domain)}$.

In Figure \ref{fig:Diagnostics} (left) we analyse the local shape of the posterior distribution. We use a Metropolis-adjusted Langevin algorihm (MALA) to obtain samples from the posterior distribution. After centring at the MAP, rescaling by $\sqrt{N}$ and the LAN standard deviation from Theorem \ref{thm:BvM_MAP}, the empirical marginals approach the standard Gaussian density as the sample size increases from $N=150$ to $N=1000$, as predicted by Theorem \ref{thm:BvM_MAP}. Note that the MAP centring is obtained by optimisation only and the MALA posterior simulation is needed solely to verify the Gaussian approximation.

Figure \ref{fig:Diagnostics} (right) displays the mean squared error of the MAP in progressively stronger Sobolev norms using a Monte-Carlo estimate with 100 samples, namely $\norm{}_{H^1(\domain)}$, $\norm{}_{H^2(\domain)}$ and $\norm{}_{H^3(\domain)}$. The clear deterioration of the empirical slope when strengthening the norm aligns well with the results of Lemmas \ref{lem:Convergence_MAP_Nonparametric} and \ref{lem:Convergence_MAP_Asymptotic_Normality} and reflects the difficulty of recovering high-frequency components of the signal. The jumps at sample size $N=1000$ occur because the increased cut-off admits an additional resolution level.

\begin{figure}
    \centering
    \includegraphics[width=\linewidth]{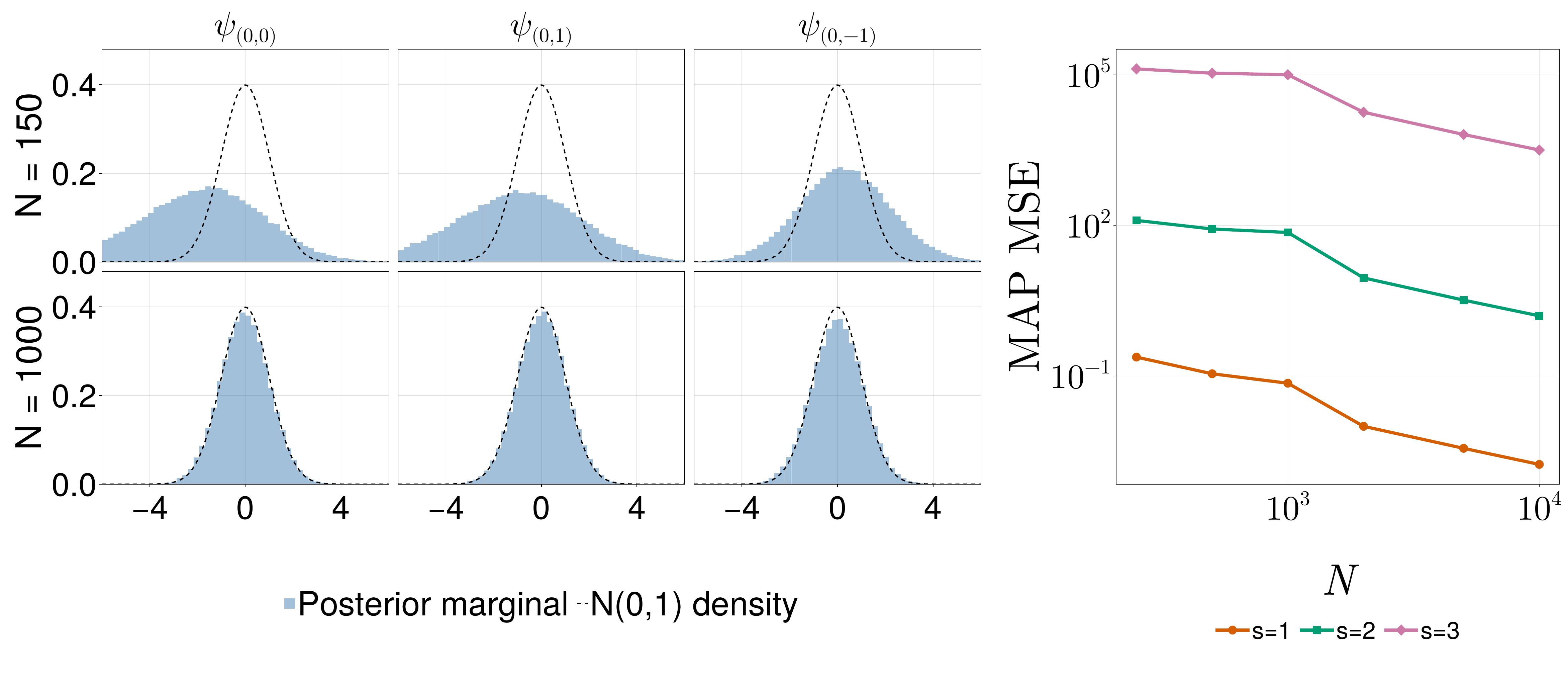}
    \caption{Left: Histograms of the centred and standardised Fourier coefficients $\sqrt{N}\iprod{u-\hat{u}_N^{\operatorname{MAP}}}{\psi_\mu}_{L^2(\domain)} \allowbreak/\norm{(\psi_\mu)_0}_{\operatorname{LAN}}$, $\mu=(0,0), (0,1), (0,-1)$ for $N=150$ and $N=1000$, compared with limit $N(0,1)$-density from Theorem \ref{thm:BvM_MAP} for $N=150$ and $N=1000$ observations. Right: Monte-Carlo estimates of the MAP MSE $\EV[u_0]{\norm{\hat{u}_N^{\operatorname{MAP}}-u_0}_{H^s(\domain)}^2}$ for $s=1,2,3$ for $100$ Monte-Carlo samples. Log-log scale.}
    \label{fig:Diagnostics}
\end{figure}

\subsection{A fourth-order elliptic equation}

Consider the linear elliptic operator of fourth order given by
\begin{equation*}
    \op[u] =\Delta^2 u + \lambda u,\quad\lambda>0.
\end{equation*}
We take $m=4$ and impose homogeneous Navier boundary conditions, such that
\begin{equation*}
    \mathcal{V}=\{u\in H^4(\domain)\colon u=0,\Delta u=0\text{ on }\partial\domain\}.
\end{equation*} 
For well-posedness and regularity theory of the corresponding boundary value problem $\op[u]=f$ we refer to Chapter 2 of \cite{gazzolaPolyharmonicBoundaryValue2010}. 

First note that for every $\kappa_0>d/2$ the operator $\op\colon\mathcal{V}\cap H^{4+\kappa_0}(\domain)\to H^{\kappa_0}(\domain)$ is continuous.
Fourth-order elliptic regularity (cf.~Theorem 2.20 of \cite{gazzolaPolyharmonicBoundaryValue2010}) and $\lambda>0$ imply the global stability \eqref{eq:Stability_Lipschitz_Bounds}. The Lipschitz bound \eqref{eq:Lipschitz_Bounds_infty} also follows. Consequently, the nonparametric contraction and convergence results of both the posterior mean and thee MAP of Section \ref{sec:Contraction} apply with $d\ge 1$ and $m=4$.

Since $\op$ is linear, verifying the linearisation properties of Assumptions \ref{assump:operator_linearisation} and \ref{assump:operator_linearisation_2} is particularly simple with
\begin{equation*}
    \opDeriv{u}[h] = \Delta^2 h +\lambda h \text{ and }\mathcal{D}^2\op_u =0,\quad u,h\in \mathcal{V}.
\end{equation*}
The gradient stability \eqref{eq:Graphnormequivalence} is implied by the fourth-order elliptic regularity, while the other conditions follow directly.
Consequently, the BvM results of Section \ref{sec:BvM} apply with $d\le 7$ and $0<\gamma<4-d/2$. In particular, we obtain parametric contraction and convergence rates in $L^\infty(\domain)$ for $d\le 3$, provided the bias condition of Remark \ref{rmk:Bias_Control} is satisfied.

\begin{comment}
Since $m=4$, we can take $d\le 7$ and $0<\gamma<4-d/2$ and obtain the Bernstein--von Mises results for Navier boundary condition in
\begin{equation*}
    \mathbb{V}_\gamma
    =
    \begin{cases}
        H^\gamma(\domain), & 0<\gamma<1/2,\\
        \{u\in H^\gamma(\domain)\colon 
          u|_{\partial\domain}=0\},
            & 1/2<\gamma<5/2,\\
        \{u\in H^\gamma(\domain)\colon 
          u|_{\partial\domain}=0,\
          (\Delta u)|_{\partial\domain}=0\},
            & 5/2<\gamma<4,
    \end{cases}
\end{equation*}
see Chapter 1 of \cite{LM72}.
In particular, we obtain parametric contraction and convergence rates in $L^\infty(\domain)$ for $d\le 3$, provided the bias condition of Remark \ref{num:Bias_Control} is satisfied.
\end{comment}

\subsection{A stationary transport-reaction equation}

Consider $\domain=(0,1)$ and the first-order forward operator
\begin{equation*}
    \op[u] = u'+\lambda u + \tau(u),\quad\lambda>0
\end{equation*}
with the boundary condition $u(0)=0$. Assume that $\tau\in  C^\infty(\mathbb{R})$ satisfies for some $K>0$ the condition $0\le\tau'(z)\le K$ for all $z\in\mathbb{R}$ and has bounded higher derivatives. Here, $d=1$, $m=1$ and $\mathcal{V} = \{u\in H^1(\domain)\colon u(0)=0\}$.

To show the global stability condition \eqref{eq:Stability_Lipschitz_Bounds} of Assumption \ref{assump:Operator_stability} we set $w=u-v$ for $u,v\in\mathcal{V}$ and
\begin{equation*}
    q_{u,v}\coloneqq \lambda + \int_0^1\tau'(v+tw)\,\D t\in [\lambda,\lambda+K]
\end{equation*}
and note that $\op[u]-\op[v] = w'+q_{u,v}w$. Since $w(0)=0$ by the boundary conditions, we find
\begin{equation*}
    \iprod{\op[u]-\op[v]}{w}_{L^2(\domain)} = \frac{1}{2}w(1)^2 + \int_0^1 q_{u,v}(x)w(x)^2\,\D x\ge\lambda\norm{w}_{L^2(\domain)}^2.
\end{equation*}
Hence $\norm{w}_{L^2(\domain)}\lesssim \norm{\op[u]-\op[v]}_{L^2(\domain)}$ and $w'=\op[u]-\op[v]-q_{u,v}w$ gives 
\begin{equation*}
    \norm{u-v}_{H^1(\domain)}\lesssim \norm{\op[u]-\op[v]}_{L^2(\domain)}.
\end{equation*}
Together with the global Lipschitz continuity of $\tau$, this proves the global stability and Lipschitz bounds \eqref{eq:Stability_Lipschitz_Bounds} and \eqref{eq:Lipschitz_Bounds_infty}. Consequently, the nonparametric contraction and convergence results of both the posterior mean and the MAP of Section \ref{sec:Contraction} apply with $m=1$.

The Fréchet derivatives of $\op$ are given by
\begin{equation*}
    \opDeriv{u}[h] = h'+(\lambda+\tau'(u))h,\quad \mathcal{D}^2\op_u[h_1,h_2]=\tau''(u)h_1h_2,\quad u,h,h_1,h_2\in\mathcal{V}.
\end{equation*}
For the gradient stability \eqref{eq:Graphnormequivalence} we define $q_0=\lambda+\tau'(u_0)\in[\lambda,\lambda+K]$ and observe
\begin{equation*}
    \iprod{\opDeriv{u_0}[h]}{h}_{L^2(\domain)} = \frac{1}{2}h(1)^2 + \int_0^1 q_0(x)h(x)^2\,\D x \ge\lambda\norm{h}_{L^2(\domain)}^2.
\end{equation*}
We deduce that $\norm{h}_{L^2(\domain)}\lesssim \norm{\opDeriv{u_0}[h]}_{L^2(\domain)}$ and, combined with $h'=\opDeriv{u_0}[h]-q_0h$, we deduce that gradient stability and Lipschitz conditions \eqref{eq:Graphnormequivalence} and \eqref{eq:Lipschitz_Bounds_infty} hold.
Taylor's theorem implies the remaining differentiability assumptions. Consequently, the BvM results of Section \ref{sec:BvM} apply with $d=1$ and $0<\gamma<1/2$.

\subsection{A gradient-nonlinear elliptic equation}

Consider the operator
\begin{equation*}
    \op[u]=-\Delta u + \lambda u + \rho(\nabla u),\quad\lambda >0,
\end{equation*}
where $\rho\in C_b^\infty(\mathbb{R}^d,\mathbb{R})$. We take $m=2$ and either Dirichlet or Neumann boundary conditions as in Section \ref{sec:semilienarElliptic}. Denote by $\mathcal{A}=-\Delta + \lambda\identity$ the linear part of $\op$ and let $c>0$ satisfy $\norm{\mathcal{A}u}_{L^2(\domain)}\ge c\norm{u}_{H^2(\domain)}$ for any $u\in\mathcal{V}$. Assume that
\begin{equation*}
    \norm{\nabla\rho}_{L^\infty(\mathbb{R}^d)}\sup_{u\in\mathcal{V},u\neq 0}\frac{\norm{\nabla u}_{L^2(\domain)}}{\norm{u}_{H^2(\domain)}}<c,
\end{equation*}
which ensures that the global stability of $\op$ is inherited from $\mathcal{A}$.

First note that $u\mapsto \rho(\nabla u)\colon H^{2+\kappa_0}(\domain)\to H^{\kappa_0}(\domain)$ is locally bounded and continuous. Furthermore, we find
\begin{align*}
    \norm{\op[u]-\op[v]}_{L^2(\domain)}&\ge \norm{\mathcal{A}(u-v)}_{L^2(\domain)} - \norm{\rho(\nabla u) - \rho(\nabla v)}_{L^2(\domain)}\\
    &\ge \Big(c - \norm{\nabla\rho}_{L^\infty(\mathbb{R}^d)}\sup_{u\in\mathcal{V},u\neq 0}\frac{\norm{\nabla u}_{L^2(\domain)}}{\norm{u}_{H^2(\domain)}}\Big)\norm{u-v}_{H^2(\domain)}.
\end{align*}
Since the reverse inequality is immediate, the global stability and Lipschitz conditions \eqref{eq:Stability_Lipschitz_Bounds} and \eqref{eq:Lipschitz_Bounds_infty} hold. Consequently, the nonparametric contraction and convergence results of both the posterior mean and the MAP of Section \ref{sec:Contraction} apply with $d\ge 1$ and $m=2$.

We proceed with verifying Assumptions \ref{assump:operator_linearisation} and \ref{assump:operator_linearisation_2} for the relevant range $d\le 3$. The derivatives are given by
\begin{equation*}
    \opDeriv{u}[h] = \mathcal{A}h + \nabla\rho(\nabla u)\cdot\nabla h\text{ and }\mathcal{D}^2\op_u[h_1,h_2]=\mathcal{D}^2 \rho(\nabla u)[\nabla h_1,\nabla h_2],\quad u,h,h_1,h_2\in\mathcal{V}.
\end{equation*}
The gradient stability \eqref{eq:Graphnormequivalence} follows from our assumption on $\rho$ since
\begin{align*}
    \norm{\opDeriv{u_0}[h]}_{L^2(\domain)}&\ge \norm{\mathcal{A}h}_{L^2(\domain)} - \norm{\nabla\rho(\nabla u_0)\cdot\nabla h}_{L^2(\domain)}\\
    &\ge \Big(c - \norm{\nabla\rho}_{L^\infty(\mathbb{R}^d)}\sup_{u\in\mathcal{V},u\neq 0}\frac{\norm{\nabla u}_{L^2(\domain)}}{\norm{u}_{H^2(\domain)}}\Big)\norm{h}_{H^2(\domain)}
\end{align*}
for any $h\in\mathcal{V}$. For $d\le 3$ we find by Taylor's theorem and Sobolev embedding that the linearisation conditions of Assumptions \ref{assump:operator_linearisation} and \ref{assump:operator_linearisation_2} hold. Consequently, the BvM results of Section \ref{sec:BvM} apply with $d\le 3$ and $0<\gamma<2-d/2$. In particular, we obtain parametric contraction and convergence rates in $L^\infty(\domain)$ for $d=1$, provided the bias condition of Remark \ref{rmk:Bias_Control} is satisfied.

\section*{Acknowledgements and AI disclosure}
SG and SW acknowledge funding by the Deutsche Forschungsgemeinschaft (DFG, German Research Foundation) – CRC/TRR 388 "Rough Analysis, Stochastic Dynamics and Related Fields" – Project ID 516748464, projects B07 and B08. We thank Andrew Stuart and Houman Owhadi for fruitful early discussions on posterior contraction for probabilistic PDE solution methods that helped initiate this line of research. We thank Markus Reiß for fruitful discussions. ChatGPT 5.6 was used for proofreading and for the simulations of Section \ref{sec:semilienarElliptic}. The authors independently verified all final mathematical arguments, simulation code, and reported results and take full responsibility for the manuscript.

\bibliography{tvbib}
\bibliographystyle{abbrv}

\newpage
\appendix

\section{Proof for Section \ref{sec:Contraction}}\label{sec:Proofs_Contraction}

\begin{lem}\label{lem:Prior}
	For 
    %$0\leq\beta_0<\beta$
    $d/2<\beta_0<\beta+d/2$
    and $u_0\in \mathcal{V}\cap H^{m+\beta}(\domain)$ consider the Gaussian prior $\pi_N$ from \eqref{eq:Prior} on $V_J$  with cut-off $J=J_N\in\mathbb{N}$ such that $2^J\sim N^{1/(2\beta+d)}$. Let 
    %$\epsilon_N = N^{-\beta/(2\beta+1)}\log(N)$
    $\epsilon_N = N^{-\beta/(2\beta+d)}\log(N)$. Then for any $C_0>1$ and $L',L'''>0$ there exist constants $L,L'',C_1,C_2>0$ satisfying the following properties:
	\begin{enumerate}[(a)]
		\item\label{num:aux_rescaled_Prior_sieve} Sieve set: $\pi_N(u\in V_J:\norm{u}_{H^{m+\beta_0}(\domain)}>L\sqrt{N}\epsilon_N)\le \exp(-C_0N\epsilon_N^2)$.
		\item\label{num:aux_rescaled_Prior_entropy} Entropy bound: \[\log\left(N\left(V_J\cap B(0,\norm{}_{H^{m+\beta_0}(\domain)},L\sqrt{N}\epsilon_N), \norm{}_{H^m(\domain)}, \epsilon_N\right)\right)\le C_1N\epsilon_N^2.\]
        \item\label{num:aux_rescaled_Prior_smallball} Small ball probability: $\pi_N (u\in V_J:\norm{u-u_0}_{H^m(\domain)}\leq L'\epsilon_N)\ge \exp(-C_2 N\epsilon_N^2)$.
		\item\label{num:aux_rescaled_Prior_restrict}
        Posterior regularity: For any $0\le \kappa\le \beta$ and $u\in V_J$ with $\norm{u-u_0}_{H^m(\domain)}\le L'''\epsilon_N$ we have
        \begin{equation*}
            \norm{u}_{H^{m+\kappa}(\domain)} \le L'' \begin{cases}
                1,&\kappa<\beta,\\
                \log(N),&\kappa=\beta.
            \end{cases}
        \end{equation*}
        %\begin{equation*}
        %    \pi_N(u\in V_J: \norm{u}_{H^{m+\kappa}(\domain)}> C_2,\norm{u-u_0}_{H^m(\domain)}\leq C_3\varepsilon_{N}\,\vert\, X^N)\xrightarrow{\prob[u_0][N]{}} 0,\qquad N\to\infty.
        %\end{equation*}
	\end{enumerate}
\end{lem}
\begin{proof}
\begin{enumerate}[(a)]
\item We have $\operatorname{dim}(V_J)\lesssim 2^{Jd}\lesssim N^{d/(2\beta+d)}\lesssim N\epsilon_N^2$. Thus, for $L$ large enough it holds that $\operatorname{dim}(V_J)\leq (L^2/2)N\epsilon_N^2$. For this $L$, we find with the random variables $Z_{\mu}\overset{i.i.d.}{\sim} N(0,1)$, $\abs{\mu}\leq J$, from \eqref{eq:Prior} that
	\begin{align*}
		\pi_N(u\in V_J:\norm{u}^2_{H^{m+\beta_0}(\domain)}> L^2N\epsilon_N^2)
			&\le \prob[][][]{\sum_{\abs{\mu}\leq J}Z^2_{\mu}>cL^2N\epsilon_N^2}\\
			&\leq\prob[][][]{\sum_{\abs{\mu}\leq J}(Z^2_{\mu}-1)>c(L^2/2)N\epsilon_N^2},
	\end{align*}
    for some $0<c<\infty$,
	where the last inequality holds since there are at most $\operatorname{dim}(V_J)$ many $Z_{\mu}$ in the sum. The random variables $Z_{\mu}^2-1$ are independent centred chi-squared random variables. Absorbing $c$ into $L$ and applying a standard subexponential inequality (e.g., Equation 3.29 of \cite{GN16}) implies for the last probability the upper bound 
	\begin{align*}
		\exp\bigg(-\frac{(L^2/2)^2(N\epsilon_N^2)^2}{4\operatorname{dim}(V_J)+(L^2/2)N\epsilon_N^2}\bigg) \leq \exp(-(L^2/10)N\epsilon_N^2).
	\end{align*}
	\item For some constant $0<C<\infty$ we have 
	\begin{align*}
		&\log\left(N\left(V_J\cap B(0,\norm{}_{H^{m+\beta_0}(\domain)},L\sqrt{N}\epsilon_N), \norm{}_{H^m(\domain)}, \epsilon_N\right)\right)\\
        &\quad \le \log\left(N\left(V_J\cap B(0,\norm{}_{H^{m}(\domain)},cL\sqrt{N}\epsilon_N), \norm{}_{H^m(\domain)}, \epsilon_N\right)\right)\\
		&\quad  \leq \log\left(N\left(B(0,\norm{}_{\mathbb{R}^{\operatorname{dim}(V_J)}},cL\sqrt{N}\epsilon_N), \norm{}_{\mathbb{R}^{\operatorname{dim}(V_J)}}, \epsilon_N\right)\right)\\
		&\quad \leq \operatorname{dim}(V_J)\log(3cL\sqrt{N})\lesssim N^{d/(2\beta+d)}\log(N)\lesssim N\epsilon_N^2,
	\end{align*} 
	using metric entropy estimates for finite-dimensional Euclidean balls (e.g.~Proposition 4.3.34 of \cite{GN16}).
    \item For $u_0\in H^{m+\beta}(\domain)$ we have $\norm{u_0-\ProjOne_J u_0}_{H^m(\domain)}\lesssim 2^{-J\beta}\norm{u_0}_{H^{m+\beta}(\domain)}\lesssim \epsilon_N$, so for $C=L'/2$
	\begin{align*}
		\pi_N (u\in V_J:\norm{u-u_0}_{H^m(\domain)}\leq L'\epsilon_N) &\geq \pi_N (u\in V_J:\norm{u-\ProjOne_Ju_0}_{H^m(\domain)}\leq C\epsilon_N)\\
		&\geq \pi_N (u\in V_J:\norm{u-\ProjOne_J u_0}_{H^{m+\beta_0}(\domain)}\leq C\epsilon_N)
    \end{align*}
    Using that $\ProjOne_J u_0\in V_J$, \eqref{eq:RKHS_Norm}, \eqref{eq:Bernstein} and $\beta_0<\beta+d/2$ we find 
    \begin{equation*}
        \norm{\ProjOne_J u_0}_{\mathbb{H}_N}\lesssim \norm{\ProjOne_Ju_0}_{H^{m+\beta_0+\epsilon}(\domain)} \lesssim 2^{J(\beta_0+\epsilon-\beta)_+}\norm{u_0}_{H^{m+\beta}(\domain)}\lesssim \sqrt{N}\epsilon_N,
    \end{equation*}
    with $0<\epsilon<\beta+d/2-\beta_0$. Consequently, we can apply Corollary 2.6.18 of \cite{GN16} to find 
    \begin{equation*}
        \pi_N (u\in V_J:\norm{u-u_0}_{H^m(\domain)}\leq L'\epsilon_N)\ge e^{-cN\epsilon_N^2}\pi_N(u\in V_J:\norm{u}_{H^{m+\beta_0}(\domain)}\leq C\epsilon_N)
	\end{equation*}
	Using the notation from (a) and letting $Z\overset{d}{\sim} N(0,1)$, the last expression is lower bounded by
	\begin{align*}
		&e^{-cN\epsilon_N^2}\prob[][][]{\sum_{\abs{\mu}\le J}Z^2_{\mu}\leq c'\epsilon_N^2}\geq e^{-cN\epsilon_N^2}\prob[][][]{\operatorname{dim}(V_J)\max_{\abs{\mu}\le J}Z^2_{\mu}\leq C^2\epsilon_N^2}\\
		&\quad = e^{-cN\epsilon_N^2}\prob[][][]{\operatorname{dim}(V_J)Z^2\leq C^2\epsilon_N^2}^{\operatorname{dim}(V_J)}\geq e^{-cN\epsilon_N^2}\prob[][][]{|Z|\leq C'N^{-1/2}}^{\operatorname{dim}(V_J)}
	\end{align*}
	for some $c',C'>0$, using again that $\operatorname{dim}(V_J)\lesssim N\epsilon_N^2$. Since the standard normal density is uniformly bounded away from zero near the origin, the last probability is up to a constant lower bounded by $N^{-\operatorname{dim}(V_J)/2}\gtrsim \exp(-cN^{d/(2\beta+d)}\log(N))\ge\exp(-cN\epsilon_N^2)$ for some $0<C<\infty$. Adjusting the constant for multiplicative factors yields the desired small ball lower bound.
	\item Suppose $u\in V_J$ with $\norm{u-u_0}_{H^m(\domain)}\leq L'''\epsilon_N=L'''N^{-\beta/(2\beta+d)}\log(N)$. Then, by the triangle inequality, we have
	\begin{align*}
		\norm{u}_{H^{m+\kappa}(\domain)}
			&\leq \norm{u-\ProjOne_Ju_0}_{H^{m+\kappa}(\domain)} + \norm{\ProjOne_Ju_0}_{H^{m+\kappa}(\domain)}\\
            &\le 2^{J\kappa}\norm{u-\ProjOne_Ju_0}_{H^m(\domain)}+ \norm{u_0}_{H^{m+\beta}(\domain)}\\
            &\leq 2^{J\kappa}\norm{u-u_0}_{H^m(\domain)} + 2^{J\kappa}\norm{(\identity-\ProjOne_J)u_0}_{H^m(\domain)} + \norm{u_0}_{H^{m+\beta}(\domain)}\\
            &\lesssim  N^{(\kappa-\beta)/(2\beta+d)}(1+\log(N)) + \norm{u_0}_{H^{m+\beta}(\domain)},
	\end{align*}
	which is bounded by a constant uniformly in $N\in\mathbb{N}$ as long as $\kappa<\beta$ and grows logarithmically if $\kappa=\beta$. This proves the claim taking $L''>\norm{u_0}_{H^{m+\beta}(\domain)}$ large enough.
    \end{enumerate}

\end{proof}

The following theorem is an adaptation of Theorem 13 in \cite{GN20} to our setting.

\begin{thm}\label{thm:GeneralContraction}
Let $\pi_N$ be a sequence of prior Borel probability measures on $V_J$. Fix some $u_0$ for which we observe the data $\data$, and let $\Pi_N(\MTemptyplaceholder|\data)$ be the resulting posterior distribution. Suppose $\epsilon_N>0$ is a sequence such that $\epsilon_N\to 0$ and $\sqrt N\epsilon_N\to \infty$ as $N\to \infty$. Define the `KL-moment' neighbourhoods
\begin{equation*}
\begin{split}
    B_N\coloneqq \Big\{ u\in V_J \colon &\EV[\prob[u_0]{}][][\Big]{\log \Big(\frac{\D\prob[u_0]{(X_1,Y_1)}}{\D\prob[u]{(X_1,Y_1)}} \Big)}\le \epsilon_N^2,\\
    ~&\EV[\prob[u_0]{}][][\Big]{ \log \Big(\frac{\D\prob[u_0]{(X_1,Y_1)}}{\D\prob[u]{(X_1,Y_1)}} \Big)^2} \le \epsilon_N^2 \Big\}.
\end{split}
\end{equation*}
Assume the following conditions:
\begin{enumerate}[(a)]
    \item\label{num:General_Contraction_SmallBall} For all $N$ large enough,
    \[ \pi_N(B_N)\ge c_1e^{-c_2N\epsilon_N^2},~\text{ for some }c_1,c_2>0. \]
    \item\label{num:General_Contraction_SieveMass} There is a sequence of Borel sets $\mathcal{U}_N\subseteq V_J$ such that 
    \[ \pi_N(\mathcal{U}_N^c)\lesssim e^{-c_3N\epsilon_N^2},~ \text{ for some } c_3>c_2+2. \]
    \item\label{num:General_Contraction_SieveEntropy} We have the entropy bound
    \[ \log (N(\mathcal{U}_N, h,\epsilon_N))\le c_4 N\epsilon_N^2 \]
    for some $c_4>0$.
\end{enumerate}
Then $\epsilon_N$ is a rate of contraction for estimating $\prob[u_0]{}$, in the Hellinger distance, i.e. for any $L>4$ with $L^2>12\max(c_3, c_4)$ and all $0<D<c_3-c_2-2$ we have
\begin{equation*}
    \Pi_N( u\in V_J\colon  h(\prob[u]{},\prob[u_0]{})\ge L\epsilon_N|\data ) = \mathcal{O}_{\prob[u_0][N]{}}(e^{-DN\epsilon_N^2})
\end{equation*}
as $N\to\infty$.
\end{thm}

\begin{proof}[Proof of Theorem \ref{thm:Contraction}]
Our goal is to apply Theorem \ref{thm:GeneralContraction} to first derive contraction rates for $\op[u]$ in Hellinger distance. In the next step, we use the equivalence of the Hellinger distance to $\norm{}_{L^2(\domain)}$ from Lemma \ref{lem:RandomDesign_Hellinger/KL} to extend the contraction to $L^2(\domain)$-distance for $\op[u]$ and subsequently we use the bounds \eqref{eq:Stability_Lipschitz_Bounds} to extend the contraction to $u$.
 Before verifying the small ball condition, we note that \eqref{eq:Lipschitz_Bounds_infty}, the Sobolev embedding, the Jackson and Bernstein estimates \eqref{eq:Jackson}, \eqref{eq:Bernstein}, and Lemma \ref{lem:Prior} \eqref{num:aux_rescaled_Prior_restrict} imply that for any $u\in V_J$ with $\norm{u-u_0}_{H^m(\domain)}\le \epsilon_N$ and any $d/2<\eta<\beta$ we have
 \begin{align*}
 %\begin{split}
 S(u,u_0)&\coloneqq \norm{\op[u]-\op[u_0]}_{L^\infty(\domain)}\lesssim \norm{u-u_0}_{W^{m,\infty}(\domain)}\\
 &\lesssim \norm{u-\ProjOne_J u_0}_{H^{m+\eta}(\domain)} + \norm{(\identity-\ProjOne_J)u_0}_{H^{m+\eta}(\domain)}\\
 &\lesssim 2^{J\eta}\epsilon_N + 2^{-J(\beta-\eta)}\lesssim 1.
 %\end{split}\label{eq:UniformBound_Candidates}
 \end{align*} Consequently, for such $u$ we have $S(u,u_0)\le \bar{C}$ with $S(u,u_0)$ from Lemma \ref{lem:RandomDesign_Hellinger/KL} for some constant $\bar{C}>0$ independent of $u$ and $N$.
 By Lemma \ref{lem:RandomDesign_Hellinger/KL} and the upper bound from \eqref{eq:Stability_Lipschitz_Bounds} we have
 \begin{align*}
    B_N&=  \Big\{u\in V_J \colon \EV*[\prob[u_0]{}]{\log \left(\frac{\D\prob[u_0]{(X_1,Y_1)}}{\D\prob[u]{(X_1,Y_1)}}\right) } \le \epsilon_N^2,\\
    &\qquad\qquad\qquad \qquad \EV*[\prob[u_0]{}]{\log \left(\frac{\D\prob[u_0]{(X_1,Y_1)}}{\D\prob[u]{(X_1,Y_1)}}\right) }^2 \le \epsilon_N^2\Big\}\\
    &\supset\big\{u\in V_J \colon \norm{\op[u_0]-\op[u]}_{L^2(\domain)}^2 \le 2\sigma^2\epsilon_N^2/\bar{p}, \\
    &\qquad\qquad\qquad\qquad \norm{\op[u_0]-\op[u]}_{L^2(\domain)}^2 \le \epsilon_N^2\sigma^4/(2\bar{p}(\bar{C}^2 +\sigma^2))\big\}\\
    %&\supset \left\{u\in V_J \colon \norm{u_0-u}_{H^m(\domain)}^2 \le c\sigma^2\epsilon_N^2, \EV*[u_0]{\log \left(\frac{p(u_0,Y)}{p(u,Y)}\right) }^2 \le \epsilon_N^2\right\}\\
    %&\supset \left\{u\in V_J\colon \norm{u_0-u}_{H^m(\domain)}^2 \le c\sigma^2\epsilon_N^2, \norm{\op[u_0]-\op[u]}_{L^2(\domain)}^2\le \epsilon_N^2\sigma^4/(2\bar{C}^2+\sigma^2)\right\}\\
    &\supset  \left\{u\in V_J\colon \norm{u_0-u}_{H^m(\domain)}^2 \le L'\epsilon_N^2\right\}
 \end{align*}
for some constant $0<L'<1$ independent of $N$ sufficiently small.
Lemma \ref{lem:Prior} \eqref{num:aux_rescaled_Prior_smallball} implies that
\begin{equation}
\pi_N(B_N)\ge  \pi_N\left(\norm{u_0-u}_{H^m(\domain)}^2\le \tilde{C}\epsilon_N^2\right)\ge \exp(-C_2N\epsilon_N^2)\label{eq:aux_KL_Ball_Mass}
\end{equation}
and condition \eqref{num:General_Contraction_SmallBall} of Theorem \ref{thm:GeneralContraction} follows with $c_1 = 1$ and $c_2=C_2$. Conditions \eqref{num:General_Contraction_SieveMass} and \eqref{num:General_Contraction_SieveEntropy} are verified by Lemma \ref{lem:Prior} \eqref{num:aux_rescaled_Prior_sieve} and \eqref{num:aux_rescaled_Prior_entropy} with $c_3 = C_0$ and some $c_4>0$ since 
\begin{equation*}
    h^2(\prob[u]{},\prob[v]{})\le  \frac{1}{4\sigma^2}\norm{\op[u]-\op[v]}_{L_p^2(\domain)}^2\le \frac{\bar{p}C}{4\sigma^2}\norm{u-v}_{H^m(\domain)}^2,\quad u,v\in \mathcal{V}
\end{equation*}
by Lemma \ref{lem:RandomDesign_Hellinger/KL} and \eqref{eq:Stability_Lipschitz_Bounds}. The condition $c_3>c_2+2$ can be achieved by choosing $C_0$ large enough.
Consequently, we have shown 
\begin{equation*}
    \Pi_N( u\colon  h(\prob[u]{},\prob[u_0]{})\ge L\epsilon_N\,|\, \data )=\mathcal{O}_{\prob*[u_0][N]{}}(e^{-DN\epsilon_N^2}),
\end{equation*}
for any $L>4$ with $L^2>12\max(c_3,c_4)$ and any $0<D<c_3-c_2-2$. By choosing $c_3=C_0$ large enough, we can ensure that $\tilde{D}<D< c_3-c_2-2$.\\
Before proceeding with the proof, we show that
\begin{equation}
    \prob*[u_0][N]{\int_{V_J} \prod_{i=1}^N \frac{\D \prob[v]{}}{\D \prob[u_0]{}}(X_i,Y_i) \,\D \pi_N(v) \ge e^{-\tilde{D}N\epsilon_N^2}} \to1,\label{eq:Marginal_Likelihood_Bound_RandomDesign}
\end{equation}
for some $0<\tilde{D}<D$, which will play an important role in the localisation step of the BvM proof in Lemma \ref{lem:Localisationisfine}. By Lemma 7.3.2 of \cite{GN16} we have 
\begin{equation*}
    \int_{V_J} \prod_{i=1}^N \frac{\D \prob[v]{}}{\D \prob[u_0]{}}(X_i,Y_i) \,\D \pi_N(v) \ge \pi_N(B_N)e^{-2N\epsilon_N^2}
\end{equation*}
with $\prob[u_0][N]{}$-probability at least $1-1/(N\epsilon_N^2)$. In combination with \eqref{eq:aux_KL_Ball_Mass} we find \eqref{eq:Marginal_Likelihood_Bound_RandomDesign} with $\tilde{D}= 2+ C_2=2+c_2$.\\
In the next step, we show that the posterior concentrates on the set
\begin{equation*}
    \mathcal{U}_N = \{u\in V_J\colon\norm{u}_{H^{m+\beta_0}(\domain)}\le L\sqrt{N}\epsilon_N\},
\end{equation*}
where $L$ is chosen such that Lemma \ref{lem:Prior} \eqref{num:aux_rescaled_Prior_sieve} yields $\pi_N(\mathcal{U}_N^c)\le e^{-C_0N\epsilon_N^2}$ for $C_0$ as large as desired. Since $\EV[\prob[u_0][N]{}]{\prod_{i=1}^N \frac{\D\prob[u]{}}{\D\prob[u_0]{}}(X_i,Y_i)} = 1$ we find by Markov's inequality
\begin{equation*}
    \prob[u_0][N]{\int_{\mathcal{U}_N^c}\prod_{i=1}^N\frac{\D\prob[u]{}}{\D\prob[u_0]{}}(X_i,Y_i)\,\D\pi_N(u)\ge e^{-(C_0-1)N\epsilon_N^2}}\le e^{-N\epsilon_N^2}.
\end{equation*}
On the intersection of this event with the event from \eqref{eq:Marginal_Likelihood_Bound_RandomDesign} we have $\Pi_N(\mathcal{U}_N^c|\data)\le e^{-(C_0-1-\tilde{D})N\epsilon_N^2}$ such that choosing $C_0$ sufficiently large yields posterior concentration on $\mathcal{U}_N$.\\
In the final step of the proof, we deduce the contraction on $\norm{}_{H^m(\domain)}$. %To this end, note that we restrict the analysis to the set $\mathcal{B}_N= \{u\in V_J\colon \norm{u}_{H^{m+\beta_0}(\domain)}\le L\sqrt{N}\epsilon_N\}$, where the posterior distribution asymptotically concentrates since $c_3>c_2+2$.
Fix any $d/2<\eta<\beta$ and deduce from \eqref{eq:Lipschitz_Bounds_infty}, the Sobolev embedding, and the
Jackson and Bernstein estimates \eqref{eq:Jackson},
\eqref{eq:Bernstein} the bound
\begin{align*}
    S(u,u_0)&\lesssim \norm{u-\ProjOne_Ju_0}_{W^{m,\infty}(\domain)}+\norm{\ProjOne_Ju_0-u_0}_{W^{m,\infty}(\domain)}\\
    &\lesssim \norm{u-\ProjOne_Ju_0}_{H^{m+\eta}(\domain)}+\norm{\ProjOne_Ju_0-u_0}_{H^{m+\eta}(\domain)}\\
    &\lesssim 2^{J\eta}\norm{u-\ProjOne_Ju_0}_{H^m(\domain)}+2^{-J(\beta-\eta)}\norm{u_0}_{H^{m+\beta}(\domain)}.
\end{align*}
The stability estimate \eqref{eq:Stability_Lipschitz_Bounds} and the Jackson estimate further imply
\begin{align*}
    \norm{u-\ProjOne_Ju_0}_{H^m(\domain)}&\le\norm{u-u_0}_{H^m(\domain)}+\norm{(\identity-\ProjOne_J)u_0}_{H^m(\domain)}\\
    &\lesssim \norm{\op[u]-\op[u_0]}_{L_p^2(\domain)}+2^{-J\beta}.
\end{align*}
Consequently, we find
\begin{equation}
    S(u,u_0)\lesssim 2^{J\eta}\big(\norm{\op[u]-\op[u_0]}_{L_p^2(\domain)}+2^{-J\beta}\big),\quad u\in V_J.\label{eq:Contraction_finite_dimensional_envelope_2}
\end{equation}
By Lemma \ref{lem:RandomDesign_Hellinger/KL} and the elementary bound $z/(1-e^{-z})\lesssim 1+z$, $z\ge0$, we have
\begin{equation}
    \norm{\op[u]-\op[u_0]}_{L_p^2(\domain)}^2 \lesssim\big(1+S(u,u_0)^2\big)
    h^2(\prob[u]{},\prob[u_0]{}).\label{eq:Hellinger_to_prediction_with_envelope}
\end{equation}
Let $L>0$ be fixed and suppose that $h(\prob[u]{},\prob[u_0]{})\le L\epsilon_N$. Combining \eqref{eq:Contraction_finite_dimensional_envelope_2} and \eqref{eq:Hellinger_to_prediction_with_envelope}, and absorbing $L$ into the constants, yields
\begin{align}
    \norm{\op[u]-\op[u_0]}_{L_p^2(\domain)}^2&\lesssim
    \epsilon_N^2+\epsilon_N^2 2^{2J\eta}\big(\norm{\op[u]-\op[u_0]}_{L_p^2(\domain)}+2^{-J\beta}\big)^2\nonumber\\
    &\lesssim\epsilon_N^2+\epsilon_N^2 2^{2J\eta}\norm{\op[u]-\op[u_0]}_{L_p^2(\domain)}^2+
    \epsilon_N^2 2^{-2J(\beta-\eta)}.\label{eq:Contraction_absorption_inequality}
\end{align}
Since $2^J\sim N^{1/(2\beta+d)}$ and $\epsilon_N=N^{-\beta/(2\beta+d)}\log(N)$, we have
\begin{equation*}
    \epsilon_N^2 2^{2J\eta}\lesssim N^{-2(\beta-\eta)/(2\beta+d)}\log(N)^2=\smallo(1)
\end{equation*}
as $N\to\infty$.
Thus, for all $N$ sufficiently large, the second term on the right-hand side of \eqref{eq:Contraction_absorption_inequality} can be absorbed into the left-hand side. It follows that for any $u\in V_J$ with $h(\prob[u]{},\prob[u_0]{})\le L\epsilon_N$ we have $\norm{\op[u]-\op[u_0]}_{L_p^2(\domain)}\lesssim \epsilon_N$.
The constants in this implication may depend on $L$, but not on $u$ or $N$. The lower stability estimate in \eqref{eq:Stability_Lipschitz_Bounds} then shows that also $\norm{u-u_0}_{H^m(\domain)}\lesssim \epsilon_N$.
Consequently, after increasing the multiplicative constant in the
$H^m(\domain)$-ball if necessary,
\begin{align*}
    \Pi_N\left(u\in V_J\colon \norm{u-u_0}_{H^m(\domain)}>C_L\epsilon_N \,\middle|\,\data\right)&\le \Pi_N\left(u\in V_J\colon h(\prob[u]{},\prob[u_0]{})>L\epsilon_N \,\middle|\,\data\right)\\
    &=\mathcal{O}_{\prob[u_0][N]{}}
    \big(e^{-DN\epsilon_N^2}\big).
\end{align*}
This completes the proof.
\end{proof}

\begin{proof}[Proof of Lemma \ref{lem:Convergence_PosteriorMean_Nonparametric}]
    For $L>0$ from Theorem \ref{thm:Contraction} define the set $\mathcal{A}_N=\{u\in V_J\colon \norm{u-u_0}_{H^m(\domain)}\le L\epsilon_N\}$ such that $\Pi_N(\mathcal{A}_N^c|\data)=\mathcal{O}_{\prob[u_0][N]{}}(e^{-DN\epsilon_N^2})$ as $N\to\infty$. By Jensen's inequality we find
    \begin{equation*}
        \norm{\bar{u}_N-u_0}_{H^m(\domain)}\le \EV[\Pi_N]{\norm{u-u_0}_{H^m(\domain)}\given \data}\le L\epsilon_N + \EV[\Pi_N]{\norm{u-u_0}_{H^m(\domain)}\indicator_{\mathcal{A}_N^c}\given\data}.
    \end{equation*}
    we only need to control the second summand. Since $\EV[\prob[u_0][N]{}]{e^{\ell_N(u)-\ell_N(u_0)}} = 1$, we find with $\Omega_N$ the event from \eqref{eq:Marginal_Likelihood_Bound_RandomDesign} that
    \begin{align*}
        \EV[\prob[u_0][N]{}]{\indicator_{\Omega_N}\EV[\Pi_N]{\norm{u-u_0}_{H^m(\domain)}^2\given\data}} \hspace{-5em}&\\
        &\le e^{\tilde{D}N\epsilon_N^2}\int_{V_J}\norm{u-u_0}_{H^m(\domain)}^2\EV[\prob[u_0][N]{}]{e^{\ell_N(u)-\ell_N(u_0)}}\,\D\pi_N(u)\\
        &= e^{\tilde{D}N\epsilon_N^2}\int_{V_J}\norm{u-u_0}_{H^m(\domain)}^2\,\D\pi_N(u).
    \end{align*}
    Moreover, we have
    \begin{align*}
        \int_{V_J}\norm{u-u_0}_{H^m(\domain)}^2\,\D\pi_N(u)&\lesssim  1+ \int_{V_J}\norm{u}_{H^m(\domain)}^2\,\D\pi_N(u)\lesssim 1 + \sum_{\abs{\mu}\le J}2^{2m\abs{\mu}}2^{-2(m+\beta_0)\abs{\mu}}\\
        &\lesssim \sum_{j\le J}2^{jd}2^{-2j\beta_0}\le \sum_{j\in\mathbb{N}}2^{-j(2\beta_0-d)}\lesssim 1
    \end{align*}
    by \eqref{eq:CharacterisationSobolevSpaces} and since $\beta_0>d/2$. Consequently, we find 
    \begin{equation*}
        \EV[\Pi_N]{\norm{u-u_0}_{H^m(\domain)}^2\given\data}=\mathcal{O}_{\prob[u_0][N]{}}(e^{\tilde{D}N\epsilon_N^2})
    \end{equation*}
    and thus
    \begin{align*}
        \EV[\Pi_N]{\norm{u-u_0}_{H^m(\domain)}\indicator_{\mathcal{A}_N^c}\given\data}^2 &\le \EV[\Pi_N]{\norm{u-u_0}_{H^m(\domain)}^2\given\data}\Pi_N(\mathcal{A}_N^c|\data)\\
        &= \mathcal{O}_{\prob[u_0][N]{}}(e^{-(D-\tilde{D})N\epsilon_N^2}).
    \end{align*}
    Since $\tilde{D}<D$, this term is $\smallo_{\prob[u_0][N]{}}(\epsilon_N)$ the proof is complete.
\end{proof}

\begin{lem}\label{lem:MAP_Convergence}
    Grant Assumptions \ref{assump:ApproximatonSets} and \ref{assump:Operator_stability}.
    Assume that $u_0\in\mathcal{V}\cap H^{m+\beta}(\domain)$ with $\beta>d/2$ and $d/2<\beta_0<\beta+d/2$.
     Consider the prior $\pi_N$ from \eqref{eq:Prior} with cut-off $2^J\sim N^{1/(2\beta+d)}$. Then the MAP estimator satisfies 
%\[ \|\op[\hat{u}_N^{\operatorname{MAP}}]-\op[u_0]\|_{L^2(\domain)} = O_{P}(\delta_N)~~~~\text{for}~~~~\delta_N:= N^{-\frac{2\beta-d}{2(2\beta+d)}}. \]
%If also $\beta >d$, then we also have the following convergence in the $W^{2,\infty}$-norm,
%    \[ \|\hat{u}_N^{\operatorname{MAP}} - u_0\|^2_{W^{2,\infty}}= O_P(N^{-\frac{\beta -d}{2\beta+d}}).\]
\begin{align}
    \norm{\hat{u}_N^{\operatorname{MAP}}-u_0}_{H^m(\domain)}&=\mathcal{O}_{\prob[u_0][N]{}}(\epsilon_N),\label{eq:MAP_H^M_Convergence}\\
    \norm{\hat{u}_N^{\operatorname{MAP}}-u_0}_{W^{m,\infty}(\domain)}&=\mathcal{O}_{\prob[u_0][N]{}}(2^{J\eta}\epsilon_N)\label{eq:MAP_W^M_Convergence}
\end{align}
for any $d/2<\eta<\beta$,
as well as
\begin{equation}
    \EV[\prob[u_0][N]{}]{\norm{\op[\hat{u}_N^{\operatorname{MAP}}]-\op[u_0]}_N^2}\lesssim \epsilon_N^2.\label{eq:MAP_PredictionConvergence}
\end{equation}
%\[\E\big[ \|\op[\hat{u}_N^{\operatorname{MAP}}]- \op[u_0]\|_{N}^{2} \big] \le C N^{-2\beta/(2\beta+d)}\log(N)^2,\]
%and also
%\[ \E \|\hat{u}_N^{\operatorname{MAP}} - u_0\|_{H^2}^2 \le C'  \E \|\op[\hat{u}_N^{\operatorname{MAP}}] - \op[u_0]\|_{L^2(\domain)}^2\le C'' N^{-(2\beta-d)/(2\beta+d)} \log(N)^3. \]

\end{lem}
\begin{proof} \textbf{Step 1: rate for empirical norm.}
We use Theorem 3.1 in \cite{wangGlobalPolynomialtimeEstimation2026} to first derive a rate for the empirical norm. Define the functional
\[\tau^2(u,v) = \|\op[u]-\op[v]\|_N^2+\frac{\sigma^2}{N} \|u\|_{\mathbb{H}_N}^2, \quad u,v\in \mathcal{V}. \]
We define the sets of functions
\[\begin{aligned}
     \mathcal{U}^\ast(R)= \{ u\in V_J:  \tau^2(u,\ProjOne_J u_0) \le R^2\},\quad \mathcal{V}^\ast(R)= \{ \op[u]: u \in  \mathcal{U}^\ast(R)\},\quad R>0.\\
\end{aligned}  \]
For $u\in\mathcal{U}^\ast(R)$ we have $\norm{u}_{\mathbb{H}_N}\le \sqrt{N}R$. The Lipschitz bound \eqref{eq:Lipschitz_Bounds_infty} combined with the Sobolev embedding and \eqref{eq:CharacterisationSobolevSpaces} yields 
\begin{equation*}
   \norm{\op[u]-\op[v]}_N\le \norm{\op[u]-\op[v]}_{L^\infty(\domain)}\lesssim \norm{u-v}_{H^{m+\beta_0}(\domain)}\lesssim \norm{u-v}_{\mathbb{H}_N}\lesssim \sqrt{N}R
\end{equation*}
for $u,v\in \mathcal{U}^\ast(R)$.
By the covering numbers for finite-dimensional balls (Proposition 4.3.34 of \cite{GN16}) we find
\begin{equation*}
    \log(N(\mathcal{V}^\ast(R),\norm{}_N,\rho))\lesssim \operatorname{dim}(V_J)\log\left(\frac{C\sqrt{N}R}{\rho}\right)
\end{equation*}
for any $\rho>0$.
%By \eqref{eq:CharacterisationSobolevSpaces} we find \[\mathcal{U}^\ast(R)\subseteq \{u \in V_J : \|u\|_{H^{m+\beta_0}(\domain)}^2 \le N R^2 \} \subseteq \{ u\in V_J: \|u\|_{L^2(\domain)}^2\le NR^2 \}.\]
%Note that $\norm{u}_{L^\infty(\domain)}\le C \norm{u}_{H^{\eta}(\domain)}\le \bar{C} 2^{J\eta }\norm{u}_{L^2(\domain)}$ for any $u\in V_J$ and $\eta>d/2$ by the Sobolev embedding and the Bernstein estimate \eqref{eq:Bernstein}.
%Denoting by $\operatorname{dim}(V_J)\simeq 2^{Jd}\simeq N^{d/(2\beta+d)}$  the dimension of the level-$J$ wavelet space and using the global Lipschitz estimate $\|\mathcal Lu - \mathcal Lu '\|_\infty\lesssim \|u - u '\|_\infty$, we have the 
%entropy bound (here $C,C',...$ denote generic constants)
%\begin{equation*}
%    \begin{split}
%    \forall \rho >0:~~~~~~H(\rho,\mathcal{V}^\ast(R), \|\cdot\|_N) &\le H(\rho,\mathcal{V}^\ast(R), \|\cdot\|_\infty) \\
%    &\le 
%    H(\rho,\mathcal{U}^\ast(R), C\|\cdot\|_\infty) \\
%    &\le H(\rho,\mathcal{U}^\ast(R), C'2^{Jd/2}\|\cdot\|_{L^2(\domain)})\\
%    &\le \operatorname{dim}(V_J) \log \big(C''2^{Jd/2}\sqrt NR / \rho \big),
%    \end{split}
%\end{equation*}
%since $\mathcal{U}^\ast(R)$ is bounded in %the $L^2$-norm by $\sqrt{N}R$. Hence
%\[
%H(\rho,\mathcal{V}^\ast(R),\|\cdot\|_N)
%\lesssim
%2^{Jd}\log\left(\frac{C2^{Jd/2}\sqrt{N}R}{\rho}\right).
%\]
A change of variables $\rho=Rt$ gives
\begin{equation*}
    \begin{split}
        \int_0^{\sigma R} \sqrt{\log(N(\mathcal{V}^\ast(R),\|\cdot\|_N,\rho)})\,\D\rho
&\lesssim
2^{Jd/2}
\int_0^{\sigma R}
\sqrt{\log\left(\frac{C\sqrt{N}R}{\rho}\right)}
\,\D\rho\\
&
\lesssim 2^{Jd/2}R\int_0^\sigma
\sqrt{\log\left(\frac{C\sqrt{N}}{t}\right)}
\,\D t \\
&\lesssim 
2^{Jd/2}R\sqrt{\log(N)}.
\end{split}
\end{equation*}
The right hand side is an upper bound for the metric entropy integral in Theorem 3.1 of \cite{wangGlobalPolynomialtimeEstimation2026}. For a suitable constants $0<C<\infty$ we define
\[\Psi(R) \coloneqq R+ C N^{\frac{d}{2(2\beta+d)}} R\sqrt{\log(N)}.\]
Note that $R\mapsto \Psi(R)/R^2$ is non-increasing and
\[ \sqrt N \epsilon_N^2 = N^{-\frac{2\beta-d}{2(2\beta+d)}}\log(N)^2 \gtrsim \Psi(\epsilon_N). \]
Indeed, this follows from $\epsilon_N \gtrsim 1/\sqrt N$ and the fact that
\[ \sqrt N \epsilon_N^2 \ge \sqrt N\epsilon_N^2/\log(N) = \epsilon_N N^{\frac{d}{2(2\beta+d)}},\]
We may thus apply Theorem 3.1 of \cite{wangGlobalPolynomialtimeEstimation2026} to deduce that for all $k\ge 1$ we have the moment bounds
\begin{equation*}
\EV[\prob[u_0][N]{}]{ \tau^{2k}(\hat{u}_N^{\operatorname{MAP}},u_0) } \le C_k \Big(
\tau^{2k}(\ProjOne_J u_0,u_0) 
+ \epsilon_N^{2k}+\epsilon_N^{2k-2}\frac{\sigma^2}{N}\Big)\lesssim \tau^{2k}(\ProjOne_J u_0,u_0) + \epsilon_N^{2k}
\end{equation*}
for some $C_k>0$, where we used $\epsilon_N\gtrsim N^{-1/2}$.
Furthermore, we obtain the concentration inequality
 \begin{equation}\label{tau-concentration}
     \prob[u_0][N][\Big]{\tau^{2}(\hat{u}_N^{\operatorname{MAP}},u_0) \ge 2\big(\tau^{2}(\ProjOne_J u_0,u_0) + R^2 \big)}\le C\exp\Big( -\frac{NR^2}{C}\Big),\quad\forall R\ge \epsilon_N.
 \end{equation}
Next, we control the (random) empirical norm within the term $\tau^{2}(\ProjOne_J u_0,u_0)$ on the right hand side.
By the Jackson estimate \eqref{eq:Jackson} and the Lipschitz bounds \eqref{eq:Stability_Lipschitz_Bounds}, \eqref{eq:Lipschitz_Bounds_infty} we find
\begin{align*}
    \norm{\op[\ProjOne_J u_0]-\op[u_0]}_{L_p^2(\domain)}&\lesssim \norm{\ProjOne_J u_0-u_0}_{H^m(\domain)}\lesssim 2^{-J\beta}\norm{u_0}_{H^{m+\beta}(\domain)}\lesssim 2^{-J\beta}\\
    \norm{\op[\ProjOne_J u_0]-\op[u_0]}_{L^\infty(\domain)}&\lesssim \norm{\ProjOne_J u_0-u_0}_{W^{m,\infty}(\domain)}\lesssim 2^{-J(\beta-\eta)}\norm{u_0}_{H^{m+\beta}(\domain)}\lesssim 2^{-J(\beta-\eta)}
\end{align*}
for any $\eta>d/2$. Define the centred random variables
\begin{equation*}
    Z_i\coloneqq(\op[u_0]- \op[\ProjOne_J u_0])^2(X_i)-\|\op[u_0]-\op[\ProjOne_J u_0]\|_{L_p^2(\domain)}^2,\quad i=1,\dots, N,
\end{equation*}
then
\begin{align*}
    \abs{Z_i}&\le \|\op[u_0]-\op[\ProjOne_J u_0]\|_{L^\infty(\domain)}^2 \lesssim 2^{-2J(\beta-\eta)},\\
        \EV{Z_i^2} &\lesssim \|\op[u_0]-\op[\ProjOne_J u_0]\|_{L^2(\domain)}^2\|\op[u_0]-\op[\ProjOne_J u_0]\|_{\infty}^2\lesssim 2^{-2J(2\beta-\eta)},
\end{align*}
for any $i=1,\dots, N$. Applying Bernstein's inequality
%\begin{equation*}
%    \begin{split}
%        Z_i&=(\op[u_0]- \op[\ProjOne_J u_0])^2(X_i)-\|\op[u_0]-\op[\ProjOne_J u_0]\|_{L^2(\domain)}^2,\\ |Z_i|&\le \|\op[u_0]-\op[\ProjOne_J u_0]\|_\infty^2 \lesssim 2^{-2J(\beta -d/2)}\|\op[u_0]\|_{H^{\beta}}^2 \lesssim \epsilon_N^2 2^{Jd},\\
%        \E [Z_i^2] &\lesssim \|\op[u_0]-\op[\ProjOne_J u_0]\|_{L^2(\domain)}^2\|\op[u_0]-\op[\ProjOne_J u_0]\|_{\infty}^2\lesssim \epsilon_N^42^{Jd},
%    \end{split}
%\end{equation*}
shows that for any $L\ge 1$, and some $C,C'>0$,
\begin{equation*}
    \begin{split}
    \prob[u_0][N]{\|\op[u_0]-\op[\ProjOne_J u_0]\|_N^2-\|\op[u_0]-\op[\ProjOne_J u_0]\|_{L_p^2(\domain)}^2\ge L^2\epsilon_N^2 }\hspace{-10em}&\\
    &\le 2\exp\Big(-\frac{NL^4\epsilon_N^4}{C(2^{-2J(2\beta-\eta)}+L^2\epsilon_N^2 2^{-2J(\beta-\eta)})}\Big)\\
    &\le 2\exp(-C' L^2 N 2^{-2J\eta}\log(N)^2)\\
    &\le 2\exp(-C' L^2N\epsilon_N^2),
    %&\lesssim \exp(-N L^2/2^{Jd}C')\\
    %&\lesssim \exp(-\epsilon_N^{-2}L^2/C'')\\
    %&\lesssim \exp(-N\epsilon_N^2L^2/C'').
    \end{split}
\end{equation*}
where, in the final step, we used $\beta>\eta$.
%Here, in the final step, we also used that $\epsilon_N = o( N^{-1/4})$ due to our assumption $\beta > d/2 $.
Noting also that 
\begin{equation*}
    \frac{1}{N}\|\ProjOne_J u_0\|_{\mathbb{H}_N}^2\lesssim \frac{1}{N}\|\ProjOne_Ju_0\|_{H^{m+\beta_0+\epsilon}}^2\le \frac{1}{N}2^{2J(\beta_0+\epsilon-\beta)_+}\|u_0\|_{H^{m+\beta}}^2\lesssim \epsilon_N^2
\end{equation*}
since $\beta_0<\beta+d/2$, the above computations as well as \eqref{tau-concentration} imply that there exists $K_0$ such that for all $K\ge K_0$ and some $0<C<\infty$, 
\begin{equation}
\prob[u_0][N]{\tau^2(\hat{u}_N^{\operatorname{MAP}}, u_0)\ge K^2\epsilon_N^2} \le C\exp\Big(-\frac{NK^2\epsilon_N^2}{C}\Big).\label{eq:MAP_Convergence_EmpiricalNorm}
\end{equation}
We also have the moment bounds
\begin{equation*}
    \EV[\prob[u_0][N]{}]{ \tau^{2k}(\hat{u}_N^{\operatorname{MAP}}, u_0) } \le C_k\epsilon_N^{2k}
\end{equation*}
for any $k\ge 1$, which shows \eqref{eq:MAP_PredictionConvergence}.\\
\textbf{Step 2: rate for the $L_p^2(\domain)$-norm.}
For $s_N\ge K\epsilon_N$ and with $K,L>0$ to be chosen later, define the shell
\begin{equation*}
    \mathcal{M}_{s_N}\coloneqq\left\{u\in V_J\colon s_N<\norm{\op[u]-\op[u_0]}_{L_p^2(\domain)} \le 2s_N,\norm{u}_{\mathbb H_N}\le L\sqrt{N}\epsilon_N \right\}.
\end{equation*}
For every $u\in\mathcal{M}_{s_N}$, the stability estimate \eqref{eq:Stability_Lipschitz_Bounds} and the Jackson estimate \eqref{eq:Jackson} imply
\begin{equation*}
    \norm{u-\ProjOne_Ju_0}_{H^m(\domain)}\le\norm{u-u_0}_{H^m(\domain)}+\norm{(\identity-\ProjOne_J)u_0}_{H^m(\domain)}\lesssim s_N+2^{-J\beta}\lesssim s_N,
\end{equation*}
where we used $2^{-J\beta}\lesssim\epsilon_N\le s_N/K$ Moreover, the $H^m(\domain)$-radius of $\mathcal{M}_{s_N}$ is of the order $s_N$:
\[\mathcal{M}_{s_N}\subseteq \{ u\in V_J: \norm{u-u_0}_{H^m(\domain)}\le Cs_N \}\]
for some $C>0$. Let $d/2<\eta<\beta$.
By \eqref{eq:Lipschitz_Bounds_infty}, the Sobolev embedding, and the Jackson and Bernstein estimates \eqref{eq:Bernstein}, \eqref{eq:Jackson}, it follows that
\begin{align*}
    \norm{\op[u]-\op[u_0]}_{L^\infty(\domain)}&\le\norm{\op[u]-\op[\ProjOne_Ju_0]}_{L^\infty(\domain)}+\norm{\op[\ProjOne_Ju_0]-\op[u_0]}_{L^\infty(\domain)}\nonumber\\
    &\lesssim 2^{J\eta}\norm{u-\ProjOne_Ju_0}_{H^m(\domain)}+2^{-J(\beta-\eta)}\lesssim 2^{J\eta}s_N.
\end{align*}
Consequently, the class $\mathcal{W}_{s_N}\coloneqq\left\{\op[u]-\op[u_0]\colon u\in\mathcal{M}_{s_N}\right\}$ satisfies
\begin{equation}
    M_{s_N}\coloneqq \sup_{f\in\mathcal{W}_{s_N}} \norm{f}_{L^\infty(\domain)} \lesssim 2^{J\eta}s_N.\label{eq:MAP_shell_envelope_M}
\end{equation}
For any $u,v\in \mathcal{M}_{s_N}\subset V_J$ we find using the stability and Lipschitz estimates \eqref{eq:Stability_Lipschitz_Bounds}, \eqref{eq:Lipschitz_Bounds_infty}, and the definition of $\mathcal{M}_{s_N}$ that
\begin{equation*}
    \norm{\op[u]-\op[v]}_{L^\infty(\domain)}\lesssim \norm{u-v}_{W^{m,\infty}(\domain)}\lesssim2^{J\eta}\norm{u-v}_{H^m(\domain)}%\lesssim2^{J\eta}\norm{\op[u]-\op[v]}_{L_p^2(\domain)}\\
   % &\lesssim 2^{J\eta}s_N
\end{equation*}
with hidden constant independent of $u,v\in\mathcal{M}_{s_N}$. 
Using the covering-number bound
for finite-dimensional balls (Proposition 4.3.34 of \cite{GN16}), we obtain
\begin{equation}
\begin{split}
\log\Big(N\Big(\mathcal{W}_{s_N},\norm{}_{L^\infty(\domain)},\rho \Big)\Big)&\lesssim 
\log\Big(N\Big(\mathcal{M}_{s_N},C2^{J\eta}\norm{}_{H^m(\domain)},\rho \Big)\Big)
\\
&\lesssim \operatorname{dim}(V_J)\log\left(\frac{C2^{J\eta}s_N}{\rho} \right)\\
&\lesssim \operatorname{dim}(V_J) \log(N)\label{eq:MAP_shell_entropy}
\end{split}
\end{equation}
for any $0<\rho\le s_N$.
Choose a sufficiently small constant $c_0>0$ and set $\delta_N=c_0s_N$. By \eqref{eq:MAP_shell_envelope_M} and \eqref{eq:MAP_shell_entropy}, the entropy condition \eqref{eq:aux_Entropycondition} of Lemma \ref{lem:EmpiricalNorm} follows, because
\begin{equation*}
    \frac{N\delta_N^2}{M_{s_N}^2 \operatorname{dim}(V_J)\log(N)}\gtrsim\frac{N^{2(\beta-\eta)/(2\beta+d)}}{\log(N)}\to\infty,\quad N\to\infty,
\end{equation*}
for $\eta<\beta$. Furthermore, we find
\begin{equation*}
    \frac{M_{s_N}}{\sqrt N}\lesssim\frac{2^{J\eta}}{\sqrt N}s_N=\smallo(s_N).
\end{equation*}
After decreasing $c_0$ if necessary, the condition $s_N\ge C(\delta_N+M_{s_N}/\sqrt N)$ in Lemma \ref{lem:EmpiricalNorm} is therefore satisfied for all sufficiently large $N$. Applying that lemma with $R=s_N$ yields
\begin{equation}
    \prob[u_0][N][\Big]{\exists u\in\mathcal{M}_{s_N}\colon\norm{\op[u]-\op[u_0]}_N\le \frac{s_N}{2}}\le 2\exp\left(-\frac{Ns_N^2}{CM_{s_N}^2}\right)\le2\exp\left(-cN2^{-2J\eta}\right).\label{eq:MAP_shell_probability_revised}
\end{equation}
We now verify that only logarithmically many shells have to be considered for $\hat{u}_N^{\operatorname{MAP}}$. Assume that $u\in V_J$ satisfies $\norm{u}_{\mathbb H_N}\le L\sqrt{N}\epsilon_N$, then by \eqref{eq:CharacterisationSobolevSpaces} we also have $\norm{u}_{H^m(\domain)}\lesssim L\sqrt{N}\epsilon_N$ and we deduce that
\begin{equation}
    \norm{\op[u]-\op[u_0]}_{L_p^2(\domain)}\lesssim\norm{u-u_0}_{H^m(\domain)}\lesssim L\sqrt{N}\epsilon_N+\norm{u_0}_{H^m(\domain)}\le C_L\sqrt{N}\epsilon_N,\label{eq:MAP_maximal_shell_radius}
\end{equation}
for some $0<C_L<\infty$. Define the event
\begin{equation*}
    \mathcal{E}_N\coloneqq\left\{\norm{\op[\hat{u}_N^{\operatorname{MAP}}]-\op[u_0]}_N\le L\epsilon_N,\norm{\hat{u}_N^{\operatorname{MAP}}}_{\mathbb H_N}\le L\sqrt{N}\epsilon_N\right\}.
\end{equation*}
By \eqref{eq:MAP_Convergence_EmpiricalNorm}, for every sufficiently large $L$, we have $\prob[u_0][N]{\mathcal{E}_N}\to1$.
Fix such an $L$ and choose $K>2L$. If
$\mathcal{E}_N$ occurs and $\norm{\op[\hat{u}_N^{\operatorname{MAP}}]-\op[u_0]}_{L_p^2(\domain)}>K\epsilon_N$, then there is a unique integer $j\ge0$ such that
\begin{align*}
    &2^jK\epsilon_N<\norm{\op[\hat{u}_N^{\operatorname{MAP}}]-\op[u_0]}_{L_p^2(\domain)}\le 2^{(j+1)}K\epsilon_N\text{ and }\\
    &
    \norm{\op[\hat{u}_N^{\operatorname{MAP}}]-\op[u_0]}_N\le L\epsilon_N\le \frac{2^jK\epsilon_N}2.
\end{align*}
By \eqref{eq:MAP_maximal_shell_radius}, we deduce that $2^jK\epsilon_N<C_L\sqrt N\epsilon_N$.
Consequently, we require
\begin{equation*}
    0\le j\le j_{\max,N}\coloneqq\left\lceil\log_2\left(\frac{C_L\sqrt N}{K}\right)\right\rceil=\mathcal{O}(\log(N)).
\end{equation*}
Using \eqref{eq:MAP_shell_probability_revised} and the union bound, we conclude that
\begin{align*}
    \prob[u_0][N]{\norm{\op[\hat{u}_N^{\operatorname{MAP}}]-\op[u_0]}_{L_p^2(\domain)}>K\epsilon_N}&\le \prob[u_0][N]{\mathcal{E}_N^c}+2(j_{\max,N}+1)\exp\left(-cN2^{-2J\eta}\right)\to0
\end{align*}
as $N\to\infty$. We have shown that
\begin{equation*}
    \norm{\op[\hat{u}_N^{\operatorname{MAP}}]-\op[u_0]}_{L_p^2(\domain)}=\mathcal{O}_{\prob[u_0][N]{}}(\epsilon_N).
\end{equation*}
The stability estimate \eqref{eq:Stability_Lipschitz_Bounds} yields \eqref{eq:MAP_H^M_Convergence}. Finally, the Jackson and Bernstein estimates \eqref{eq:Jackson}, \eqref{eq:Bernstein} imply, for every $d/2<\eta<\beta$, that
\begin{align*}
    \norm{\hat{u}_N^{\operatorname{MAP}}-u_0}_{W^{m,\infty}(\domain)}&\lesssim 2^{J\eta}\norm{\hat{u}_N^{\operatorname{MAP}}-\ProjOne_Ju_0}_{H^m(\domain)}+2^{-J(\beta-\eta)}\norm{u_0}_{H^{m+\beta}(\domain)}\\
    &=\mathcal{O}_{\prob[u_0][N]{}}
    \left(2^{J\eta}\epsilon_N\right),
\end{align*}
which shows \eqref{eq:MAP_W^M_Convergence} and concludes the proof.
\end{proof}

\section{Proof for Section \ref{sec:BvM}}\label{sec:Proofs_BvM}

For $\psi\in\mathcal{V}^\ast$ define $\psi_0 \coloneqq \mathcal{I}_{u_0}^{-1}\psi\in\mathcal{V}$ and the perturbation
\begin{equation*}
    u_t = u - \frac{t}{\sqrt{N}}\ProjTwo_{J}\psi_0\in V_J,\quad u\in V_J, t\in\mathbb{R},\psi\in\mathcal{V}^\ast.
\end{equation*}
For $K,L,M>0$ and $t\in\mathbb{R}$ define the sets
\begin{align*}
\mathbb{A}_N &= \{ u\in V_J \,\colon\, \norm{u-u_0}_{H^m(\domain)}\le L\epsilon_N \text{ and }\norm{u}_{H^{m+\beta}(\domain)}\le K\log(N)\}\\
    A_N &=  \{ u_t\colon  u\in \mathbb{A}_N \text{ and }\psi\in \mathcal{V}^\ast, \norm{\psi}_{\mathcal{V}^\ast} \le  M\}.
\end{align*}

For notational simplicity we abbreviate the notations for the posterior distribution $\Pi_N(\MTemptyplaceholder)\coloneqq \Pi_N(\MTemptyplaceholder|\data)$ and $\EV[\Pi_N]{\MTemptyplaceholder}\coloneqq \EV[\Pi_N]{\MTemptyplaceholder\given\data}$ in this section.
Define the localised prior $\pi_N^{\mathbb{A}_N}\coloneqq \pi_N|_{\mathbb{A}_N}/\pi_N(\mathbb{A}_N)$ with corresponding posterior $\Pi_N^{\mathbb{A}_N}(\MTemptyplaceholder)=\Pi_N^{\mathbb{A}_N}(\MTemptyplaceholder|\data)$.

We also recall that $\mathcal{I}_{u_0}\colon \mathcal{V}\to\mathcal{V}^\ast$ is a topological isomorphism.
    In particular,
    \begin{align}
    \begin{split}
         \norm{\mathcal{I}_{u_0}^{-1}(v)}_{H^m(\domain)} &=\norm{\mathcal{I}_{u_0}^{-1}(v)}_{\mathcal{V}} \le C \norm{v}_{\mathcal{V}^\ast},\quad v\in \mathcal{V}^\ast,\\
        \norm{\mathcal{I}_{u_0}^{-1}(v|_\mathcal{V})}_{H^m(\domain)} &=\norm{\mathcal{I}_{u_0}^{-1}(v|_\mathcal{V})}_{\mathcal{V}}\le C \norm{(v|_\mathcal{V})}_{\mathcal{V}^\ast}\le C\norm{v}_{H^{-m}(\domain)},\quad v\in H^{-m}(\domain),
        \end{split}\label{eq:InformationIsomorphism}
    \end{align}
    where we denote by $(v|_\mathcal{V})$ the restriction of $v\in H^{-m}(\domain)$ to $\mathcal{V}\subset H^m(\domain)$. Furthermore, for $u\in \mathcal{V}$ and $\psi\in\mathcal{V}^\ast$, we have
    \begin{equation}
        \iprod[\mathcal{V}^\ast][\mathcal{V}]{\psi}{u} = \iprod[\mathcal{V}^\ast][\mathcal{V}]{\mathcal{I}_{u_0}\mathcal{I}_{u_0}^{-1}\psi}{u} = \frac{1}{\sigma^2}\iprod{\opDeriv{u_0}\mathcal{I}_{u_0}^{-1}\psi}{\opDeriv{u_0}[u]}_{L_p^2(\domain)} = \iprod{\mathcal{I}_{u_0}^{-1}\psi}{u}_{\operatorname{LAN}},\label{eq:FromDualtoLAN}
    \end{equation}
    where $\iprod[\mathcal{V}^\ast][\mathcal{V}]{}{}$ denotes the dual pairing between $\mathcal{V}^\ast$ and $\mathcal{V}$.

%The regularity condition ensures that $u_t=u+t\ProjTwo_{J}\psi_0/\sqrt{N}\in A_N$ for any fixed $t\in\mathbb{R}$ and $N$ large enough whenever $\norm{\psi}_{H^{\gamma}(\domain)}\le C$ for some $0<C<\infty$.

\subsection{Uniform LAN-expansion and convergence of the Laplace-transformation}

\begin{prop}[Uniform LAN expansion]\label{prop:uniformLANExpansion_RandomDesign}
    Grant Assumptions \ref{assump:ApproximatonSets}, \ref{assump:Operator_stability} and \ref{assump:operator_linearisation}. For $\beta>2d$ and $d/2<\beta_0<\beta+d/4$ consider $u_0\in \mathcal{V}\cap H^{m+\beta}(\domain)$ and the prior $\pi_N$ from \eqref{eq:Prior} with cut-off $2^J\sim N^{1/(2\beta+d)}$. Then we have the uniform LAN expansion
    \begin{equation*}
        \ell_N(u)-\ell_N(u_0) = -\frac{N}{2}\norm{u-u_0}_{\operatorname{LAN}}^2 +\frac{1}{\sigma}\sum_{i=1}^N \epsilon_i \opDeriv{u_0}[u-u_0] (X_i) + N(u),\quad u\in A_N,
    \end{equation*}
   % with
    %\begin{equation}
    %    \sup_{u\in A_N}\abs{N(u)}=\smallo_{\prob[u_0][N]{}}(1)\label{eq:remainder_general}
    %\end{equation}
   with     
   \begin{equation*}
        \sup_{u\in \mathbb{A}_N, \norm{\psi}_{\mathcal{V}^\ast}\le M}(\abs{N(u)} + \abs{N(u_t)})=\smallo_{\prob[u_0][N]{}}(1)
    \end{equation*}
   for any fixed $t\in\mathbb{R}$, $M>0$ and $N\in\mathbb{N}$ large enough.
\end{prop}

\begin{proof}%[Proof of Proposition \ref{prop:uniformLANExpansion_RandomDesign}]
%Denote $\epsilon_N =N^{-\beta/(2\beta+d)}$.
    Before starting the proof, note that for any $0\le \eta \le \beta$ we have by Sobolev interpolation that for any $u\in A_N$ and $0\le \eta\le\beta$ with $r_{N,\eta}=\epsilon_N^{(\beta-\eta)/\beta}\log(N)^{\eta/\beta}$ we have
    \begin{equation}
        \sup_{u\in A_N}\norm{u-u_0}_{H^{m+\eta}(\domain)}\le  \sup_{u\in A_N}\norm{u-u_0}_{H^m(\domain)}^{1-\eta/\beta} \norm{u-u_0}_{H^{m+\beta}(\domain)}^{\eta/\beta}\lesssim r_{N,\eta}\label{eq:interpolation_u}
    \end{equation}
    and by \eqref{eq:Lipschitz_Bounds_infty} for $d/2<\eta\le\beta$
    \begin{equation}
        \sup_{u\in A_N}\norm{\op[u]-\op[u_0]}_{L^\infty(\domain)}\lesssim \sup_{u\in A_N}\norm{u-u_0}_{H^{m+\eta}(\domain)}\lesssim r_{N,\eta},\label{eq:interpolation_Lu}
    \end{equation}
    %\begin{equation}
    %    \norm{u-u_0}_{H^{2+\eta}(\domain)}\le \norm{u-u_0}_{H^2(\domain)}^{1-\eta/\beta}\norm{u-u_0}_{H^{2+\beta}(\domain)}^{\eta/\beta}\lesssim \epsilon_N^{(\beta-\eta)/\beta},\quad u\in A_N,\label{eq:interpolation_Lu}
    %\end{equation}
    where we used that according to the Bernstein estimate \eqref{eq:Bernstein_orthogonalprojection} and \eqref{eq:InformationIsomorphism} we have
    \begin{equation*}
        \sup_{\psi\colon\norm{\psi}_{\mathcal{V}^\ast} \le  M}\frac{1}{\sqrt{N}}\norm{\ProjTwo_{J}\psi_0}_{H^{m+\beta}(\domain)}\lesssim N^{-1/2}2^{J\beta}\norm{\psi_0}_{H^m(\domain)}\lesssim N^{-1/2+\beta/(2\beta+d)}=\smallo(1).
    \end{equation*}
    %Furthermore, choose $\kappa' = \kappa(2\beta-d)/(2\beta)$, which satisfies $d/2<\kappa'<\kappa$ by the condition on $\beta$.   
For $u\in \mathbb{A}_N$ the loglikelihood difference can be decomposed as    
\begin{align*}
    \ell_N(u)-\ell_N(u_0) &= -\left(\frac{1}{2\sigma^2}\sum_{i=1}^N (Y_i - \op[u](X_i))^2 - \frac{1}{2\sigma^2}\sum_{i=1}^N (Y_i - \op[u_0](X_i))^2\right)\\
    &= -\frac{1}{2\sigma^2}\sum_{i=1}^N (\op[u](X_i)-\op[u_0](X_i))^2 + \frac{1}{\sigma}\sum_{i=1}^N \epsilon_i(\op[u](X_i)-\op[u_0](X_i))\\
    %&= -\frac{N}{2\sigma^2}\norm{\op[u]-\op[u_0]}_{L^2(\domain)}^2 + \frac{\sqrt{N}}{\sigma}\iprod{\eta}{\op[u]-\op[u_0]}_{L^2(\domain)}\\
    %&\quad + \frac{N}{2\sigma^2}\left(\norm{\op[u]-\op[u_0]}_{L^2(\domain)}^2- \frac{1}{N}\sum_{i=1}^N (\op[u](X_i)-\op[u_0](X_i))^2\right)\\
    %&\quad + \frac{\sqrt{N}}{\sigma}\left(\frac{1}{\sqrt{N}}\sum_{i=1}^N \epsilon_i(\op[u](X_i)-\op[u_0](X_i))- \sqrt{N}\iprod{\eta}{\op[u]-\op[u_0]}_{L^2(\domain)}\right).\\
    %&= -\frac{1}{2\sigma^2}\sum_{i=1}^N (gma}\sum_{i=1}^N \epsilon_i\opDeriv{u_0}[u-u_0](X_i) + \frac{1}{\sigma}\sum_{i=1}^N \epsilon_i(\op[u](X_i)-\op[u_0](X_i) - \opDeriv{u_0}[u-u_0](X_i))\\
    &= -\frac{N}{2}\norm{u-u_0}_{\operatorname{LAN}}^2 + \frac{\sqrt{N}}{\sigma}\frac{1}{\sqrt{N}}\sum_{i=1}^N \epsilon_i \opDeriv{u_0}[u-u_0] (X_i)\\
    &\quad  + \mathcal{R}_{N,1} + \mathcal{R}_{N,2}+ \mathcal{R}_{N,3}
\end{align*}
with 
\begin{align*}
    \mathcal{R}_{N,1} &= \sqrt{N}\frac{1}{\sigma \sqrt{N}}\sum_{i=1}^N \epsilon_i (\op[u](X_i)-\op[u_0](X_i)-\opDeriv{u_0}[u-u_0] (X_i))\\
    \mathcal{R}_{N,2} &= \frac{N}{2\sigma^2}\left(\norm{\op[u]-\op[u_0]}_{L_p^2(\domain)}^2- \frac{1}{N}\sum_{i=1}^N (\op[u](X_i)-\op[u_0](X_i))^2\right)\\
    %\mathcal{R}_{N,2} &= \frac{\sqrt{N}}{\sigma}\left(\frac{1}{\sqrt{N}}\sum_{i=1}^N \epsilon_i(\op[u](X_i)-\op[u_0](X_i))- \iprod{\eta}{\op[u]-\op[u_0]}_{L^2(\domain)}\right)\\
    \mathcal{R}_{N,3}&= \frac{N}{2}\left(\norm{u-u_0}_{\operatorname{LAN}}^2-\frac{1}{\sigma^2}\norm{\op[u]-\op[u_0]}_{L_p^2(\domain)}^2\right).
    %\mathcal{R}_{N,4} &= \frac{\sqrt{N}}{\sigma}\left(\iprod{\eta}{\op[u]-\op[u_0]}_{L^2(\domain)} - \iprod{\eta}{\opDeriv{u_0}[u-u_0]}_{L^2(\domain)}\right).
\end{align*}
%\begin{align*}
%    \mathcal{R}_{N,1} &= -\frac{1}{2\sigma^2}\left(\sum_{i=1}^N  \opDeriv{u_0}[u-u_0](X_i) ^2 -(\op[u](X_i)-\op[u_0](X_i))^2\right)\\
%    \mathcal{R}_{N,2} &= \frac{N}{2\sigma^2}\left(\frac{1}{N}\sum_{i=1}^N (\opDeriv{u_0}[u-u_0](X_i))^2-\norm{\opDeriv{u_0}[u-u_0]}_{L^2(\domain)}^2 \right)\\
%    \mathcal{R}_{N,3} &= \frac{1}{\sigma}\sum_{i=1}^N \epsilon_i\opDeriv{u_0}[u-u_0](X_i) - \frac{\sqrt{N}}{\sigma}\iprod{\eta}{\opDeriv{u_0}[u-u_0]}_{L^2(\domain)}\\
%    \mathcal{R}_{N,4} &= \frac{1}{\sigma}\sum_{i=1}^N \epsilon_i(\op[u](X_i)-\op[u_0](X_i) - \opDeriv{u_0}[u-u_0](X_i)).
%\end{align*}
We will show that each of the remainders is uniformly small over $u\in A_N$ by using the chaining lemma for empirical process in the form of Lemma 3.12 in \cite{NW20}.\\
\textbf{$\mathcal{R}_{N,1}$}:
For $u\in A_N$ denote $h_u\coloneqq \op[u]-\op[u_0] - \opDeriv{u_0}[u-u_0]$ such that 
\begin{equation*}
    \mathcal{R}_{N,1} = \frac{\sqrt{N}}{\sigma}\left(\frac{1}{\sqrt{N}}\sum_{i=1}^N \epsilon_ih_u(X_i)\right).
\end{equation*}
To apply Lemma 3.12 of \cite{NW20} we first compute using the linearisation \eqref{eq:linearisation_infty}, the Sobolev embedding and \eqref{eq:interpolation_u} that
\begin{align*}
    \sup_{u\in A_N}\norm{h_u}_{L^\infty(\domain)} &= \sup_{u\in A_N}\norm{\op[u]-\op[u_0]-\opDeriv{u_0}[u-u_0]}_{L^\infty(\domain)}\\
    %&= \sup_{u\in A_N}\norm{\tau(u)-\tau(u_0) - \tau'(u_0)(u-u_0)}_{L^\infty(\domain)}\\
    %&= \sup_{u\in A_N}\norm*{\int_0^1 \tau''(u_0 + t(u-u_0))(u-u_0)^2\,dt}_{L^\infty(\domain)}\\
    &\lesssim  \sup_{u\in A_N}\norm{u-u_0}_{W^{m,\infty}(\domain)}^2\lesssim \sup_{u\in A_N}\norm{u-u_0}_{H^{m+\eta}(\domain)}^2\\
    &\lesssim r_{N,\eta}^2
\end{align*}
for $\eta>d/2$.
%where we used the Sobolev embedding $\norm{}_{L^\infty(\domain)}\lesssim \norm{}_{H^2(\domain)}$ for $d\le 3$ and
%.
%Since $d\le 8$ we have $\norm{}_{L^4(\domain)}\lesssim \norm{}_{H^2(\domain)}$ by the Sobolev embedding theorem. 
Applying \eqref{eq:linearisation_2} and \eqref{eq:Perturbation_Contraction} we find
\begin{equation*}
    \EV[\prob[u_0]{}]{h_u(X_1)^2} \le \bar{p} \norm{h_u}_{L^2(\domain)}^2\lesssim \norm{u-u_0}_{H^m(\domain)}^4\lesssim \epsilon_N^4.
\end{equation*}
We are left to compute the entropy numbers. Using \eqref{eq:linearisation_2}, \eqref{eq:linearisation2} and \eqref{eq:Perturbation_Contraction} we find for $u,v\in A_N$ the bound
\begin{align*}
d_2(u,v) &=\sqrt{\EV[\prob[u_0]{}]{(h_u(X_1) - h_v(X_1))^2}}\\
&\le \sqrt{\bar{p}}\norm{\op[u]-\op[v] - \opDeriv{u_0}[u-v]}_{L^2(\domain)}\\
&\le \sqrt{\bar{p}}\norm{\op[u]-\op[v] - \opDeriv{v}[u-v]}_{L^2(\domain)} + \sqrt{\bar{p}}\norm{(\opDeriv{v}-\opDeriv{u_0})(u-v)}_{L^2(\domain)}\\
&\lesssim \norm{u-v}_{H^m(\domain)}^2 + \norm{v-u_0}_{H^m(\domain)}\norm{u-v}_{H^m(\domain)}\\
%&= \norm{\tau(u)-\tau(v) - \tau'(u_0)(u-v)}_{L^2(\domain)}\\
%&\le \norm{\tau(u)-\tau(v)-\tau'(u)(u-v)}_{L^2(\domain)} + \norm{(\tau'(u_0) - \tau'(u))(u-v)}_{L^2(\domain)}\\
%&= \norm*{\int_0^1 \tau''(v + t(u-v))(u-v)^2\,dt}_{L^2(\domain)}\\
%&\quad + \norm*{\int_0^1 \tau''(u_0 + t(u-u_0))(u-u_0)(u-v) \, dt}_{L^2(\domain)}\\
%&\lesssim \norm{\tau''}_{L^\infty(\mathbb{R})}\norm{u-v}_{L^4(\domain)}^2 + \norm{\tau''}_{L^{\infty}(\mathbb{R})}\norm{u-v}_{L^4(\domain)} \norm{u-u_0}_{L^4(\domain)}\\
&\lesssim \epsilon_N\norm{u-v}_{H^m(\domain)}.
\end{align*}
Similarly, we find 
\begin{align*}
    d_\infty(u,v) &= \norm{\op[u]-\op[v] - \opDeriv{u_0}[u-v]}_{L^\infty(\domain)}\\
    &\le \norm{\op[u]-\op[v] - \opDeriv{v}[u-v]}_{L^\infty(\domain)} + \norm{(\opDeriv{v}-\opDeriv{u_0})(u-v)}_{L^\infty(\domain)}\\
&\lesssim \norm{u-v}_{W^{m,\infty}(\domain)}^2 + \norm{v-u_0}_{W^{m,\infty}(\domain)}\norm{u-v}_{W^{m,\infty}(\domain)}\\
    %= \norm{\tau(u)-\tau(v) - \tau'(u_0)(u-v)}_{L^\infty(\domain)}\\
%&\le \norm{\tau(u)-\tau(v)-\tau'(u)(u-v)}_{L^\infty(\domain)} + \norm{(\tau'(u_0) - \tau'(u))(u-v)}_{L^\infty(\domain)}\\
%&= \norm*{\int_0^1 \tau''(v + t(u-v))(u-v)^2\,dt}_{L^\infty(\domain)}\\
%&\quad + \norm*{\int_0^1 \tau''(u_0 + t(u-u_0))(u-u_0)(u-v) \, dt}_{L^\infty(\domain)}\\
%&\lesssim \norm{\tau''}_{L^\infty(\mathbb{R})}\norm{u-v}_{L^\infty(\domain)}^2 + \norm{\tau''}_{L^\infty(\mathbb{R})}\norm{u-v}_{L^\infty(\domain)} \norm{u-u_0}_{L^\infty(\domain)}\\
&\lesssim r_{N,\eta}\norm{u-v}_{W^{m,\infty}(\domain)}.
\end{align*}
Consequently, the covering numbers can be bounded as
\begin{align*}
    N(A_N,d_2,\rho) &\le N\left(B\left(0,Cr_{N,\eta},\norm{}_{H^{m+\eta}(\domain)}\right),\norm{}_{H^m(\domain)},\frac{C'\rho}{\epsilon_N}\right),\\
    N(A_N,d_\infty,\rho) &\le N\left(B\left(0,Cr_{N,\eta},\norm{}_{H^{m+\eta}(\domain)}\right),\norm{}_{W^{m,\infty}(\domain)},\frac{C'\rho}{r_{N,\eta}}\right)\\
    &\le N\left(B\left(0,C r_{N,\eta},\norm{}_{H^{m+\eta}(\domain)}\right),\norm{}_{B_{\infty,1}^m(\domain)},\frac{C''\rho}{r_{N,\eta}}\right)
\end{align*}
for some constants $0<C,C',C''<\infty$ and the Besov-space $B_{\infty,1}^m(\domain)$ of Definition 2.5.1/1 of \cite{edmundsFunctionSpacesEntropy1996a}.
Consequently, using the covering numbers provided by Theorem 4.10.3 of \cite{T83} and Theorem 3.3.2 of \cite{edmundsFunctionSpacesEntropy1996a} we find
\begin{align*}
    \int_0^{C\epsilon_N^2}\sqrt{\log(N(A_N,d_2,\rho))}d\rho &\lesssim \int_0^{C\epsilon_N^2} \left(\frac{\epsilon_Nr_{N,\eta}}{\rho}\right)^{d/(2\eta)} d\rho \\
    &\lesssim r_{N,\eta}^{d/(2\eta)}\epsilon_N^{2-d/(2\eta)}\\
     \int_0^{Cr_{N,\eta}^2}\log(N(A_N,d_\infty,\rho))d\rho &\lesssim \int_0^{Cr_{N,\eta}^2} \left(\frac{r_{N,\eta}^2}{\rho}\right)^{d/\eta} d\rho \\
     %&\lesssim\epsilon_N^{2d(\beta-\eta)/(\eta\beta)}\epsilon_N^{2(1 - d/\eta)( (\beta-\eta)/\beta)}\\
     &\lesssim r_{N,\eta}^2
\end{align*}
if $\eta>d$.
An application of Lemma 3.12 of \cite{NW20} therefore yields 
\begin{equation*}
    \prob*{\sup_{u\in A_N} \abs{\mathcal{R}_{N,1}} \ge C''\sqrt{N}\left[r_{N,\eta}^{d/(2\eta)}\epsilon_N^{2-d/(2\eta)}+\epsilon_N^2 \sqrt{\delta} + \frac{r_{N,\eta}^2(1+\delta)}{\sqrt{N}}\right]} \le 2e^{-\delta}
\end{equation*}
for some constant $0<C''<\infty$ and any $\delta>0$. To show $\sup_{u\in A_N}\abs{\mathcal{R}_{N,1}} = \smallo_{\prob[u_0][N]{}}(1)$, we require
\begin{equation*}
    \left(2-\frac{d}{2\beta}\right)\frac{\beta}{2\beta+d}>\frac{1}{2},\quad \frac{2\beta}{2\beta+d}>\frac{1}{2},\quad 2-\frac{2\eta}{\beta}>0,\quad\eta>d.
\end{equation*}
These conditions are satisfied as soon as $\beta>\eta>d$, which is possible since $\beta> 2d>d$.\\
\textbf{$\mathcal{R}_{N,2}$}: First recall that $\norm{\psi_0}_{H^m(\domain)}\lesssim \norm{\psi}_{\mathcal{V}^\ast}\le 1$ from \eqref{eq:InformationIsomorphism} such that the Bernstein estimate \eqref{eq:Bernstein_orthogonalprojection} yields
\begin{align*}
    %\norm*{\op[\ProjTwo_{J}\psi_0/\sqrt{N}]}_{L^\infty(\domain)}&\lesssim
    \norm{\ProjTwo_{J}\psi_0/\sqrt{N}}_{H^{m+\eta}(\domain)}\lesssim N^{\eta/(2\beta+d)-1/2} \norm{\psi_0}_{H^m(\domain)}&\lesssim \epsilon_N^{(2\beta-2\eta+d)/(2\beta)}
    %&\le N^{-\beta(1-d/(2\beta))/(2\beta+d)}=\epsilon_N^{1-d/(2\beta)} \le \epsilon_N^{(\beta-\eta)/\beta}
\end{align*}
for any $\eta\ge 0$.
Define $h_u\coloneqq (\op[u]-\op[u_0])^2$ for $u\in A_N$ , then
\begin{equation*}
    \mathcal{R}_{N,2} =-\frac{\sqrt{N}}{2\sigma^2}\left(\frac{1}{\sqrt{N}}\sum_{i=1}^N (h_u(X_i) - \EV[\prob[u_0]{}]{h_u(X_1)})\right).
\end{equation*}
%Before we proceed, we note that on $A_N$ we have by $1\gtrsim \norm{u}_{H^{2+\kappa}}\gtrsim \norm{\op[u]}_{L^\infty(\domain)}$ if $\kappa>d/2$. By interpolation, this implies
%\begin{align*}
%    \sup_{u\in{A_N}}\norm{\op[u]-\op[u_0]}_{L^\infty(\domain)}&\lesssim \norm{u-u_0}_{H^{2+d/2}(\domain)}\lesssim \norm{u-u_0}_{H^2(\domain)}^{(\kappa-d/2)/\kappa}\norm{u-u_0}_{H^{2+\kappa}(\domain)}\\
%    &\lesssim \epsilon_N^{(\kappa-d/2)/\kappa}.
%\end{align*}
With $\eta>d/2$ we have from \eqref{eq:Stability_Lipschitz_Bounds}, \eqref{eq:Lipschitz_Bounds_infty} and \eqref{eq:interpolation_Lu} that
\begin{align*}
    \sup_{u\in A_N}&\norm{h_u}_{L^\infty(\domain)}\\
    &= \sup_{u\in A_N}\norm{(\op[u]-\op[u_0])^2}_{L^\infty(\domain)}= \sup_{u\in A_N}\norm{\op[u]-\op[u_0]}_{L^\infty(\domain)}^2\\
    &\le \sup_{u\in \mathbb{A}_N,\norm{\psi}_{\mathcal{V}^\ast}\le M}\left(\norm{\op[u]-\op[u_0]}_{L^\infty(\domain)}^2 + \norm*{\op[u]-\op[u-t\ProjTwo_{J}\psi_0/\sqrt{N}]}_{L^\infty(\domain)}^2\right)\\
    &\lesssim \sup_{u\in \mathbb{A}_N,\norm{\psi}_{\mathcal{V}^\ast}\le M}\left(\norm{u-u_0}_{H^{m+\eta}(\domain)}^2+\epsilon_N^{(2\beta-2\eta +d)/\beta}\right)\\
    &\lesssim \sup_{u\in \mathbb{A}_N}\norm{u-u_0}_{H^{m+\eta}(\domain)}^2 + \epsilon_N^{(2\beta-2\eta +d)/\beta}\\
    &\lesssim r_{N,\eta}^2.
\end{align*}
Proceeding similarly, we find
\begin{align*}
    \EV[\prob[u_0]{}]{h_u(X_1)^2} &\le \bar{p}\norm{(\op[u]-\op[u_0])^2}_{L^2(\domain)}^2\lesssim   \norm{\op[u]-\op[u_0]}_{L^2(\domain)}^2 \norm{\op[u]-\op[u_0]}_{L^\infty(\domain)}^2\\
    &\lesssim \norm{u-u_0}_{H^m(\domain)}^2 r_{N,\eta}^2\lesssim  \epsilon_N^2r_{N,\eta}^2.
\end{align*}
For the distance $d_2$ we find
\begin{align*}
    d_2(u,v) &=\sqrt{\EV[\prob[u_0]{}]{(h_u(X_1) - h_v(X_1))^2}}\\
    &\le \sqrt{\bar{p}} \norm*{(\op[u]-\op[u_0])^2 - (\op[v]-\op[u_0])^2}_{L^2(\domain)}\\
    &=\sqrt{\bar{p}}\norm*{(\op[u]-\op[v])(\op[u]+\op[v]-2\op[u_0])}_{L^2(\domain)}\\
    &\le \sqrt{\bar{p}}\norm*{\op[u]-\op[v]}_{L^2(\domain)}\norm*{\op[u]+\op[v]-2\op[u_0]}_{L^\infty(\domain)}\\
    &\lesssim \norm*{u-v}_{H^m(\domain)}\left(\norm*{u-u_0}_{H^{m+\eta}(\domain)} + \norm*{v-u_0}_{H^{m+\eta}(\domain)}\right)\\
    &\lesssim r_{N,\eta}\norm{u-v}_{H^m(\domain)}.
\end{align*}
    Similarly,
\begin{align*}
    d_\infty(u,v) &= \norm*{(\op[u]-\op[u_0])^2 - (\op[v]-\op[u_0])^2}_{L^\infty(\domain)}\\
    &= \norm*{(\op[u]-\op[v])(\op[u]+\op[v]-2\op[u_0])}_{L^\infty(\domain)}\\
    &\le \norm*{\op[u]-\op[v]}_{L^\infty(\domain)}\norm*{\op[u]+\op[v]-2\op[u_0]}_{L^\infty(\domain)}\\
    &\lesssim \norm*{u-v}_{W^{m,\infty}(\domain)}\left(\norm*{u-u_0}_{H^{m+\eta}(\domain)} + \norm*{v-u_0}_{H^{m+\eta}(\domain)}\right)\\
    &\lesssim r_{N,\eta}\norm{u-v}_{W^{m,\infty}(\domain)}.
\end{align*}
%Since by interpolation for any $d/2<\eta<\kappa$ we have $\norm{u-u_0}_{H^{2+\eta}(\domain)}\le \norm{u-u_0}_{H^2(\domain)}^{(\kappa-\eta)/\kappa}\norm{u-u_0}_{H^{2+\kappa}(\domain)}^{\eta/\beta}\lesssim \epsilon_N^{(\kappa-\eta)/\kappa}$ we find that $\mathbb{A}_N\subset \{u\in H^{2+\kappa}(\domain)\colon \norm{u-u_0}_{H^{2+\eta}(\domain)}\le \epsilon_N^{(\kappa-\eta)/\kappa}\}$.
 The entropy numbers can therefore be bounded as
\begin{align*}
    N(A_N,d_2,\rho) &\le N\left(B\left(0,Cr_{N,\eta},\norm{}_{H^{m+\eta}(\domain)}\right),\norm{}_{H^m(\domain)},\frac{C'\rho}{r_{N,\eta}}\right),\\
    N(A_N,d_\infty,\rho) &\le N\left(B\left(0,Cr_{N,\eta},\norm{}_{H^{m+\eta}(\domain)}\right),\norm{}_{W^{m,\infty}(\domain)},\frac{C'\rho}{r_{N,\eta}}\right)\\
    &\le N\left(B\left(0,C r_{N,\eta},\norm{}_{H^{m+\eta}(\domain)}\right),\norm{}_{B_{\infty,1}^m(\domain)},\frac{C''\rho}{r_{N,\eta}}\right)
\end{align*}
for some constants $0<C,C',C''<\infty$ and the Besov-space $B_{\infty,1}^m(\domain)$ of Definition 2.5.1/1 of \cite{edmundsFunctionSpacesEntropy1996a}.
Consequently, using the covering numbers provided by Theorem 4.10.3 of \cite{T83} and by Theorem 3.3.2 of \cite{edmundsFunctionSpacesEntropy1996a} we find
\begin{align*}
    \int_0^{C\epsilon_Nr_{N,\eta}}\sqrt{\log(N(A_N,d_2,\rho))}d\rho &\lesssim \int_0^{C\epsilon_Nr_{N,\eta}} \left(\frac{r_{N,\eta}^2}{\rho}\right)^{d/(2\eta)} d\rho
    %& \lesssim \epsilon_N^{d(\beta-\eta)/(\eta\beta)}\epsilon_N^{(1 - d/(2\eta))(1+ (\beta-\eta)/\beta)}\\
    %&= \epsilon_N^{2-(2\eta+d)/(2\beta)},\\
    \lesssim r_{N,\eta}^{1+d/(2\eta)}\epsilon_N^{1-d/(2\eta)}\\
     \int_0^{Cr_{N,\eta}^2}\log(N(A_N,d_\infty,\rho))d\rho &\lesssim \int_0^{Cr_{N,\eta}^2} \left(\frac{r_{N,\eta}^2}{\rho}\right)^{d/\eta} d\rho 
     %&\lesssim\epsilon_N^{2d(\beta-\eta)/(\eta\beta)}\epsilon_N^{2(1 - d/\eta)( (\beta-\eta)/\beta)}\\
     %&=\epsilon_N^{2 - 2\eta/\beta}.
     \lesssim r_{N,\eta}^2
\end{align*}
for $\eta>d$.
An application of Lemma 3.12 of \cite{NW20} therefore yields
\begin{equation*}
    \prob*{\sup_{u\in A_N} \abs{\mathcal{R}_{N,2}} \ge C'''\sqrt{N}\left[r_{N,\eta}^{1+d/(2\eta)}\epsilon_N^{1-d/(2\eta)}+\epsilon_Nr_{N,\eta}\sqrt{\delta} + \frac{r_{N,\eta}^2(1+\delta)}{\sqrt{N}}\right]} \le 2e^{-\delta}
\end{equation*}
for some constant $0<C'''<\infty$.
To show that $\sup_{u\in A_N}\abs{\mathcal{R}_{N,2}}=\smallo_{\prob[u_0][N]{}}(1)$, we thus require the conditions
\begin{equation*}
    \left(2- \frac{2\eta+d}{2\beta}\right)\frac{\beta}{2\beta+d}>\frac{1}{2},\quad \left(1+\frac{\beta-\eta}{\beta}\right)\frac{\beta}{2\beta+d}>\frac{1}{2},2 - \frac{2\eta}{\beta}>0,\quad \eta>d,
\end{equation*}
which are satisfied as soon as $\beta>\eta+d$ (possible since $\beta>2d$ by choosing $\eta>d$ sufficiently small).
%$\beta>d$, $d\beta/(2\beta - d)<\beta_{0}\le\beta$, $d/2<\eta <  \beta_{0}(1 - d/(2\beta))$.
\\
For $\mathcal{R}_{N,3}$ we find
\begin{align*}
    \abs{\mathcal{R}_{N,3}} &= \frac{N}{2\sigma^2}\left(\abs{\norm{\opDeriv{u_0}[u-u_0]}_{L_p^2(\domain)}^2-\norm{\op[u]-\op[u_0]}_{L_p^2(\domain)}^2}\right)\\
    &=\frac{N}{2\sigma^2}\abs{\iprod{\opDeriv{u_0}[u-u_0]-(\op[u]-\op[u_0])}{\opDeriv{u_0}[u-u_0]+(\op[u]-\op[u_0])}_{L_p^2(\domain)}}\\
    &\le \frac{N\bar{p}}{2\sigma^2}\norm{\opDeriv{u_0}[u-u_0]-(\op[u]-\op[u_0])}_{L^2(\domain)}\\
    &\quad\times\left(\norm*{\opDeriv{u_0}[u-u_0]}_{L^2(\domain)}+\norm*{\op[u]-\op[u_0]}_{L^2(\domain)}\right)\\
    %&\le \bar{p} \frac{N}{2\sigma^2}\left(\abs{\norm*{\opDeriv{u_0}[u-u_0]}_{L^2(\domain)}-\norm*{\op[u]-\op[u_0]}_{L^2(\domain)}}\right)\\
    %&\quad\times\left(\norm*{\opDeriv{u_0}[u-u_0]}_{L^2(\domain)}+\norm*{\op[u]-\op[u_0]}_{L^2(\domain)}\right)\\
    %&\le \bar{p}\frac{N}{2\sigma^2}\norm*{\op[u]-\op[u_0]-\opDeriv{u_0}[u-u_0]}_{L^2(\domain)}\\
    %&\quad \times\left(\norm*{\opDeriv{u_0}[u-u_0]}_{L^2(\domain)}+\norm*{\op[u]-\op[u_0]}_{L^2(\domain)}\right)\\
    &\lesssim N\norm*{u-u_0}_{H^m(\domain)}^2\norm{u-u_0}_{H^m(\domain)}\\
    &\lesssim N\epsilon_N^3 
\end{align*}
for $u\in \mathbb{A}_N$ by \eqref{eq:Stability_Lipschitz_Bounds}, \eqref{eq:linearisation_2} and \eqref{eq:interpolation_u} with $\eta=0$. %and the embedding of $H^2(\domain)$ into $L^4(\domain)$.
To ensure that this is $\smallo(1)$ we require that $3\beta /(2\beta + d)>1$, i.e. $\beta>d$.

\end{proof}

\begin{lem}[Change of measure]\label{lem:Change_of_measures}
    Grant Assumptions \ref{assump:ApproximatonSets}, \ref{assump:Operator_stability} and \ref{assump:operator_linearisation}. For $\beta>2d$ and $d/2<\beta_0<\beta+d/4$ consider the prior $\pi_N$ from \eqref{eq:Prior} with cut-off $2^J\sim N^{1/(2\beta+d)}$. Let $M>0$ and $t\in\mathbb{R}$, then
    \begin{equation*}
    \sup_{\psi\colon \norm{\psi}_{\mathcal{V}^\ast}\le M}\left|\frac{\int_{\mathbb{A}_N}e^{\ell_N(u_t)}\, \D\pi_N(u)}{\int_{\mathbb{A}_N} e^{\ell_N(u)}\, \D\pi_N(u)}-1\right|=\smallo_{\prob[u_0][N]{}}(1).
\end{equation*}
\end{lem}
\begin{proof}
By the Cameron--Martin Theorem (Proposition 2.26 of \cite{DaPrato2014}) we have 
\begin{equation*}
    \int_{\mathbb{A}_N}e^{\ell_N(u_t)}\D\pi_N(u) = \int_{\mathbb{A}_N-t\ProjTwo_{J}\psi_0/\sqrt{N}} e^{\ell_N(g)}e^{-\iprod{t\ProjTwo_{J}\psi_0/\sqrt{N}}{g}_{\mathbb{H}_N} -\norm{t\ProjTwo_{J}\psi_0/\sqrt{N}}_{\mathbb{H}_N}^2/2}\D\pi_N(g).
\end{equation*}
Using the representation of the Cameron--Martin norm \eqref{eq:RKHS_Norm}, the Bernstein estimate \eqref{eq:Bernstein_orthogonalprojection} and \eqref{eq:InformationIsomorphism} we find
\begin{align*}
    \norm{\ProjTwo_{J}\psi_0}_{\mathbb{H}_N} &\lesssim \norm{\ProjTwo_{J}\psi_0}_{H^{m+\beta_0+\epsilon}(\domain)}\lesssim N^{(\beta_0+\epsilon)/(2\beta+d)}\norm{\psi_0}_{H^m(\domain)}\\
    %&\le \begin{cases}
        %\norm{\psi_0}_{H^{4+\tilde{\gamma}}(\domain)}\lesssim \norm{\psi}_{H^{\tilde{\gamma}}(\domain)}, & \beta_0\le 0,\\
        &\lesssim N^{(\beta_0+\epsilon)/(2\beta+d)}\norm{\psi}_{\mathcal{V}^\ast}= \smallo\left(N^{1/2}\right),
\end{align*}
provided $\epsilon+\beta_0 <\beta+d/2$. Similarly, for $g\in \mathbb{A}_N$ we have from Lemma \ref{lem:Prior} \eqref{num:aux_rescaled_Prior_restrict} that 
\begin{equation*}
\norm{g}_{\mathbb{H}_N}\le \norm{g}_{H^{m+\beta_0+\epsilon}(\domain)}\lesssim 2^{J(\beta_0+\epsilon-\beta)_+}\norm{g}_{H^{m+\beta}(\domain)}\lesssim 2^{J(\beta_0+\epsilon-\beta)_+}\log(N)
\end{equation*}
 for $0<\epsilon<\beta+d/4-\beta_0$, such that
\begin{align*}
    \abs{\iprod{t\ProjTwo_{J}\psi_0/\sqrt{N}}{g}_{\mathbb{H}_N}}&\le \frac{\abs{t}}{\sqrt{N}}\norm{\ProjTwo_{J}\psi_0}_{\mathbb{H}_N}\norm{g}_{\mathbb{H}_N}\\
    &=\mathcal{O}(N^{-1/2}2^{J(\beta_0+\epsilon +(\beta_0+\epsilon-\beta)_+)}\log(N))=\smallo(1)
\end{align*}
whenever $\beta_0+\epsilon<\beta+d/4$. We deduce that
\begin{equation*}
    \sup_{g\in \mathbb{A}_N, \norm{\psi}_{\mathcal{V}^\ast}\le M} \left|e^{-\iprod{t\ProjTwo_{J}\psi_0/\sqrt{N}}{g}_{\mathbb{H}_N} -\norm{t\ProjTwo_{J}\psi_0/\sqrt{N}}_{\mathbb{H}_N}^2/2}-1\right|=\smallo(1).
\end{equation*}
Furthermore,
\begin{align*}
    \left|\frac{\int_{\mathbb{A}_N}e^{\ell_N(u_t)}\, \D\pi_N(u)}{\int_{\mathbb{A}_N} e^{\ell_N(u)}\, \D\pi_N(u)} - 1\right| &= \left|\frac{\int_{\mathbb{A}_N-t\ProjTwo_{J}\psi_0/\sqrt{N}}e^{\ell_N(u)}\, \D\pi_N(u)}{\int_{\mathbb{A}_N} e^{\ell_N(u)}\, \D\pi_N(u)}-1\right|+\smallo(1)\\
    &=  \left|\frac{\Pi_N\left(\mathbb{A}_N - t\ProjTwo_{J}\psi_0/\sqrt{N} | \data \right)}{\Pi_N\left(\mathbb{A}_N| \data \right)}-1\right|+ \smallo(1).
\end{align*}
Since $\Pi_N(\mathbb{A}_N|\data)\xrightarrow{\prob[u_0][N]{}}(1)$ by Theorem \ref{thm:Contraction}, it remains to show that, for fixed $t\in\mathbb{R}$ and uniformly over $\psi\in\mathcal{V}^\ast$ with $\norm{\psi}_{\mathcal{V}^\ast}\le M$, we have $\Pi_N(\mathbb{A}_N-t\ProjTwo_J\psi_0/\sqrt{N}|\data)\xrightarrow{\prob[u_0][N]{}}(1)$ as $N\to\infty$.
Note that we have
\begin{equation}
    \sup_{u\in A_N}\norm{u-u_0}_{H^m(\domain)}\le \sup_{v\in \mathbb{A}_N}\norm{v-u_0}_{H^m(\domain)}  + \sup_{\psi\colon \norm{\psi}_{\mathcal{V}^\ast} \le  M}\frac{1}{\sqrt{N}}\norm{\psi_0}_{H^m(\domain)}\lesssim \epsilon_N,\label{eq:Perturbation_Contraction}
\end{equation}
by Lemma \ref{lem:Properties_LAN_Norm} \eqref{num:ProjectionBoundedness} and \eqref{eq:InformationIsomorphism}. Furthermore, by Theorem \ref{thm:Contraction} and Lemma \ref{lem:Prior}, we have $\Pi_N(\mathbb{A}_N|\data)\xrightarrow{\prob[u_0][N]{}}(1)$. Since 
\begin{equation*}
    \sup_{\psi\colon\norm{\psi}_{\mathcal{V}^\ast}\le M}\frac{\norm{\ProjTwo_J\psi_0}_{H^m(\domain)}}{\sqrt{N}}\lesssim\sup_{\psi\colon\norm{\psi}_{\mathcal{V}^\ast}\le M} \frac{\norm{\psi_0}_{\mathcal{V}}}{\sqrt{N}}=\smallo(\epsilon_N)
\end{equation*}
as $N\to\infty$ by \eqref{eq:InformationIsomorphism} and
\begin{equation*}
    \sup_{\psi\colon\norm{\psi}_{\mathcal{V}^\ast}\le M}\frac{\norm{\ProjTwo_J \psi_0}_{H^{m+\beta}(\domain)}}{\sqrt{N}} \le \sup_{\psi\colon\norm{\psi}_{\mathcal{V}^\ast}\le M}\frac{C2^{J\beta}\norm{\ProjTwo_J \psi_0}_{H^m(\domain)}}{\sqrt{N}} =\smallo(\log(N)),
\end{equation*}
and by choosing $L+1$ instead of $L$, we find 
\begin{equation*}
\begin{split}
    \mathbb{A}_N &- t\ProjTwo_J\psi_0/\sqrt{N}\\
    &\supset \{ u\in V_J \,\colon\, \norm{u-u_0}_{H^m(\domain)}\le (L-1)\epsilon_N \text{ and }\norm{u}_{H^{m+\beta}(\domain)}\le (K-1)\log(N)\}
\end{split}
\end{equation*}
for any $\psi\in\mathcal{V}^\ast$ with $\norm{\psi}_{\mathcal{V}^\ast}\le M$ for $N$ sufficiently large. Consequently,
\begin{align*}
    &\inf_{\psi\colon \norm{\psi}_{\mathcal{V}^\ast}\le M}\Pi_N(\mathbb{A}_N - t\ProjTwo_J\psi_0/\sqrt{N}|\data) \\
    &~~\ge \Pi_N\Big( u\in V_J \,\colon\, \norm{u-u_0}_{H^m(\domain)}\le (L-1)\epsilon_N\text{ and }\norm{u}_{H^{m+\beta}(\domain)}\le (K-1)\log(N)\Big)\\
    &~~\xrightarrow{\prob[u_0][N]{}}1
\end{align*}
 as $N\to\infty$ by Theorem \ref{thm:Contraction} choosing $L$ and $K$ large enough.
\end{proof}

Define 
\begin{equation}
    \tilde{u}_N = \ProjTwo_J u_0 + \frac{1}{\sigma N}\sum_{\abs{\mu},\abs{\mu'}\le J} (G_J^{-1})_{\mu,\mu'} \sum_{i=1}^N \epsilon_i\opDeriv{u_0}[\psi_\mu](X_i) \psi_{\mu'}=\colon \ProjTwo_J u_0 + Z_N\in V_J,\label{eq:EfficientEstimator_RandomDesign}
\end{equation}
where $G_J= (G_{\mu,\mu'})_{\mu,\mu'}$ with $G_{\mu,\mu'}=\iprod{\psi_\mu}{\psi_{\mu'}}_{\operatorname{LAN}}$ is the Gram matrix, which is invertible by \eqref{eq:Graphnormequivalence}.
Note that by \eqref{eq:FromDualtoLAN} we have for any $\psi\in \mathcal{V}^\ast$ the identity
\begin{align*}
    \iprod[\mathcal{V}^\ast][\mathcal{V}]{\psi}{\tilde{u}_N} &= \iprod[\mathcal{V}^\ast][\mathcal{V}]{\psi}{\ProjTwo_J u_0} + \frac{1}{\sigma N}\sum_{i=1}^N \epsilon_i \opDeriv{u_0}[\ProjTwo_{J}\psi_0](X_i)\\
    &=\iprod[\mathcal{V}^\ast][\mathcal{V}]{\psi}{u_0} - \iprod{u_0}{(\identity - \ProjTwo_{J})\psi_0}_{\operatorname{LAN}} + \frac{1}{\sigma N}\sum_{i=1}^N \epsilon_i \opDeriv{u_0}[\ProjTwo_{J}\psi_0](X_i),
\end{align*}
since for $c_{\psi,\mu}\coloneqq \iprod[\mathcal{V}^\ast][\mathcal{V}]{\psi}{\psi_\mu}$, the coefficient vector of $\ProjTwo_J \psi_0$ is given by $G_J^{-1} c_\psi$.

%\SW{Can one somehow write this in a more intuitive notation? Like $ \frac{1}{\sqrt N}\langle \eta, \psi_k\rangle_N \lambda_k^{-1}\psi_k$ or so...}\SG{It's the inverse of the Fisher information applied to the score (projected on V_J)

\begin{prop}\label{prop:Uniform_Laplace_Transform_Control}
    Grant Assumptions \ref{assump:ApproximatonSets}, \ref{assump:Operator_stability} and \ref{assump:operator_linearisation}. For $\beta>2d$ and $d/2<\beta_0<\beta+d/4$ consider $u_0\in \mathcal{V}\cap H^{m+\beta}(\domain)$ and the prior $\pi_N$ from \eqref{eq:Prior} with cut-off $2^J\sim N^{1/(2\beta+d)}$. Then, for any  $\psi\in \mathcal{V}^\ast$ and $t\in\mathbb{R}$, we have
    \begin{equation}
        \EV*[\Pi_N^{\mathbb{A}_N}]{\exp\left(t\sqrt{N}\iprod[\mathcal{V}^\ast][\mathcal{V}]{\psi}{u-\tilde{u}_N}\right)\given \data} = \exp\left(\frac{t^2}{2}\norm{\ProjTwo_J\psi_0}_{\operatorname{LAN}}^2\right)(1+R_{N,t}(\psi)),\label{eq:claim_LaplaceConvergence}
    \end{equation}
    where for any fixed $M>0$ and $t\in\mathbb{R}$
    \begin{equation*}
        \sup_{\norm{\psi}_{\mathcal{V}^\ast}\le M}\abs{R_{N,t}(\psi)}=\smallo_{\prob[u_0][N]{}}(1)
    \end{equation*}
    as $N\to\infty$. Moreover, the following properties hold:
    \begin{enumerate}[(i)]
        \item For any $\psi\in\mathcal{V}^\ast$ we have
        \begin{equation*}
            \EV*[\Pi_N^{\mathbb{A}_N}]{\exp\left(t\sqrt{N}\iprod[\mathcal{V}^\ast][\mathcal{V}]{\psi}{u-\tilde{u}_N}\right)\given \data}\xrightarrow{\prob[u_0][N]{}} \exp\left(\frac{t^2}{2}\norm{\psi_0}_{\operatorname{LAN}}^2\right)
        \end{equation*}
        as $N\to\infty$.
        \item For any $M>0$ the uniform bound
        \begin{equation*}
            \sup_{\norm{\psi}_{\mathcal{V}^\ast}\le M}\EV*[\Pi_N^{\mathbb{A}_N}]{\exp\left(t\sqrt{N}\iprod[\mathcal{V}^\ast][\mathcal{V}]{\psi}{u-\tilde{u}_N}\right)\given \data} =\mathcal{O}_{\prob[u_0][N]{}}(1)
        \end{equation*}
        as $N\to\infty$ holds.
    \end{enumerate}
\end{prop}
\begin{proof}
Note that we have 
\begin{align*}
    \EV*[\Pi_N^{\mathbb{A}_N}]{\exp\left(t\sqrt{N}\iprod[\mathcal{V}^\ast][\mathcal{V}]{\psi}{u-\tilde{u}_N}\right)\indicator_{\mathbb{A}_N}\given \data} \hspace{-10em}&\\
    &=\frac{\int_{\mathbb{A}_N}e^{t\sqrt{N}\iprod[\mathcal{V}^\ast][\mathcal{V}]{\psi}{u-\tilde{u}_N}}e^{\ell_N(u)-\ell_N(u_t)}e^{\ell_N(u_t)}\, \D\pi_N(u)}{\int_{\mathbb{A}_N} e^{\ell_N(u)}\, \D\pi_N(u)}.
\end{align*}
    %Since $\Pi_N(\mathbb{A}_N^c|\data) = \smallo_{\prob[u_0][N]{}}(1)$ by Theorem \ref{thm:Contraction} and Lemma \ref{lem:Prior} \eqref{num:aux_rescaled_Prior_restrict}, it suffices to show that with $u_t= u - t\ProjTwo_{J}\psi_0/\sqrt{N}$ the ratio
    %\begin{equation*}
    %    \frac{\int_{\mathbb{A}_N}e^{\frac{t\sqrt{N}}{\sigma}(\iprod[\mathcal{V}][\mathcal{V}^\ast]{u}{\psi} - \tilde{u}_N)}e^{\ell_N(u)-\ell_N(u_t)}e^{\ell_N(u_t)}\, \D\pi_N(u)}{\int_{\mathbb{A}_N} e^{\ell_N(u)}\, \D\pi_N(u)}%\hspace{-15em} &\\
        %&= \frac{\int_{\mathbb{A}_N}e^{\frac{t\sqrt{N}}{\sigma}(\iprod[\mathcal{V}][\mathcal{V}^\ast]{u}{\psi} - \tilde{u}_N)}e^{\ell_N(u)-\ell_N(u_t)}e^{\ell_N(u_t)}\, \D\pi_N(u)}{\int e^{\ell_N(u)}\, \D\pi_N(u)}\Pi_N(\mathbb{A}_N|\data)^{-1}
    %\end{equation*}
    %converges in $\prob[u_0]{}$ to $\exp(t^2\norm{\psi_0}_{\operatorname{LAN}}^2/2)$ as $N\to\infty$, see also Lemma 1 in the supplement of \cite{castilloBernsteinMisesTheorem2015} for convergence in probability of the Laplace transformation.
    By Proposition \ref{prop:uniformLANExpansion_RandomDesign} we have the expansion
    \begin{align*}
        \ell_N(u)-\ell_N(u_t) &= - \frac{N}{2}\norm{u-u_0}_{\operatorname{LAN}}^2 + \frac{N}{2}\norm{u_t-u_0}_{\operatorname{LAN}}^2 + \frac{1}{\sigma}\sum_{i=1}^N \epsilon_i\opDeriv{u_0}[u-u_t](X_i)\\
        &\quad +\smallo_{\prob[u_0][N]{}}(1)\\
        &= - t\sqrt{N}\iprod{\ProjTwo_{J}\psi_0}{u-u_0}_{\operatorname{LAN}}  +\frac{t^2}{2}\norm{\ProjTwo_{J}\psi_0}_{\operatorname{LAN}}^2\\
        &\quad + \frac{t}{\sigma\sqrt{N}}\sum_{i=1}^N\epsilon_i \opDeriv{u_0}[\ProjTwo_{J}\psi_0](X_i) +\smallo_{\prob[u_0][N]{}}(1).
    \end{align*}
    Note that the remainder estimate is uniform over $u\in \mathbb{A}_N$ and $\norm{\psi}_{\mathcal{V}^\ast}\le M$ by Proposition \ref{prop:uniformLANExpansion_RandomDesign}.
    Recalling that the expression \eqref{eq:EfficientEstimator_RandomDesign} for $\tilde{u}_N$ and using \eqref{eq:FromDualtoLAN}, we find
    %\begin{equation*}
    %    \tilde{u}_N = \iprod[\mathcal{V}][\mathcal{V}^\ast]{u_0}{\psi} - \iprod{u_0}{(\identity - \ProjTwo_{J})\psi_0}_{\operatorname{LAN}} + \frac{1}{\sigma N}\sum_{i=1}^N \epsilon_i \opDeriv{u_0}[\ProjTwo_{J}\psi_0](X_i),
    %\end{equation*}
    \begin{align*}
        t\sqrt{N}\iprod[\mathcal{V}^\ast][\mathcal{V}]{\psi}{u-\tilde{u}_N} + \ell_N(u)-\ell_N(u_t) \hspace{-15em}&\\
        &= \frac{t^2}{2}\norm{\ProjTwo_{J}\psi_0}_{\operatorname{LAN}}^2 - t\sqrt{N}\iprod{\ProjTwo_{J}\psi_0}{u-u_0}_{\operatorname{LAN}} \\
        &\quad +t\sqrt{N}\iprod{u}{\psi_0}_{\operatorname{LAN}}-t\sqrt{N}\iprod{u_0}{\psi_0}_{\operatorname{LAN}} + t\sqrt{N}\iprod{u_0}{(\identity - \ProjTwo_{J})\psi_0}_{\operatorname{LAN}}+ \smallo_{\prob[u_0][N]{}}(1)\\
        &= \frac{t^2}{2}\norm{\ProjTwo_{J}\psi_0}_{\operatorname{LAN}}^2 -t\sqrt{N}\iprod{\ProjTwo_{J}\psi_0}{u-u_0}_{\operatorname{LAN}} + t\sqrt{N}\iprod{u-u_0}{\psi_0}_{\operatorname{LAN}}\\
        &\quad + t\sqrt{N}\iprod{u_0}{(\identity - \ProjTwo_{J})\psi_0}_{\operatorname{LAN}}  + \smallo_{\prob[u_0][N]{}}(1)\\
        &= \frac{t^2}{2}\norm{\ProjTwo_{J}\psi_0}_{\operatorname{LAN}}^2 + t\sqrt{N}\iprod{u}{(\identity -\ProjTwo_{J})\psi_0}_{\operatorname{LAN}} + \smallo_{\prob[u_0][N]{}}(1).
    \end{align*}
    Since $\ProjTwo_{J}$ is the $\iprod[][\operatorname{LAN}]{}{}$-orthogonal projection on $V_{J}$ and $u\in V_J$, the second summand is zero. Note that the term included in $\smallo_{\prob[u_0][N]{}}(1)$ is understood uniformly over $u\in\mathbb{A}_N$ and $\norm{\psi}_{\mathcal{V}^\ast} \le  M$, such that we have 
    \begin{align*}
         \EV*[\Pi_N^{\mathbb{A}_N}]{\exp\left(t\sqrt{N}\iprod[\mathcal{V}^\ast][\mathcal{V}]{\psi}{u-\tilde{u}_N}\right)\indicator_{\mathbb{A}_N}\given \data}\hspace{-7em}&\\
        &= e^{t^2\norm{\ProjTwo_J\psi_0}_{\operatorname{LAN}}^2/2}(1+\smallo_{\prob[u_0][N]{}}(1))\frac{\int_{\mathbb{A}_N}e^{\ell_N(u_t)}\,\D\pi_N(u)}{\int_{\mathbb{A}_N}e^{\ell_N(u)}\,\D\pi_N(u)}.
    \end{align*}
    %Using $\norm{}_{\operatorname{LAN}}=\norm{\opDeriv{u_0}\MTemptyplaceholder}_{L^2(\domain)}$ and the direct estimate \eqref{eq:directestimate_eigenbasis} we find
    %\begin{align}
    %    \iprod*{u-u_0}{(\identity -\ProjTwo_{J})\psi_0}_{\operatorname{LAN}} &= \iprod*{\opDeriv{u_0}[u-u_0]}{\opDeriv{u_0}(\identity -\ProjTwo_{J})\psi_0}_{L^2(\domain)}\nonumber\\
    %    &= \iprod*{(\identity -\ProjTwo_{J}) \opDeriv{u_0}[u-u_0]}{\opDeriv{u_0}(\identity -\ProjTwo_{J})\psi_0}_{L^2(\domain)}\label{eq:critical}\\
    %    &\le \norm{(\identity -\ProjTwo_{J}) \opDeriv{u_0}[u-u_0]}_{L^{2}(\domain)}\norm{\opDeriv{u_0}(\identity -\ProjTwo_{J})\psi_0}_{L^{2}(\domain)}\nonumber\\
    %    &\lesssim \norm{(\identity -\ProjTwo_{J})(u-u_0)}_{H^{2}(\domain)}\norm{(\identity -\ProjTwo_{J})\psi_0}_{H^{2}(\domain)}\nonumber\\
    %    &\le  N^{-\beta/(2\beta)}\norm{u-u_0}_{H^{2+\beta}(\domain)} N^{-\epsilon/(2\beta)}\norm{\psi_0}_{H^{2+\epsilon}(\domain)}\nonumber\\
        %&=2^{-J\beta_0}\mathcal{O}(1)\smallo(1)\\
    %    &=\smallo(N^{-1/2}).\nonumber
   % \end{align}
    %\begin{align*}
    %   \iprod{u-u_0}{(\identity -\ProjTwo_{J})\psi_0}_{\operatorname{LAN}} &\le \norm{u-u_0}_{H^2(\domain)}\norm{(\identity -\ProjTwo_{J})\psi_0}_{H^{2}(\domain)}\lesssim \epsilon_N 2^{-J(\epsilon+d/2)}\norm{\psi_0}_{H^{2+\epsilon+d/2}(\domain)}\\
     %   &\lesssim \epsilon_N 2^{-J(\epsilon+d/2)}\norm{\psi}_{H^{-2+\epsilon+d/2}(\domain)}\lesssim \epsilon_N 2^{-J(\epsilon+d/2)}\\
    %    & = N^{-(\beta+\epsilon+d/2)/(2\beta + d)}\log(N)= \smallo(N^{-1/2}),
    %\end{align*}
    %since $\epsilon>0$.
    We deduce \eqref{eq:claim_LaplaceConvergence} from Lemma \ref{lem:Change_of_measures}.
    Furthermore, for any fixed $\psi\in\mathcal{V}^\ast$ we have 
    \begin{equation}
        \abs{\norm{\ProjTwo_{J}\psi_0}_{\operatorname{LAN}}- \norm{\psi_0}_{\operatorname{LAN}}} \le \norm{(\identity-\ProjTwo_{J})\psi_0}_{\operatorname{LAN}} =\smallo(1)\label{eq:aux_convergence_projections}
    \end{equation}
    as $N\to\infty$ by Lemma \ref{lem:Properties_LAN_Norm} \eqref{num:ProjectionDensity}.
    %Since the remainder is uniformly small over $u\in \mathbb{A}_N$ and $\norm{\psi}_{H^{2}(\domain)}\le M$,
    We conclude the proof by noting that $\norm{\ProjTwo_J \psi_0}_{\operatorname{LAN}}\le \norm{\psi_0}_{\operatorname{LAN}}\lesssim \norm{\psi}_{\mathcal{V}^\ast}\le M  $ by \eqref{eq:InformationIsomorphism} uniformly for $\norm{\psi}_{\mathcal{V}^\ast}\le M$.
\end{proof}

\subsection{Proofs for the posterior mean centring}

\begin{proof}[Proof of Theorem \ref{thm:BvM_RandomDesign}]
    The proof strategy follows similar steps as the proof of Theorem 2 of \cite{nicklBernsteinvonMisesTheorems2025}. Lemma \ref{lem:Localisationisfine} shows that $\mathcal{W}_{1,H^\gamma(\domain)}((\tau_N)_{\#}\Pi_N,(\tau_N)_{\#}\Pi_N^{\mathbb{A}_N})=\smallo_{\prob[u_0][N]{}}(1)$, such that it suffices to prove
     \begin{equation}
        \mathcal{W}_{1,H^{\gamma}(\domain)}\left((\tau_N)_{\#}\Pi_N^{\mathbb{A}_N}, \gaussian\right) = \smallo_{\prob[u_0][N]{}}(1)\label{eq:claim_BvM_localised}
     \end{equation}
     as $N\to\infty$.
     %Let $(\psi_j)_{j\in\mathbb{N}}$ be the eigenfunctions of $-\opDeriv{u_0}$ with corresponding eigenvalues $\lambda_j\sim  j^{2/d}$.
     %The corresponding orthogonal projection onto the first $K$ eigenfunctions is denoted by $\ProjTwo_{J}$.\SG{choose $K=J$?}
     We first prove \eqref{eq:claim_BvM_localised} for a different recentering than the posterior mean, namely $\tilde{u}_N$ from \eqref{eq:EfficientEstimator_RandomDesign} understood as an element in $\mathcal{V}$. By \eqref{eq:EfficientEstimator_RandomDesign} it is evident that $\tilde{u}_N\in V_J\subset \mathcal{V}$.
For the $\norm{}_{H^\gamma(\domain)}$-norm we even find by \eqref{eq:CharacterisationSobolevSpaces} and \eqref{eq:InformationIsomorphism} that
     \begin{align}
     \begin{split}
         \EV[\prob[u_0][N]{}][][\Big]{\norm{ Z_N}_{H^{\gamma}(\domain)}^2}&\lesssim \sum_{\abs{\mu}\le J} 2^{2\gamma\abs{\mu}}\EV[\prob[u_0][N]{}]{\iprod{Z_N}{\psi_\mu}_{L^2(\domain)}^2}\\
         &=\frac{1}{N^2}\sum_{\abs{\mu}\le J} 2^{2\gamma\abs{\mu}}\EV[\prob[u_0][N]{}]{\Big(\sum_{i=1}^N\epsilon_i \opDeriv{u_0}[\ProjTwo_J\mathcal{I}_{u_0}^{-1}\psi_\mu]\Big)^2}\\
         %&\lesssim \frac{1}{\sigma^2N^2}\EV[][][\Big]{\norm{\sum_{\mu\in\Lambda} \sum_{i=1}^N \epsilon_i\opDeriv{u_0}[\ProjTwo_{J}\mathcal{I}_{u_0}^{-1}\psi_\mu](X_i) \psi_\mu](X_i)\psi_\mu}_{H^{\gamma}(\domain)}^2}\\
         %&= \frac{1}{\sigma^2 N^2}\sum_{\mu\in\Lambda}2^{2\abs{\mu}\gamma} \EV[][][\Big]{\Big(\sum_{i=1}^N \epsilon_i \opDeriv{u_0}[\ProjTwo_{J} \mathcal{I}_{u_0}^{-1}\psi_\mu](X_i)\Big)^2}\\
         &\lesssim \frac{1}{N}\sum_{\abs{\mu}\le J} 2^{2\abs{\mu}\gamma}\norm{\ProjTwo_{J} \mathcal{I}_{u_0}^{-1}\psi_\mu}_{\operatorname{LAN}}^2\le \frac{1}{N}\sum_{\mu\in\Lambda} 2^{2\abs{\mu}\gamma}\norm{\mathcal{I}_{u_0}^{-1}\psi_\mu}_{\operatorname{LAN}}^2\\
         &\lesssim \frac{1}{N}\sum_{\abs{\mu}\le J} 2^{2\abs{\mu}\gamma}\norm{\mathcal{I}_{u_0}^{-1}\psi_\mu}_{H^m(\domain)}^2\lesssim \frac{1}{N}\sum_{\mu\in\Lambda} 2^{2\abs{\mu}\gamma}\norm{\psi_\mu}_{\mathcal{V}^\ast }^2\\
        &\le \frac{1}{N}\sum_{\abs{\mu}\le J} 2^{2\abs{\mu}\gamma}\norm{\psi_\mu}_{H^{-m}(\domain)}^2\lesssim \frac{1}{N}\sum_{\mu\in\Lambda} 2^{2\abs{\mu}(\gamma-m+\epsilon)}\\
        &\lesssim\frac{1}{N}\sum_{j=1}^J  2^{2j(\gamma-m+\epsilon+d/2)}=\mathcal{O}(N^{-1})
         \end{split}\label{eq:stochastic_remainder}
     \end{align}
     for $0<\epsilon<m-\gamma-d/2$ small enough.

     For a cut-off $K\in\mathbb{N}$ and $u\in\mathbb{A}_N$ define the transformations $\tilde{\tau}_N(u)\coloneqq \sqrt{N}(u-\tilde{u}_N)$ and $\tilde{\tau}_N^K(u)\coloneqq \sqrt{N}\ProjThree_{K} (u-\tilde{u}_N)$.
The triangle inequality gives
\begin{align}
    \begin{split}
    \mathcal{W}_{1,H^{\gamma}(\domain)}((\tilde{\tau}_N)_{\#}\Pi_N^{\mathbb{A}_N}, \gaussian)&\le \mathcal{W}_{1,H^{\gamma}(\domain)}((\tilde{\tau}_N)_{\#}\Pi_N^{\mathbb{A}_N},(\tilde{\tau}_N^K)_{\#}\Pi_N^{\mathbb{A}_N})\\
    &\quad +  \mathcal{W}_{1,H^{\gamma}(\domain)}((\tilde{\tau}_N^K)_{\#}\Pi_N^{\mathbb{A}_N},\gaussian^{K}) + \mathcal{W}_{1,H^{\gamma}(\domain)}(\gaussian^{K},\gaussian)\\
    &=\colon A+B+C,
    \end{split}\label{eq:decompBvM}
\end{align}
where $\gaussian^{K}=(\ProjThree_K)_{\#}\gaussian$ is the push-forward of $\gaussian$ under the projection $\ProjThree_{K}$ on $V_{K}$.
We treat the three terms separately.
For $A$ we proceed as follows:
\begin{align*}
    A^2 &= \sup_{F\colon \norm{F}_{\operatorname{Lip}(H^{\gamma}(\domain),\mathbb{R})}\le 1} \abs{\EV[\Pi_N^{\mathbb{A}_N}]{F(\sqrt{N}\ProjThree_{K}( u-\tilde{u}_N)) - F(\sqrt{N}(u-\tilde{u}_N))\given\data}}^2\\
    &\le \EV[\Pi_N^{\mathbb{A}_N}]{\norm{\sqrt{N}(\identity-\ProjThree_{K})(u-\tilde{u}_N)}_{H^{\gamma}(\domain)}\given\data}^2.
\end{align*}
By the characterisation of the Sobolev norms using the wavelet basis from \eqref{eq:CharacterisationSobolevSpaces} we find 
\begin{align}
    \begin{split}
    A^2&\lesssim \EV[\Pi_N^{\mathbb{A}_N}][][\Big]{\sum_{\abs{\mu}>K} 2^{2\abs{\mu}\gamma}\iprod{\sqrt{N}(u-\tilde{u}_N)}{\psi_\mu}_{L^2(\domain)}^2\given\data}\\
    %&= \sqrt{N}\EV[\Pi_N^{\mathbb{A}_N}]{\sup_{\norm{\psi}_{H^\gamma(\domain)}\le 1}\iprod{\tilde{u}_N_{K} - \tilde{u}_N}{\psi}_{L^2(\domain)}\given\data}\\
    %&\lesssim \EV[u][\mathbb{A}_N]{\sup_{\norm{\psi}_{H^\gamma(\domain)}=1}\sqrt{N}\iprod{u-\tilde{u}_N}{\psi-\ProjOne_K\psi}_{L^2(\domain)}\given \data}\\
    %&= \EV[u][\mathbb{A}_N]{\sqrt{N}\norm{\ProjOne_K^\perp[u-\tilde{u}_N]}_{H^\gamma(\domain)}\given \data}\\
    %&\le \EV[u][\mathbb{A}_N]{\norm{\ProjOne_K^\perp[\sqrt{N}(u-\tilde{u}_N)]}_{H^\gamma(\domain)}^2\given\data}^{1/2}\\
    %&=\EV[u][\mathbb{A}_N][\Bigg]{\sum_{j=J+1}^\infty j^{2\gamma/d} \iprod{\sqrt{N}(u-\tilde{u}_N)}{\phi_j}^2\given\data}^{1/2}\\
    %&=\Bigg(\sum_{j=J+1}^\infty j^{-2\alpha+2\gamma/d}\EV[u][\mathbb{A}_N]{ \iprod{\sqrt{N}(u-\tilde{u}_N)}{j^\alpha\phi_j}^2\given\data}\Bigg)^{1/2}.
    &\le \sum_{\abs{\mu}>K}2^{2\abs{\mu}(\gamma -m+\epsilon)} \EV[\Pi_N^{\mathbb{A}_N}]{\exp\left(\sqrt{N}\iprod{u - \tilde{u}_N}{2^{(m-\epsilon)\abs{\mu}}\psi_\mu}_{L^2(\domain)}\right)\given\data}\\
    &\quad +\sum_{\abs{\mu}>K}2^{2\abs{\mu}(\gamma -m+\epsilon)} \EV[\Pi_N^{\mathbb{A}_N}]{\exp\left(\sqrt{N}\iprod{u - \tilde{u}_N}{-2^{(m-\epsilon)\abs{\mu}}\psi_\mu}_{L^2(\domain)}\right)\given\data}
\end{split}\label{eq:Helper_Approximation_Error}
\end{align}
with $\epsilon>0$, where we used $x^2\le e^x+e^{-x}$ for $x\in\mathbb{R}$ in the last line.
Note that for any $\mu$ we have
\begin{equation*}
    \iprod{u - \tilde{u}_N}{\pm \psi_\mu}_{L^2(\domain)}  =\iprod[\mathcal{V}^\ast][\mathcal{V}]{\pm \psi_\mu}{u-\tilde{u}_N}
\end{equation*}
%with $\tilde{u}_N$ from \eqref{eq:EfficientEstimator_RandomDesign} with $\psi=\psi_\mu$.
Since for some $M>0$ we have $\norm{2^{(m-\epsilon)\abs{\mu}}\psi_\mu}_{\mathcal{V}^\ast}\lesssim \norm{2^{(m-\epsilon)\abs{\mu}}\psi_\mu}_{H^{-m}(\domain)}\lesssim 1$ for any $\mu$ by \eqref{eq:CharacterisationSobolevSpaces},
%\begin{equation*}
%    \norm{k^{\alpha}\psi_k}_{H^{-2+\kappa}(\domain)} \lesssim \norm{k^{\alpha}\psi_k}_{H^{\alpha d}(\domain)}\lesssim 1,
%\end{equation*}
 we can apply Proposition \ref{prop:Uniform_Laplace_Transform_Control} to deduce for any $\mu$ the bound
\begin{equation*}
    \sup_{\mu\in\Lambda}\EV[\Pi_N^{\mathbb{A}_N}]{\exp\left(\sqrt{N}\iprod{u - \tilde{u}_N}{\pm 2^{(m-\epsilon)\abs{\mu}}\psi_\mu}_{L^2(\domain)}\right)\given\data} = \mathcal{O}_{\prob[]{}}(1).
\end{equation*}
The summability of $2^{2\abs{\mu}(\gamma -m+\epsilon)}$ is ensured by $\gamma < m-d/2$ and $0<\epsilon< m-d/2-\gamma$:
\begin{equation*}
    \sum_{\mu\in\Lambda}2^{2\abs{\mu}(\gamma-m+\epsilon)}\lesssim \sum_{j\in\mathbb{N}}2^{2j(\gamma-m +\epsilon+d/2)}<\infty.
\end{equation*}
Thus, we conclude that 
\begin{equation*}
    \lim_{K\to\infty} \limsup_{N\to\infty}\prob[u_0][N]{A>\delta}=0
\end{equation*}
for any $\delta>0$.\\
For $B$ note that the convergence in $\prob[u_0][N]{}$ of the Laplace transformation from Proposition \ref{prop:Uniform_Laplace_Transform_Control} combined with Exercise 4.3.3 of \cite{N22} and the Cram\'er--Wold device (Page 16 of \cite{Vaart1998}) shows that 
\begin{equation*}
    d_{\operatorname{weak}}((\tilde{\tau}_N^K)_{\#}\Pi_N^{\mathbb{A}_N},\gaussian^K)\xrightarrow{\prob[u_0][N]{}}0
\end{equation*}
as $N\to\infty$ for every $K\in\mathbb{N}$ fixed, where $d_{\operatorname{weak}}$ is any metric for weak convergence of measures on $\mathbb{R}^{\operatorname{dim}(V_K)}$.
%\SW{Bit more detail: We use again that $\Pi_N(\tilde A_n|\data)\to 1$ etc.}\SG{We are already on the set $\mathbb{A}_N$} 
We are left to strengthen the topology to Wasserstein-1. 
By Theorem 6.9 (see also part (iii) of Definition 6.8) of \cite{villaniOptimalTransport2009}, it suffices to prove the uniform integrability. % \SW{The below is a convergence in probability? Perhaps one can strengthen it to almost sure convergence using Borel-Cantelli but it seems unnecessary.}
%\SW{One has to be careful because the measure is random.. thus the limsup is only defined $\omega$-wise ...}
%\begin{equation}
%    \lim_{R\to\infty}\limsup_{N\to\infty}\int_{\norm{u}_{H^{-\gamma}(\domain)}>R}\norm{u}_{H^{-\gamma}(\domain)}d \Pi_N^{\mathbb{A}_N}(u|\data)=0.\label{eq:claim:Wassersteinlift}
%\end{equation}
We do so by bounding the second moments
%\begin{align*}
%    \EV[\Pi_N^{\mathbb{A}_N}]{\norm{\sqrt{N}(\ProjThree_{K}( u - \tilde{u}_N))}_{H^{-\gamma}(\Omega)}^2\given\data},\quad K\in\mathbb{N}.
%\end{align*}
and proceed as in \eqref{eq:Helper_Approximation_Error} to find for any $\epsilon> 0$ the bound
\begin{align*}
    \EV[\Pi_N^{\mathbb{A}_N}]{\norm{\sqrt{N}(\ProjThree_{K}( u - \tilde{u}_N))}_{H^{\gamma}(\domain)}^2\given\data} \hspace{-10em}&\\
    &\le \EV[\Pi_N^{\mathbb{A}_N}][][\Big]{\sum_{\abs{\mu}\le K} 2^{2\abs{\mu}\gamma}\iprod{\sqrt{N}(u-\tilde{u}_N)}{\psi_\mu}_{L^2(\domain)}^2\given\data}\\
    &\le \sum_{\abs{\mu}\le K} 2^{2\abs{\mu}(\gamma-m)} \EV[\Pi_N^{\mathbb{A}_N}]{\exp\left(\sqrt{N}\iprod{u - \tilde{u}_N}{2^{m\abs{\mu}}\psi_\mu}_{L^2(\domain)}\right)\given\data}\\
    &\quad +\sum_{\abs{\mu}\le K} 2^{2\abs{\mu}(\gamma-m)} \EV[\Pi_N^{\mathbb{A}_N}]{\exp\left(\sqrt{N}\iprod{u - \tilde{u}_N}{-2^{m\abs{\mu}}\psi_\mu}_{L^2(\domain)}\right)\given\data},
    %&=\SW{\le} 2\left(\sum_{k=1}^K k^{-2\gamma/d+2\alpha} \EV[\Pi_N^{\mathbb{A}_N}]{\exp\left(\sqrt{N}\iprod{u - \tilde{u}_N}{k^{-\alpha}\psi_k}_{L^2(\domain)}\right)\given\data}\right)^{1/2}.
\end{align*}
where we used $x^2\le e^x+e^{-x}$ for any $x\in\mathbb{R}$.
Applying Proposition \ref{prop:Uniform_Laplace_Transform_Control} with $\psi=\pm2^{m\abs{\mu}}\psi_\mu$ such that $\norm{2^{m\abs{\mu}}\psi_\mu}_{\mathcal{V}^\ast}\le \norm{2^{m\abs{\mu}}\psi_\mu}_{H^{-m}(\domain)}\le M$ for some $M>0$ we find a uniform (in $N\in\mathbb{N}$) upper bound for the second moments and conclude the proof of the convergence in $\mathcal{W}_{1,H^{\gamma}(\domain)}$.\\
Part C is controlled in Lemma \ref{lem:Convergence_Gaussian_Measures}. Substituting the previous bounds for $A$, $B$ and $C$ back into \eqref{eq:decompBvM}, we find
\begin{equation*}
    \mathcal{W}_{1,H^\gamma(\domain)}((\tilde{\tau}_N)_{\#}\Pi_N^{\mathbb{A}_N},\gaussian)\xrightarrow{\prob[u_0][N]{}}0
\end{equation*}
as $N\to\infty$\\
To show  \eqref{eq:claim_BvM_localised}, we next argue that we can replace the centering $\tilde{u}_N$ by the localised posterior mean $\bar{u}_N^{\mathbb{A}_N}\coloneqq \EV[\Pi_N^{\mathbb{A}_N}]{u\given\data}$. 
By Equation (6.3) of \cite{villaniOptimalTransport2009} we have
%\begin{equation*}
%    \norm{\EV[\Pi_N^{\mathbb{A}_N}]{\sqrt{N}(u-\tilde{u}_N)|\data}-\EV[X\sim\gaussian]{X}}_{H^\}
%\end{equation*}
%Since convergence in 1-Wasserstein distance implies convergence of the first moments by Theorem 6.9 of \cite{villaniOptimalTransport2009}, we have \SW{This needs to be made precise. This is convergence in probability with respect to the $H^{\gamma}(\domain)$ norm (right?) -- I would perhaps even be tempted to write this in $\eps,~\delta$ style ...}\SG{@Sven better now?}
 \begin{align}
 \begin{split}
     \norm{\sqrt{N}(\bar{u}_N^{\mathbb{A}_N}-\tilde{u}_N)}_{H^{\gamma}(\domain)}&=\norm{\EV[\Pi_N^{\mathbb{A}_N}]{\sqrt{N}(u-\tilde{u}_N)|\data} - \EV[X\sim\gaussian]{X}}_{H^{\gamma}(\domain)}\\
     &\le \mathcal{W}_{1,H^\gamma(\domain)}((\tilde{\tau}_N)_{\#}\Pi_N^{\mathbb{A}_N},\gaussian)\xrightarrow{\prob[u_0][N]{}}0
     \end{split}\label{eq:Comparison_PosteriorMean_EfficientEstimator}
 \end{align}
 as $N\to\infty$ and by Lemma \ref{lem:Localisationisfine} we have
 \begin{align}
 \begin{split}
     \norm{\sqrt{N}(\bar{u}_N^{\mathbb{A}_N}-\bar{u}_N)}_{H^\gamma(\domain)} &= \norm{\EV[\Pi_N^{\mathbb{A}_N}]{\sqrt{N}(u-\bar{u}_N)}-\EV[\Pi_N]{\sqrt{N}(u-\bar{u}_N)}}_{H^\gamma(\domain)} \\
     &\le \mathcal{W}_{1,H^\gamma(\domain)}((\tau_N)_{\#}\Pi_N^{\mathbb{A}_N},(\tau_N)_{\#}\Pi_N)\xrightarrow{\prob[u_0][N]{}}0
     \end{split}\label{eq:Comparison_PosteriorMean_Locality}
 \end{align}
 as $N\to\infty$.
 Consequently,
%\begin{equation*}
%    \sqrt{N}\left(\hat{u} - \tilde{u}_N\right) =\EV[\Pi_N]{\sqrt{N}(u-\tilde{u}_N)\given\data}\xrightarrow[H^{-\gamma}(\domain)]{N\to\infty} \EV[X\sim\gaussian]{X} = 0,
%\end{equation*}
%such that
\begin{align*}
    \mathcal{W}_{1,H^{\gamma}(\domain)}\left((\tilde{\tau}_N)_{\#}\Pi_N^{\mathbb{A}_N}, (\tau_N)_{\#}\Pi_N^{\mathbb{A}_N}\right)&\le \sqrt{N}\norm{\bar{u}_N^{\mathbb{A}_N}-\tilde{u}_N}_{H^{\gamma}(\domain)}+\sqrt{N}\norm{\bar{u}_N^{\mathbb{A}_N}-\bar{u}_N}_{H^{\gamma}(\domain)}\\
    &\xrightarrow{\prob[u_0][N]{}} 0
\end{align*}
as $N\to\infty$ and \eqref{eq:claim_BvM_localised} follows. The proof is completed by appealing to Lemma \ref{lem:Localisationisfine}.
\end{proof}

\begin{proof}[Proof of Lemma \ref{lem:Convergence_PosteriorMean}]
    The proof strategy is similar to the proof of Theorem 2 of \cite{nicklBernsteinvonMisesTheorems2025}.
    By \eqref{eq:Comparison_PosteriorMean_EfficientEstimator}, \eqref{eq:Comparison_PosteriorMean_Locality} and Slutsky's Lemma, it suffices to prove the statement with $\tilde{u}_N$ instead of $\bar{u}_N$.

    By the triangle inequality, we have for any $K\in\mathbb{N}$ the decomposition
    \begin{align}
    \begin{split}
        \mathcal{W}_{1,H^{\gamma}(\domain)}\left(\Law[\prob[u_0][N]{}]{\sqrt{N}(\tilde{u}_N-\ProjTwo_J u_0)}, \gaussian\right) \hspace{-16em}& \\
        &\le \mathcal{W}_{1,H^{\gamma}(\domain)}\left(\Law[\prob[u_0][N]{}]{\sqrt{N}\ProjThree_{K}(\tilde{u}_N-\ProjTwo_J u_0)}, \Law[\prob[u_0][N]{}]{\sqrt{N}(\tilde{u}_N-\ProjTwo_J u_0)}\right)\\
        &\quad+\mathcal{W}_{1,H^{\gamma}(\domain)}\left(\Law[\prob[u_0][N]{}]{\sqrt{N}\ProjThree_{K}(\tilde{u}_N-\ProjTwo_J u_0)}, \gaussian^{K}\right)\\
        &\quad + \mathcal{W}_{1,H^{\gamma}(\domain)}\left(\gaussian^{K}, \gaussian\right).
        \end{split}\label{eq:aux_WassersteinDecomposition_PosteriorMean}
    \end{align}
    To control the first summand in \eqref{eq:aux_WassersteinDecomposition_PosteriorMean}, we use the stability \eqref{eq:H^m_stability} for $\kappa=\gamma$ and %and proceed as in \eqref{eq:determinisitic_remainder} and 
    \eqref{eq:stochastic_remainder} to compute
    \begin{align*}
        \sup_{N\in\mathbb{N}}\mathcal{W}_{1,H^{\gamma}(\domain)}\left(\Law[\prob[u_0][N]{}]{\sqrt{N}\ProjThree_{K}(\tilde{u}_N-\ProjTwo_J u_0)}, \Law[\prob[u_0][N]{}]{\sqrt{N}(\tilde{u}_N-\ProjTwo_J u_0)}\right)^2\hspace{-20em} &\\
        &\le \sup_{N\in\mathbb{N}}\EV[\prob[u_0][N]{}]{\norm{\sqrt{N}(\identity -\ProjThree_{K})(\tilde{u}_N-\ProjTwo_J u_0)}_{H^{\gamma}(\domain)}}^2\\
        &\lesssim \sum_{\abs{\mu}\ge K}2^{2\abs{\mu}(\gamma-m+\delta)}\to 0
    \end{align*}
    as $K\to\infty$ by \eqref{eq:stochastic_remainder}.
    
    To bound the second summand in \eqref{eq:aux_WassersteinDecomposition_PosteriorMean}, we first recall that
    %note that by \eqref{eq:determinisitic_remainder} we find 
    for any $v=\psi|_{\mathcal{V}}\neq 0\in \mathcal{V}^\ast$, $\psi\in H^{-\gamma}(\domain)$, we have
    \begin{equation*}
        \iprod[\mathcal{V}^\ast][\mathcal{V}]{v}{\sqrt{N}(\tilde{u}_N-\ProjTwo_J u_0)} =\frac{1}{\sigma \sqrt{N}}\sum_{i=1}^N \epsilon_i\opDeriv{u_0}[\ProjTwo_{J}\mathcal{I}_{u_0}^{-1}v](X_i) = \colon \sum_{i=1}^N\xi_{i,N}
        %+\sum_{k=1}^{J} \frac{1}{\sigma \sqrt{N}}\sum_{i=1}^N \epsilon_i \lambda_k^{-1}\psi_k(X_i)\psi_k
        %+ \smallo_{\prob[u_0][N]{}}(1).
    \end{equation*}
    with centered $\xi_{i,N} = \epsilon_i \opDeriv{u_0}[\ProjTwo_{J}\mathcal{I}_{u_0}^{-1}v](X_i)/(\sigma\sqrt{N})$, $i=1,\dots, N$. 
    Note that $\xi_{i,N}$ depends on $N$ through the projection $\ProjTwo_J$. By \eqref{eq:aux_convergence_projections} we find $\sum_{i=1}^N\EV[\prob[u_0][N]{}]{\xi_{i,N}^2}=\norm{\ProjTwo_J v_0}_{\operatorname{LAN}}^2\to \norm{v_0}_{\operatorname{LAN}}^2>0$ with $v_0=\mathcal{I}_{u_0}^{-1}v$ as $N\to\infty$. We also compute
    \begin{align*}
        \sum_{i=1}^N\EV[\prob[u_0][N]{}]{\xi_{i,N}^4} &= \frac{1}{\sigma^4 N} \EV[\prob[u_0]{}]{\epsilon_1^4}\EV[\prob[u_0]{}]{\opDeriv{u_0}[\ProjTwo_{J}\mathcal{I}_{u_0}^{-1}v](X_1)^4}\\
        &\lesssim \frac{1}{N}\norm{\opDeriv{u_0}[\ProjTwo_J\mathcal{I}_{u_0}^{-1}v]}_{L_p^4(\domain)}^4\le \frac{1}{N}\norm{\opDeriv{u_0}[\ProjTwo_J\mathcal{I}_{u_0}^{-1}v]}_{L^\infty(\domain)}^2\norm{\opDeriv{u_0}[\ProjTwo_J\mathcal{I}_{u_0}^{-1}v]}_{L_p^2(\domain)}^2\\
        &\lesssim \frac{1}{N}\norm{\ProjTwo_J\mathcal{I}_{u_0}^{-1}v}_{H^{m+\kappa}(\domain)}^2\norm{\ProjTwo_J\mathcal{I}_{u_0}^{-1} v}_{\operatorname{LAN}}^2\\
        &\lesssim \frac{1}{N}2^{2\kappa J}\norm{\mathcal{I}_{u_0}^{-1}v}_{H^m(\domain)}^2\norm{v}_{\mathcal{V}^\ast}^2\\
        &\lesssim N^{2\kappa/(2\beta+d)-1}\norm{v}_{\mathcal{V}^\ast}^4\to 0
    \end{align*}
    as $N\to\infty$ for $d/2<\kappa<\beta+d/2$, where we used the independence of $(\epsilon_1,\dots,\epsilon_N)$ from $(X_1,\dots, X_N)$, that $\opDeriv{u_0}\colon \mathcal{V}\cap W^{m,\infty}(\domain)\to L^\infty(\domain)$ is bounded by Assumption \ref{assump:operator_linearisation}, the Sobolev embedding, the $\norm{}_{\operatorname{LAN}}$-stability of $\ProjTwo_J$, the Bernstein inequality \eqref{eq:Bernstein_orthogonalprojection} and \eqref{eq:InformationIsomorphism}.
    Consequently, the Lindeberg--Feller central limit theorem (Proposition 2.27 of \cite{Vaart1998}) implies that
    %Note that for every $k\in\mathbb{N}$ we have
    %\begin{equation}
    %    \EV*{N\iprod{\tilde{u}_N-u_0}{\psi_k}_{L^2(\domain)}^2} \SW{=
   %     E \big[ \frac{1}{\sigma^2}\lambda_k^{-2} \langle \eta, \psi_k\rangle_N \big]}
    %    =\frac{1}{\sigma^2} \lambda_k^{-2}\SW{=\frac{1}{\sigma^2}\norm{\psi_k}_{\operatorname{LAN}}^2}.\label{eq:UniformControlSecondMoments}
    %\end{equation}
    %\SW{I think it should be $\lambda_k^{-2}$?}\SG{Yes, corrected}
    %the central limit theorem and \eqref{eq:aux_convergence_projections} imply
    \begin{equation*}
        \iprod[\mathcal{V}^\ast][\mathcal{V}]{v}{\sqrt{N}(\tilde{u}_N-\ProjTwo_Ju_0)} \xrightarrow{d}N\left(0,\norm{\mathcal{I}_{u_0}^{-1}v}_{\operatorname{LAN}}^2\right)
    \end{equation*}
    as $N\to\infty$. The Cram\'er--Wold device (p. 16 of \cite{Vaart1998}) implies convergence of the finite-dimensional distributions of $\sqrt{N}(\tilde{u}_N-\ProjTwo_Ju_0)$ to those of $\gaussian$. Moreover, \eqref{eq:stochastic_remainder} implies that
    \begin{equation*}
        \sup_{N\in\mathbb{N}}\EV[\prob[u_0][N]{}]{\norm{\sqrt{N}(\tilde{u}_N-\ProjTwo_Ju_0)}_{H^\gamma(\domain)}^2}<\infty,
    \end{equation*}
    such that 
    \begin{equation*}
        \mathcal{W}_{1,H^\gamma(\domain)}\left(\Law[\prob[u_0][N]{}]{\sqrt{N}\ProjOne_K (\tilde{u}-\ProjTwo_Ju_0)},\gaussian^K\right)\to 0 
    \end{equation*}
    as $N\to\infty$ for any $K\in\mathbb{N}$.
    
The third summand in \eqref{eq:aux_WassersteinDecomposition_PosteriorMean} can be made arbitrarily small by choosing $K$ large enough by Lemma \ref{lem:Convergence_Gaussian_Measures}. 
    
\end{proof}

Recall the transformation $\tau_N(u) = \sqrt{N}(u-\bar{u}_N)$ for $u\in V_J$.

\begin{lem}[Localisation]\label{lem:Localisationisfine}
Grant Assumptions \ref{assump:ApproximatonSets}, \ref{assump:Operator_stability} and \ref{assump:operator_linearisation}. Let $m>d/2$, $0<\gamma<m-d/2$, consider the prior $\pi_N$ from \eqref{eq:Prior} with cut-off $2^J\sim N^{1/(2\beta+d)}$ and assume that $u_0\in \mathcal{V}\cap H^{m+\beta}(\domain)$ for $\beta>2d$, $d/2<\beta_0<\beta+d/4$. Then
\begin{equation*}
    \mathcal{W}_{1,H^{\gamma}(\domain)}((\tau_N)_{\#}\Pi_N^{\mathbb{A}_N},(\tau_N)_{\#}\Pi_N)=\smallo_{\prob[u_0][N]{}}(1),\quad N\to\infty.
\end{equation*}
\end{lem}
\begin{proof}
    Let $T=\sqrt{N}(u_0-\bar{u}_N)$, which does not depend on an individual posterior draw $u$. We use Theorem 6.15 of \cite{villaniOptimalTransport2009} to deduce
    \begin{align*}
        \mathcal{W}_{1,H^{\gamma}(\domain)}((\tau_N)_{\#}\Pi_N^{\mathbb{A}_N},(\tau_N)_{\#}\Pi_N)\hspace{-10em}&\\
        &\le \left(\int_{H^{\gamma}(\domain)} \norm{u-T}_{H^{\gamma}(\domain)}d\abs{(\tau_N)_{\#}\Pi_N^{\mathbb{A}_N}-(\tau_N)_{\#}\Pi_N}(u)\right)\\
        &=\left(\int_{N^{-1/2}\norm{u-T}_{H^{\gamma}(\domain)}\le 1}\norm{u-T}_{H^{\gamma}(\domain)} d\abs{(\tau_N)_{\#}\Pi_N^{\mathbb{A}_N}-(\tau_N)_{\#}\Pi_N}(u) \right.\\
        &\quad \left. + \int_{N^{-1/2}\norm{u-T}_{H^{\gamma}(\domain)}>1}\norm{u-T}_{H^{\gamma}(\domain)} d\abs{(\tau_N)_{\#}\Pi_N^{\mathbb{A}_N}-(\tau_N)_{\#}\Pi_N}(u)\right)\\
        &\lesssim \sqrt{N}\operatorname{TV}((\tau_N)_{\#}\Pi_N^{\mathbb{A}_N},(\tau_N)_{\#}\Pi_N)\\
        &\quad + \int_{N^{-1/2}\norm{u-T}_{H^{\gamma}(\domain)}>1}\norm{u-T}_{H^{\gamma}(\domain)} \,\D (\tau_N)_{\#}\Pi_N^{\mathbb{A}_N}(u)\\
        &\quad + \int_{N^{-1/2}\norm{u-T}_{H^{\gamma}(\domain)}>1}\norm{u-T}_{H^{\gamma}(\domain)} \,\D (\tau_N)_{\#}\Pi_N(u).
    \end{align*}
    For the first term we use the results from page 142 of \cite{Vaart1998} to deduce
    \begin{equation*}
        \sqrt{N}\operatorname{TV}((\tau_N)_{\#}\Pi_N^{\mathbb{A}_N},(\tau_N)_{\#}\Pi_N) \le \sqrt{N}\Pi_N(\mathbb{A}_N^c|\data)=\smallo_{\prob[u_0][N]{}}\left(\sqrt{N}e^{-DN\epsilon_N^2}\right)=\smallo_{\prob[u_0][N]{}}(1)
    \end{equation*}
    by the contraction \eqref{eq:Contraction_RandomDesign} and Lemma \ref{lem:Prior} \eqref{num:aux_rescaled_Prior_restrict}. Since $\norm{u-u_0}_{H^{\gamma}(\domain)}\le \norm{u-u_0}_{H^m(\domain)}\to 0$ $\Pi_N^{\mathbb{A}_N}(\MTemptyplaceholder|\data)$ almost surely, the second term is zero for large enough $N$. For the third term, we note that it equals
    \begin{align*}
        I&\coloneqq \int_{N^{-1/2}\norm{u-T}_{H^{\gamma}(\domain)}>1}\norm{u-T}_{H^{\gamma}(\domain)} \,\D (\tau_N)_{\#}\Pi_N(u)\\
        &=\EV*[\Pi_N]{\norm*{\sqrt{N}(u-u_0)}_{H^{\gamma}(\domain)}\indicator_{\{\norm{u-u_0}_{H^{\gamma}(\domain)}^2>1\}}\given  \data}\\
        %&\le \sqrt{N}\EV*[\Pi_N]{\norm{u-u_0}_{H^{\gamma}(\domain)}\indicator_{\{\norm{u-u_0}_{H^{\gamma}(\domain)}^2>1\}}\given \data}\\
        &\le \sqrt{N}\EV*[\Pi_N]{\norm{u-u_0}_{H^{\gamma}(\domain)}^2\given \data}^{1/2}\Pi_N(\norm{u-u_0}_{H^m(\domain)}>c|\data)^{1/2}
        %&\le \sqrt{N}\EV*[\Pi_N]{\norm{u-u_0}_{H^{\gamma}(\domain)}^2\given \data}^{1/2}e^{-DN\epsilon_N^2/2},
    \end{align*}
    for some constant $0<c<\infty$.
    Note that $\int e^{\ell_N(u)-\ell_N(u_0)}\D\pi_N(u) \ge e^{-\tilde{D}N\epsilon_N^2}$ with probability tending to one by \eqref{eq:Marginal_Likelihood_Bound_RandomDesign}. Denote this sequence of events by $\Omega_N$, then we have 
    \begin{equation*}
        \EV[\prob[u_0][N]{}]{\indicator_{\Omega_N} \EV[\Pi_N]{\norm{u-u_0}_{H^{\gamma}(\domain)}^2\given \data}}\le e^{\tilde{D}N\epsilon_N^2}\int_{V_J}\norm{u-u_0}_{H^\gamma(\domain)}^2\,\D\pi_N(u)
    \end{equation*}
    since $\EV[\prob[u_0][N]{}]{e^{\ell_N(u)-\ell_N(u_0)}} = 1$ (expectation of the likelihood ratio). %Since $\gamma<m-d/2<m+\beta_0-d/2$, we can view the prior as a truncated Gaussian measure on $H^{\gamma}(\domain)$ and obtain from Fernique's Theorem (Theorem 2.7 of \cite{DaPrato2014}) a constant $0<\bar{C}<\infty$, independent of $N\in\mathbb{N}$, such that
    Moreover, by \eqref{eq:CharacterisationSobolevSpaces} and since $\beta_0>d/2$, we find 
    \begin{align*}
        \int_{V_J}\norm{u-u_0}_{H^\gamma(\domain)}^2\,\D\pi_N(u)&\lesssim  1+ \int_{V_J}\norm{u}_{H^m(\domain)}^2\,\D\pi_N(u)\lesssim 1 + \sum_{\abs{\mu}\le J}2^{2m\abs{\mu}}2^{-2(m+\beta_0)\abs{\mu}}\\
        &\lesssim \sum_{j\le J}2^{jd}2^{-2j\beta_0}\le \sum_{j\in\mathbb{N}}2^{-j(2\beta_0-d)}\lesssim 1.
    \end{align*}
    Consequently, we find
    \begin{equation*}
        I = \mathcal{O}_{\prob[u_0][N]{}}(\sqrt{N}e^{-(D-\tilde{D})N\epsilon_N^2/2}),
    \end{equation*}
    where we used the contraction \eqref{eq:Contraction_RandomDesign}. Since $D>\tilde{D}$, we find $I=\smallo_{\prob[u_0][N]{}}(1)$ as $N\to\infty$.
\end{proof}

\begin{lem}[Convergence of the Gaussian measures]\label{lem:Convergence_Gaussian_Measures}
    Assume that $m>d/2$, $0<\gamma<m-d/2$ and let $\gaussian^{K}$, $K\in\mathbb{N}$ be the push-forward measure of $\gaussian$ under $\ProjThree_K$. Then
    \begin{equation*}
    \mathcal{W}_{1,H^{\gamma}(\domain)}(\gaussian^{K},\gaussian)\to 0
\end{equation*}
as $K\to\infty$.
\end{lem}
\begin{proof}
    By Fernique's Theorem (Theorem 2.7 of \cite{DaPrato2014}) we have $\EV[X\sim\gaussian]{\norm{X}_{H^{\gamma}(\domain)}^2}<\infty$ and $\EV[X\sim\gaussian^K]{\norm{X}_{H^{\gamma}(\domain)}^2}<\infty$. Consequently, we find
    \begin{align*}
        \mathcal{W}_{1,H^{\gamma}(\domain)}(\gaussian^{K},\gaussian)^2 &= \sup_{F\colon \norm{F}_{\operatorname{Lip}}\le 1}\abs{\EV[X\sim\gaussian]{F(\ProjThree_{K} X)-F(X)}}^2\\
     &\le \EV[X\sim\gaussian]{\norm{(\mathbb{I}-\ProjThree_{K})X}_{H^{\gamma}(\domain)}^2}.
    \end{align*}
    Similarly to \eqref{eq:aux_DensityOriginalprojections} we can show that $\norm{(\mathbb{I}-\ProjThree_{K})X}_{H^{\gamma}(\domain)}\to 0$ $\gaussian$-almost surely, such that the claim follows by the $H^\gamma(\domain)$ stability of $\ProjOne_K$ and the dominated convergence theorem.
    %which transfers to $\EV{\abs{F(X)}}<\infty$ for any Lipschitz-continuous $F\colon H^{\gamma}(\domain)\to\mathbb{R}$. Moreover, we find
    %\begin{align*}
    % \mathcal{W}_{1,H^{\gamma}(\domain)}(\gaussian^{K},\gaussian) &= \sup_{F\colon \norm{F}_{\operatorname{Lip}}\le 1}\abs{\EV[X\sim\gaussian]{F(\ProjThree_{K} X)-F(X)}}\\
    % &\le \EV[X\sim\gaussian]{\norm{(\mathbb{I}-\ProjOne_{K})X}_{H^{\gamma}(\domain)}}\\
    % &\le \sqrt{\EV[X\sim\gaussian]{\norm{(\mathbb{I}-\ProjThree_{K})X}_{H^{\gamma}(\domain)}^2}}\\
    % &\lesssim \sqrt{\sum_{\abs{\mu}>{K}}2^{2\abs{\mu}\gamma}\norm{\psi_\mu}_{\operatorname{LAN}}^2}\\
    % &\lesssim \sqrt{\sum_{\abs{\mu}>{K}}2^{2\abs{\mu}(\gamma+2)}},
%\end{align*}
%where the summability of the full sum is verified for $0<\gamma<m-d/2$ and tail convergence follows as $K\to\infty$.
\end{proof}

\begin{proof}[Proof of Theorem \ref{thm:nonlinear_wasserstein_delta_method}]
%By Lemma \ref{lem:Localisationisfine} we have 
%\begin{equation*}
%    \mathcal{W}_{1,H^\gamma(\domain)}(\bar{\tau}_N,\tau_N)=\smallo_{\prob[u_0][N]{}}(1),
%\end{equation*}
%where $\tau_N$ and $\bar{\tau}_N$ are the laws of $\sqrt{N}(u-\hat{u})$ with $u\sim \Pi_N(\MTemptyplaceholder|\data)$ and $u\sim \Pi_N^{\mathbb{A}_N}(\MTemptyplaceholder|\data)$, respectively.
\textbf{Step 1: Preliminaries.}
Recall that from Theorem \ref{thm:Contraction} and Lemma \ref{lem:Prior} \eqref{num:aux_rescaled_Prior_restrict} we have
\begin{equation}
    \Pi_N(\mathbb{A}_N^c)=\Pi_N(\mathbb{A}_N^c|\data) = \mathcal{O}_{\prob[u_0][N]{}}\left(e^{-DN\epsilon_N^2}\right).
\label{eq:GoodSetComplementProbability}
\end{equation}
 By \eqref{eq:claim_BvM_localised} we know that
\begin{equation}
    \mathcal{W}_{1,H^\gamma(\domain)}((\tau_N)_{\#}\Pi_N^{\mathbb{A}_N},\gaussian) =\smallo_{\prob[u_0][N]{}}(1),\label{eq:LocalisedBvMForDeltaMethod}
\end{equation}
where $\tau_N(u) =\sqrt{N}(u-\bar{u}_N)$ for $u\in V_J\supset \mathbb{A}_N$,
such that
\begin{equation}
    \int_{\mathbb{A}_N} \norm{\sqrt{N}(u-\bar{u}_N)}_{H^\gamma(\domain)}\,\D \Pi_N^{\mathbb{A}_N}(u) = \mathcal{O}_{\prob[u_0][N]{}}(1).\label{eq:LocalisedPosteriorFirstMoment}
\end{equation}
Furthermore, Lemma \ref{lem:Convergence_PosteriorMean} and the Jackson estimate \eqref{eq:Jackson_orthogonalprojection} imply
\begin{align}
\begin{split}
\norm{\bar{u}_N- u_0}_{H^\gamma(\domain)}&\le \norm{\bar{u}_N-\ProjTwo_J u_0}_{H^\gamma(\domain)}+\norm{(\identity -\ProjTwo_J) u_0}_{H^\gamma(\domain)}\\
&=\mathcal{O}_{\prob[u_0][N]{}}(N^{-1/2}) + \norm{(\identity -\ProjTwo_J) u_0}_{H^m(\domain)}\\
&= \mathcal{O}_{\prob[u_0][N]{}}(N^{-1/2}) + \mathcal{O}(\epsilon_N)= \mathcal{O}_{\prob[u_0][N]{}}(\epsilon_N).
\end{split}
\label{eq:PosteriorMeanRootNTight}
\end{align}
Since $0<\gamma<m$ we have
\begin{equation*}
    \delta_N\coloneqq \sup_{u\in\mathbb{A}_N}\norm{u-u_0}_{H^\gamma(\domain)}\lesssim \epsilon_N\to 0,\quad N\to\infty
\end{equation*}
and for sufficiently large $N\in\mathbb{N}$ we have $\mathbb{A}_N \subset \mathscr{U}$.\\
\textbf{Step 2 (Localisation).}
We aim to show that
\begin{equation}
    \mathcal{W}_{1,V}\left((S_{N})_{\#}\Pi_N, (S_{N})_{\#}\Pi_N^{\mathbb{A}_N}\right) = \smallo_{\prob[u_0][N]{}}(1)\label{eq:TransformedLocalisation}
\end{equation}
as $N\to\infty$.
With 
\begin{equation*}
 b_{N}\coloneqq \sqrt{N}(\Phi(u_0) - \Phi(\bar{u}_N))
\end{equation*}
we have that $S_{N}(u) - b_{N} = \sqrt{N}(\Phi(u) - \Phi(u_0))$.
Using Theorem 6.15 of \cite{villaniOptimalTransport2009} we find
\begin{align}
    \mathcal{W}_{1,V}\left((S_{N})_{\#}\Pi_N, (S_{N})_{\#}\Pi_N^{\mathbb{A}_N}\right) \hspace{-5em}&\hspace{5em}\le \int_{\mathbb{A}_N^c}\norm{S_{N}(u) - b_{N}}_V \,\D \Pi_N(u)\nonumber\\
    &\quad +\Pi_N(\mathbb{A}_N^c) \int_{\mathbb{A}_N}\norm{S_{N}(v) - b_{N}}_V \,\D \Pi_N^{\mathbb{A}_N}(v)\nonumber\\
    \begin{split}
    &=\int_{\mathbb{A}_N^c}\norm{\sqrt{N}(\Phi(u)-\Phi(u_0))}_V \,\D \Pi_N(u)\\
    &\quad +\Pi_N(\mathbb{A}_N^c) \int_{\mathbb{A}_N}\norm{\sqrt{N}(\Phi(v)-\Phi(u_0))}_V \,\D \Pi_N^{\mathbb{A}_N}(v).\end{split}\label{eq:TransformedLocalisationCoupling}
\end{align}
Since $\Phi$ is locally  Lipschitz-continous, we can bound the
second term in \eqref{eq:TransformedLocalisationCoupling} by
\begin{equation}
    \sqrt{N}\epsilon_N \Pi_N(\mathbb{A}_N^c)=\smallo_{\prob[u_0][N]{}}(1),\quad N\to\infty,\label{eq:TransformedLocalisationGoodPart}
\end{equation}
where we used \eqref{eq:GoodSetComplementProbability}.
For the first term in \eqref{eq:TransformedLocalisationCoupling} combining the Cauchy--Schwarz inequality and the polynomial growth condition for $\Phi$ yields
\begin{align*}
    \sqrt{N}\int_{\mathbb{A}_N^c}\norm{\Phi(u)-\Phi(u_0)}_V\,\D\Pi_N(u)& \le \sqrt{N\EV[\Pi_N]{\norm{\Phi(u) - \Phi(u_0)}_V^2}\Pi_N(\mathbb{A}_N^c)}\\
    &\lesssim \sqrt{N\EV[\Pi_N]{1+\norm{u}_{H^\gamma(\domain)}^{2q}}\Pi_N(\mathbb{A}_N^c)}
\end{align*}
On the event where $\int_{V_J}e^{\ell_N(u)-\ell_N(u_0)}\,\D\pi_N(u) \ge e^{-\tilde{D}N\epsilon_N^2}$, which has $\prob[u_0][N]{}$-mass approaching one by \eqref{eq:Marginal_Likelihood_Bound_RandomDesign} we find that
\begin{align*}
    \EV[\Pi_N]{1+\norm{u}_{H^\gamma(\domain)}^{2q}}&\le e^{\tilde{D}N\epsilon_N^2}\int_{V_J}(1+\norm{u}_{H^\gamma(\domain)}^{2q})e^{\ell_N(u)-\ell_N(u_0)}\,\D \pi_N(u).
\end{align*}
Since $\EV[\prob[u_0][N]{}]{e^{\ell_N(u)-\ell_N(u_0)}} = 1$ for any $u\in V_J$, we find using Fernique's Theorem (Theorem 2.7 of \cite{DaPrato2014}) that
\begin{equation*}
    \EV*[\prob[u_0][N]{}]{\int_{V_J}(1+\norm{u}_{H^\gamma(\domain)}^{2q})e^{\ell_N(u)-\ell_N(u_0)}\,\D \pi_N(u)} = \int_{V_J}(1+\norm{u}_{H^\gamma(\domain)}^{2q})\,\D \pi_N(u)\le \bar{C}
\end{equation*}
uniformly in $N\in\mathbb{N}$ for some constant $0<\bar{C}<\infty$ since $\gamma<m-d/2<m+\beta_0-d/2$ such that the prior can be realised as the truncation of a Gaussian measure on $H^{\gamma}(\domain)$.
Markov's inequality and \eqref{eq:GoodSetComplementProbability} therefore imply
\begin{equation}
    N\EV[\Pi_N]{1+\norm{u}_{H^\gamma(\domain)}^{2q}}\Pi_N(\mathbb{A}_N^c) = \mathcal{O}_{\prob[u_0][N]{}}(N e^{\tilde{D}N\epsilon_N^2}\bar{C}e^{-DN\epsilon_N^2}) = \smallo_{\prob[u_0][N]{}}(1),\label{eq:PosteriorMomentPhi}
\end{equation}
where the latter bound follows since $D>\tilde{D}$ in Theorem \ref{thm:Contraction}.
Combining \eqref{eq:TransformedLocalisationCoupling}, \eqref{eq:TransformedLocalisationGoodPart} and \eqref{eq:PosteriorMomentPhi} shows the claimed localisation bound \eqref{eq:TransformedLocalisation}.\\

\textbf{Step 3 (Conclusion).} 
Note that \eqref{eq:PosteriorMeanRootNTight} implies that $\norm{\bar{u}_N-u_0}_{H^\gamma(\domain)}=\smallo_{\prob[u_0][N]{}}(1)$, such that $\bar{u}_N\in\mathscr{U}$ with probability tending to one as $N\to\infty$. Choose $r>0$ such that $\{v\in\mathbb{V}_\gamma\colon\norm{v-u_0}_{H^\gamma(\domain)}<r\}\subset \mathscr{U}$. With probability tending to one we have $\delta_N+\norm{\bar{u}_N-u_0}_{H^\gamma(\domain)}<r$. On this event, for every $u\in\mathbb{A}_N$ and $0\le t\le 1$, we have 
\begin{equation*}
    \norm{\bar{u}_N+t(u-\bar{u}_N)-u_0}_{H^\gamma(\domain)}\le \norm{\bar{u}_N-u_0}_{H^\gamma(\domain)}+\delta_N<r.
\end{equation*}
Hence, the line segment between $\bar{u}_N$ and $u$ is contained in $\mathscr{U}$ for any $u\in\mathbb{A}_N$ with probability tending to one. Define $R_2(u)\coloneqq \Phi(u)-\Phi(\bar{u}_N)-\mathcal{D}\Phi_{u_0}(u-\bar{u}_N)$ for any $u\in\mathbb{A}_N$. By the fundamental theorem of calculus (Proposition A.2.3 of \cite{Roeckner2015}), we find
\begin{equation*}
    R_2(u) = \int_0^1 (\mathcal{D}\Phi_{\bar{u}_N + t(u-\bar{u}_N)}-\mathcal{D}\Phi_{u_0})(u-\bar{u}_N)\,\D t, \quad u\in \mathbb{A}_N,
\end{equation*}
such that
\begin{align*}
    \norm{R_2(u)}_V&\lesssim \int_0^1 \norm{\bar{u}_N+t(u-\bar{u}_N)-u_0}_{H^\gamma(\domain)}\,\D t \norm{u-\bar{u}_N}_{H^\gamma(\domain)}\\
    & \le (\delta_N + \norm{\bar{u}_N-u_0}_{H^\gamma(\domain)})\norm{u-\bar{u}_N}_{H^\gamma(\domain)}.
\end{align*}
This implies that 
\begin{align*}
    \sqrt{N}\int_{\mathbb{A}_N}&\norm{R_2(u)}_V\,\D \Pi_N^{\mathbb{A}_N}(u)\\
    &\lesssim (\delta_N+\norm{\bar{u}_N-u_0}_{H^\gamma(\domain)})\int_{\mathbb{A}_N}\norm{\sqrt{N}(u-\bar{u}_N)}_{H^\gamma(\domain)}\,\D\Pi_N^{\mathbb{A}_N}(u)\\
    &= \smallo_{\prob[u_0][N]{}}(1)
\end{align*}
by \eqref{eq:LocalisedPosteriorFirstMoment},
\eqref{eq:PosteriorMeanRootNTight} and $\delta_N\to0$. Since for any $u\in\mathbb{A}_N$, we have $S_{N,2}(u)=\mathcal{D}\Phi_{u_0}(\sqrt{N}(u-\bar{u}_N))+\sqrt{N}R_2(u)$, we obtain using Theorem 6.15 of \cite{villaniOptimalTransport2009} the bound
\begin{equation*}
    \mathcal{W}_{1,V}\left((S_{N,2})_{\#}\Pi_N^{\mathbb{A}_N},(\mathcal{D}\Phi_{u_0}\circ \tau_N)_{\#}\Pi_N^{\mathbb{A}_N}\right)=\smallo_{\prob[u_0][N]{}}(1),\quad N\to\infty.
\end{equation*}
Furthermore, since $\mathcal{D}\Phi_{u_0}$ is bounded and linear, we have
\begin{align*}
    \mathcal{W}_{1,V}((\mathcal{D}\Phi_{u_0}\circ\tau_N)_{\#}\Pi_N^{\mathbb{A}_N},(\mathcal{D}\Phi_{u_0})_{\#}\gaussian) &\le \norm{\mathcal{D}\Phi_{u_0}}_{\mathcal{L}(\mathbb{V}_\gamma,V)}\mathcal{W}_{1,H^\gamma(\domain)}((\tau_N)_{\#}\Pi_N^{\mathbb{A}_N},\gaussian)\\
    &= \smallo_{\prob[u_0][N]{}}(1)
\end{align*}
by \eqref{eq:LocalisedBvMForDeltaMethod}.
In combination with \eqref{eq:TransformedLocalisation} and the
triangle inequality, the previous two bounds prove the claim.
\end{proof}

\subsection{Proofs for the MAP centring}

Let $V_J^\ast$ be the topological dual space to $(V_J,\norm{}_{H^m(\domain)})$ equipped with its dual norm. For $u\in V_J$ denote by $\mathcal{I}_{u_0,J}u$ the restriction of $\mathcal{I}_{u_0}u\in\mathcal{V}^\ast$ to $V_J$, i.e.
\begin{equation*}
    (\mathcal{I}_{u_0,J} u)(v)=\iprod{u}{v}_{\operatorname{LAN}},\quad u,v\in V_J.
\end{equation*}
By \eqref{eq:Normequivalence} it follows that $\mathcal{I}_{u_0,J}\colon V_J\to V_J^\ast$ is a topological isomorphism, where the norms $\norm{\mathcal{I}_{u_0,J}}_{\mathcal{L}(V_J,V_J^\ast)}$ and $\norm{\mathcal{I}_{u_0,J}^{-1}}_{\mathcal{L}(V_J^\ast,V_J)}$ are bounded uniformly in $J\in\mathbb{N}$. For the open neighbourhood $\mathcal{U}$ of Assumption \ref{assump:operator_linearisation} define
\begin{equation*}
    \mathcal{J}_N(u)\coloneqq -\frac{1}{N}\mathcal{D}^2 \ell_N^{\operatorname{reg}}(u),\quad u\in \mathcal{U},
\end{equation*}
such that $\mathcal{J}_N(u)\colon V_J\to V_J^\ast$ with
\begin{align*}
    \mathcal{J}_N(u)[v,v']&=\frac{1}{\sigma^2N}\sum_{i=1}^N \opDeriv{u}[v](X_i)\opDeriv{u}[v'](X_i)\\
    &\quad - \frac{1}{\sigma^2 N}\sum_{i=1}^N (Y_i-\op[u](X_i))\mathcal{D}^2\op_u[v,v'](X_i)+\frac{1}{N}\iprod{v}{v'}_{\mathbb{H}_N}
\end{align*}
for $v,v'\in V_J$. Furthermore, define the line segment
\begin{equation*}
    u_t^{\operatorname{MAP}}\coloneqq \ProjOne_J u_0 + t(\hat{u}_N^{\operatorname{MAP}} - \ProjOne_J u_0),\quad 0\le t\le 1,
\end{equation*}
and the operator $\bar{\mathcal{J}}_N\colon V_J\to V_J^\ast$ by
\begin{equation}
    \bar{\mathcal{J}}_N\coloneqq \int_0^1 \mathcal{J}_N(u_t^{\operatorname{MAP}})\,\D t.\label{eq:IntegratedMidPointOperator}
\end{equation}
\begin{lem}\label{lem:IntegratedIntermediatepoint_Invertibility}
     Grant Assumptions \ref{assump:ApproximatonSets}, \ref{assump:Operator_stability}, \ref{assump:operator_linearisation} and \ref{assump:operator_linearisation_2}. Let $m>d/2$, $0<\gamma<m-d/2$, consider the prior $\pi_N$ from \eqref{eq:Prior} with cut-off $2^J\sim N^{1/(2\beta+d)}$ and assume that $u_0\in \mathcal{V}\cap H^{m+\beta}(\domain)$ for $\beta>2d$, $d/2<\beta_0<\beta+d/4$. Then 
    \begin{equation*}
        \norm{\bar{\mathcal{J}}_N-\mathcal{I}_{u_0,J}}_{\mathcal{L}(V_J,V_J^\ast)} = \smallo_{\prob[u_0][N]{}}(\operatorname{dim}(V_J)^{-1/2}).
    \end{equation*}
    Moreover, with probability tending to one, $\bar{\mathcal{J}}_N$ is invertible and
    \begin{equation*}
        \norm{\bar{\mathcal{J}}_N^{-1}}_{\mathcal{L}(V_J^\ast,V_J)}=\mathcal{O}_{\prob[u_0][N]{}}(1)\text{ and }\norm{\bar{\mathcal{J}}_N^{-1}-\mathcal{I}_{u_0,J}^{-1}}_{\mathcal{L}(V_J^\ast,V_J)}=\smallo_{\prob[u_0][N]{}}(\operatorname{dim}(V_J)^{-1/2}).
    \end{equation*}
\end{lem}
\begin{proof}
    Consider an $\iprod{}{}_{H^m(\domain)}$ orthonormal basis $(\tilde{\psi}_k\colon k=1,\dots,\operatorname{dim}(V_J))$ of $V_J$. Note that the Bernstein inequality \eqref{eq:Bernstein} yields
    \begin{equation*}
        \sup_{k\in\mathbb{N}}\norm{\tilde{\psi}_k}_{W^{m,\infty}(\domain)}\lesssim 2^{J\eta}\sup_{k\in\mathbb{N}}\norm{\tilde{\psi}_k}_{H^m(\domain)}\le2^{J\eta} 
    \end{equation*}
    for any $\eta>d/2$. Note that
    \begin{align*}
        (\mathcal{J}_N(u_0)-\mathcal{I}_{u_0,J})[\tilde{\psi}_k,\tilde{\psi}_l] &= \left(\frac{1}{\sigma^2}\iprod{\opDeriv{u_0}[\tilde{\psi}_k]}{\opDeriv{u_0}[\tilde{\psi}_l]}_N - \iprod{\tilde{\psi}_k}{\tilde{\psi}_l}_{\operatorname{LAN}}\right)\\
        &\quad - \frac{1}{\sigma}\iprod{\eta}{\mathcal{D}^2\op_{u_0}[\tilde{\psi}_k,\tilde{\psi}_l]}_N + \frac{1}{N}\iprod{\tilde{\psi}_k}{\tilde{\psi}_l}_{\mathbb{H}_N},
    \end{align*}
    for $k,l\in\{1,\dots,\operatorname{dim}(V_J)\}$. Since $\norm{\opDeriv{u_0}[\tilde{\psi}_k]}_{L^\infty(\domain)}\lesssim \norm{\tilde{\psi}_k}_{W^{m,\infty}(\domain)}\le C 2^{J\eta}$ for some $0<C<\infty$ not depending on $k=1,\dots, \operatorname{dim}(V_J)$ or $J\in\mathbb{N}$, we find with the Bernstein inequality (Proposition 3.1.8 of \cite{GN16}) that
    \begin{equation*}
        \prob*[u_0][N]{\abs{\frac{1}{\sigma^2}\iprod{\opDeriv{u_0}[\tilde{\psi}_k]}{\opDeriv{u_0}[\tilde{\psi}_l]}_N - \iprod{\tilde{\psi}_k}{\tilde{\psi}_l}_{\operatorname{LAN}}}>x} \le 2\exp\left(-\frac{Nx^2}{C2^{2J\eta}+C2^{2J\eta}x}\right)
    \end{equation*}
    for any $x\ge 0$. Using Assumption \ref{assump:operator_linearisation_2} we find $\norm{\mathcal{D}^2\op_{u_0}[\tilde{\psi}_k,\tilde{\psi}_l]}_{L^\infty(\domain)}\le C 2^{2J\eta}$ uniformly in $k,l=1,\dots, \operatorname{dim}(V_J)$ and $J\in\mathbb{N}$ for some $0<C<\infty$. Since the Gaussian measurement noise $(\epsilon_1,\epsilon_2,\dots)$ is independent of the design $(X_1,X_2,\dots)$ we find
    \begin{equation*}
        \prob[u_0][N]{\abs{\iprod{\eta}{\mathcal{D}^2\op_{u_0}[\tilde{\psi}_k,\tilde{\psi}_l]}_N}>x\given (X_1,\dots, X_N)}\le 2\exp\left(-\frac{Nx^2}{C2^{4J\eta}}\right)
    \end{equation*}
    for any $x\ge 0$. Choosing $x=C(2^{J\eta}\sqrt{J/N} + 2^{2J\eta}J/N)$ and $x=C2^{2J\eta}\sqrt{J/N}$, respectively, the union bound yields
    \begin{align*}
        \max_{k,l=1,\dots,\operatorname{dim}(V_J)}\abs{\frac{1}{\sigma^2}\iprod{\opDeriv{u_0}[\tilde{\psi}_k]}{\opDeriv{u_0}[\tilde{\psi}_l]}_N - \iprod{\tilde{\psi}_k}{\tilde{\psi}_l}_{\operatorname{LAN}}} &=\mathcal{O}_{\prob[u_0][N]{}}\Big(2^{J\eta}\sqrt{\frac{J}{N}} + 2^{2J\eta}\frac{J}{N}\Big)\\
        \max_{k,l=1,\dots,\operatorname{dim}(V_J)}\abs{\iprod{\eta}{\mathcal{D}^2\op_{u_0}[\tilde{\psi}_k,\tilde{\psi}_l]}_N}& =\mathcal{O}_{\prob[u_0][N]{}}\Big(2^{2J\eta}\sqrt{\frac{J}{N}} \Big),
    \end{align*}
    since $\operatorname{dim}(V_J)\lesssim 2^{Jd}$.
    Furthermore, by the Bernstein estimate \eqref{eq:Bernstein} and \eqref{eq:CharacterisationSobolevSpaces} we have for any $\epsilon>0$
    \begin{equation*}
        \sup_{\norm{v}_{H^m(\domain)},\norm{v'}_{H^m(\domain)}\le 1}\frac{1}{N}\abs{\iprod{v}{v'}_{\mathbb{H}_N}}\lesssim \frac{1}{N}2^{2J(\beta_0+\epsilon)}.
    \end{equation*}
    Since the operator norm of a matrix is bounded by the dimension multiplied with the maximum of its entries we find
    \begin{equation}
        \norm{\mathcal{J}_N(u_0) - \mathcal{I}_{u_0,J}}_{\mathcal{L}(V_J,V_J^\ast)}
        %&\le \operatorname{dim}(V_J)\max_{k,l}\abs{\mathcal{J}_N(u_0) - \mathcal{I}_{u_0,J}}[\tilde{\psi}_k,\tilde{\psi}_l]\\
        = \mathcal{O}_{\prob[u_0][N]{}}\Big( \operatorname{dim}(V_J)2^{2J\eta}\Big(\sqrt{\frac{J}{N}}+\frac{J}{N}\Big) + \frac{2^{2J(\beta_0+\epsilon)}}{N}\Big).\label{eq:aux_TrueInformationvsAverages}
    \end{equation}
    We proceed to bound the error over the line segment $(\ProjOne_Ju_0 + t(\hat{u}_N^{\operatorname{MAP}} - \ProjOne_J u_0)\colon 0\le t\le 1)$. Since $\mathcal{U}$ from Assumption \ref{assump:operator_linearisation_2} is open, Lemma \ref{lem:MAP_Convergence} and the Jackson estimate \eqref{eq:Jackson} imply that the whole segment belongs to $\mathcal{U}$ with $\prob[u_0][N]{}$- probability tending to one. Furthermore, for any $u\in\mathcal{U}$ and $v,v'\in V_J$, we find
    \begin{align}
    \begin{split}
        \mathcal{J}_N(u)[v,v'] - \mathcal{J}_N(u_0)[v,v'] &= \frac{1}{\sigma^2 N}\sum_{i=1}^N (\opDeriv{u}[v]\opDeriv{u}[v'] - \opDeriv{u_0}[v]\opDeriv{u_0}[v'])(X_i)\\
        &\quad - \frac{1}{\sigma}\iprod{\eta}{(\mathcal{D}^2 \op_u-\mathcal{D}^2 \op_{u_0})[v,v']}_N\\
        &\quad - \frac{1}{\sigma^2}\iprod{\op[u_0]-\op[u]}{\mathcal{D}^2 \op_u[v,v']}_N.
        \end{split}
        \label{eq:aux_helperDecompMAPSegment}
    \end{align}
    We bound the three summands one by one. For the first one we find
    \begin{align*}
        \norm{\opDeriv{u}[v]\opDeriv{u}[v'] - \opDeriv{u_0}[v]\opDeriv{u_0}[v']}_{L^\infty(\domain)}\hspace{-5em}&\\
        &\le \norm{(\opDeriv{u}-\opDeriv{u_0})[v]\opDeriv{u_0}[v']}_{L^\infty(\domain)}\\
        &\quad + \norm{\opDeriv{u_0}[v](\opDeriv{u}-\opDeriv{u_0})[v']}_{L^\infty(\domain)}\\
        &\quad +\norm{(\opDeriv{u}-\opDeriv{u_0})[v](\opDeriv{u}-\opDeriv{u_0})[v']}_{L^\infty(\domain)}.
    \end{align*}
    Applying \eqref{eq:linearisation2} from Assumption \ref{assump:operator_linearisation} we find
    \begin{align*}
        \norm{\opDeriv{u}[v]\opDeriv{u}[v'] - \opDeriv{u_0}[v]\opDeriv{u_0}[v']}_{L^\infty(\domain)}\hspace{-11em}&\\
        &\lesssim (\norm{u-u_0}_{W^{m,\infty}(\domain)} +\norm{u-u_0}_{W^{m,\infty}(\domain)}^2)\norm{v}_{W^{m,\infty}(\domain)}\norm{v'}_{W^{m,\infty}(\domain)}.
    \end{align*}
    For the second summand in \eqref{eq:aux_helperDecompMAPSegment} we apply Assumption \ref{assump:operator_linearisation_2}  to obtain
    \begin{equation*}
        \norm{(\mathcal{D}^2\op_u - \mathcal{D}^2\op_{u_0})[v,v']}_{L^\infty(\domain)}\lesssim \norm{u-u_0}_{W^{m,\infty}(\domain)}\norm{v}_{W^{m,\infty}(\domain)}\norm{v'}_{W^{m,\infty}(\domain)}.
    \end{equation*}
    By the Lipschitz-bound \eqref{eq:Lipschitz_Bounds_infty} and Assumption \ref{assump:operator_linearisation_2} we control the third summand in \eqref{eq:aux_helperDecompMAPSegment} by
    \begin{equation*}
        \frac{1}{\sigma^2}\abs{\iprod{\op[u_0]-\op[u]}{\mathcal{D}^2 \op_u[v,v']}_N} \lesssim \norm{u-u_0}_{W^{m,\infty}(\domain)}\norm{v}_{W^{m,\infty}(\domain)}\norm{v'}_{W^{m,\infty}(\domain)}.
    \end{equation*}
    Consequently, choosing $v,v'$ with $\norm{v}_{H^m(\domain)},\norm{v'}_{H^m(\domain)}\le 1$ and using the Bernstein estimate \eqref{eq:Bernstein}, we find for any $\eta>d/2$ that 
    \begin{equation*}
        \norm{\mathcal{J}_N(u_t^{\operatorname{MAP}})-\mathcal{J}_N(u_0)}_{\mathcal{L}(V_J,V_J^\ast)}\lesssim \Big(1 + \frac{1}{N}\sum_{i=1}^N\abs{\epsilon_i}\Big)2^{2J\eta}\norm{u_t^{\operatorname{MAP}}-u_0}_{W^{m,\infty}(\domain)}
    \end{equation*}
    for all $0\le t\le 1$ on an event of $\prob[u_0][N]{}$-probability tending to one. Since $\norm{u_t^{\operatorname{MAP}}-u_0}_{W^{m,\infty}(\domain)}\le \norm{\hat{u}_N^{\operatorname{MAP}}-u_0}_{W^{m,\infty}(\domain)} + \norm{(\identity-\ProjOne_J)u_0}_{W^{m,\infty}(\domain)}=\mathcal{O}_{\prob[u_0][N]{}}(2^{J\eta}\epsilon_N)$ by Lemma \ref{lem:MAP_Convergence} and the Jackson and Bernstein estimates \eqref{eq:Jackson}, \eqref{eq:Bernstein}, and $\sum_{i=1}^N\abs{\epsilon_i}/N=\mathcal{O}_{\prob[u_0][N]{}}(1)$, we find
    \begin{equation}
        \sup_{0\le t\le 1}\norm{\mathcal{J}_N(u_t^{\operatorname{MAP}})-\mathcal{J}_N(u_0)}_{\mathcal{L}(V_J,V_J^\ast)} = \mathcal{O}_{\prob[u_0][N]{}}(2^{3J\eta}\epsilon_N).\label{eq:UniformHessianControl}
    \end{equation}
    In combination with \eqref{eq:aux_TrueInformationvsAverages} and the triangle inequality, we obtain the bound
    \begin{align*}
        \norm{\bar{\mathcal{J}}_N-\mathcal{I}_{u_0,J}}_{\mathcal{L}(V_J,V_J^\ast)}&\le \norm{\mathcal{J}_N(u_0)-\mathcal{I}_{u_0,J}}_{\mathcal{L}(V_J,V_J^\ast)} + \sup_{0\le t\le 1}\norm{\mathcal{J}_N(u_t^{\operatorname{MAP}})-\mathcal{J}_N(u_0)}_{\mathcal{L}(V_J,V_J^\ast)}\\
        &=\mathcal{O}_{\prob[u_0][N]{}}\Big(\operatorname{dim}(V_J) 2^{2J\eta}\Big(\sqrt{\frac{J}{N}}+\frac{J}{N}\Big) + \frac{2^{2J(\beta_0+\epsilon)}}{N} + 2^{3J\eta}\epsilon_N\Big)\\
        &=\mathcal{O}_{\prob[u_0][N]{}}\Big(\operatorname{dim}(V_J)^{-1/2}\Big( N^{\frac{d+2\eta-\beta}{2\beta+d}}\sqrt{\log(N)}+N^{\frac{d/2+2\eta-2\beta}{2\beta+d}}\log(N)\\
        &\quad + N^{\frac{2\beta_0+2\epsilon-2\beta-d/2}{2\beta+d}} + N^{\frac{d/2+3\eta-\beta}{2\beta+d}}\log(N)\Big)\Big)\\
        &= \smallo_{\prob[u_0][N]{}}(\operatorname{dim}(V_J)^{-1/2}),
    \end{align*}
    for $d/2<\eta<\min((\beta-d)/2,(\beta-d/2)/3)$ and $0<\epsilon<\beta+d/4-\beta_0$, which is possible since $\beta>2d$ and $d/2<\beta_0<\beta+d/4$.
    We are left to show invertibility. Note that $\mathcal{I}_{u_0,J}\colon V_J\to V_J^\ast$ is invertible with $\norm{\mathcal{I}_{u_0,J}^{-1}}_{\mathcal{L}(V_J^\ast,V_J)}\le C<\infty$ for all $J\in\mathbb{N}$. Define $\mathcal{T}_N\coloneqq -\mathcal{I}_{u_0,J}^{-1}(\bar{\mathcal{J}}_N - \mathcal{I}_{u_0,J})$, then $
        \norm{\mathcal{T}_N}_{\mathcal{L}(V_J,V_J)}=\smallo_{\prob[u_0][N]{}}(\operatorname{dim}(V_J)^{-1/2})$,
    such that $\norm{\mathcal{T}_N}_{\mathcal{L}(V_J,V_J)}\le 1/2$ on an event of $\prob[u_0][N]{}$-probability tending to one. On this event
    \begin{equation*}
        \bar{\mathcal{J}}_N = \mathcal{I}_{u_0,J}(\identity + \mathcal{I}_{u_0,J}^{-1}(\bar{\mathcal{J}}_N-\mathcal{I}_{u_0,J}))= \mathcal{I}_{u_0,J}(\identity- \mathcal{T}_N)
    \end{equation*}
    is invertible by a standard Neumann-series argument (Theorem II.1.12 of \cite{Werner2018}) and
    \begin{equation*}
        \norm{\bar{\mathcal{J}}_N^{-1}}_{\mathcal{L}(V_J^\ast,V_J)}\le \frac{\norm{\mathcal{I}_{u_0,J}^{-1}}_{\mathcal{L}(V_J^\ast,V_J)}}{1-\norm{\mathcal{T}_N}_{\mathcal{L}(V_J,V_J)}}=\mathcal{O}_{\prob[u_0][N]{}}(1).
    \end{equation*}     
    The final claim follows from 
    \begin{align*}
        \norm{\bar{\mathcal{J}}_N^{-1}-\mathcal{I}_{u_0,J}^{-1}}_{\mathcal{L}(V_J^\ast,V_J)}&\le \norm{\bar{\mathcal{J}}_N^{-1}}_{\mathcal{L}(V_J^\ast,V_J)} \norm{\bar{\mathcal{J}}_N-\mathcal{I}_{u_0,J}}_{\mathcal{L}(V_J,V_J^\ast)}\norm{\mathcal{I}_{u_0,J}^{-1}}_{\mathcal{L}(V_J^\ast,V_J)}\\
        &= \smallo_{\prob[u_0][N]{}}(\operatorname{dim}(V_J)^{-1/2}).\qedhere
    \end{align*}
\end{proof}

\begin{lem}\label{lem:ScoreExpansion}
     Grant Assumptions \ref{assump:ApproximatonSets}, \ref{assump:Operator_stability}, \ref{assump:operator_linearisation} and \ref{assump:operator_linearisation_2}. Let $m>d/2$, $0<\gamma<m-d/2$, consider the prior $\pi_N$ from \eqref{eq:Prior} with cut-off $2^J\sim N^{1/(2\beta+d)}$ and assume that $u_0\in \mathcal{V}\cap H^{m+\beta}(\domain)$ for $\beta>2d$, $d/2<\beta_0<\beta+d/4$. Then
    \begin{equation*}
        \sqrt{N}\norm{\frac{1}{N}\mathcal{D}\ell_N^{\operatorname{reg}}(\ProjOne_J u_0) - \mathcal{I}_{u_0,J}(\tilde{u}_N-\ProjOne_Ju_0)}_{V_J^\ast} \xrightarrow{\prob[u_0][N]{}}0
    \end{equation*}
    and
    \begin{equation*}
        \sqrt{N}\norm{\tilde{u}_N-\ProjOne_J u_0}_{H^m(\domain)}=\mathcal{O}_{\prob[u_0][N]{}}(\sqrt{\operatorname{dim}(V_J)})
    \end{equation*}
    as $N\to\infty$.
\end{lem}
\begin{proof}
    First note that for any $h\in V_J$ we have
    \begin{equation*}
        \mathcal{I}_{u_0,J}(\tilde{u}_N-\ProjTwo_J u_0)[h] = \iprod{\tilde{u}_N-\ProjTwo_J u_0}{h}_{\operatorname{LAN}}= \frac{1}{\sigma N}\sum_{i=1}^N \epsilon_i\opDeriv{u_0}[h](X_i)
    \end{equation*}
    by \eqref{eq:FromDualtoLAN} and \eqref{eq:EfficientEstimator_RandomDesign}. Since $\ProjTwo_J$ is $\iprod{}{}_{\operatorname{LAN}}$-orthogonal, we have 
    \begin{equation*}
        \iprod{(\ProjTwo_J  - \ProjOne_J)u_0}{h}_{\operatorname{LAN}} = \iprod{(\identity - \ProjOne_J)u_0}{h}_{\operatorname{LAN}}.
    \end{equation*}
    Moreover, we have
    \begin{align*}
        \frac{1}{N}\mathcal{D}\ell_N^{\operatorname{reg}}(\ProjOne_J u_0)[h]&= \frac{1}{\sigma}\iprod{\eta}{\opDeriv{\ProjOne_J u_0}[h]}_N + \frac{1}{\sigma^2}\iprod{\op[u_0]-\op[\ProjOne_J u_0]}{\opDeriv{\ProjOne_J u_0}[h]}_N\\
        &\quad - \frac{1}{N}\iprod{\ProjOne_J u_0}{h}_{\mathbb{H}_N}.
    \end{align*}
    In combination, the previous three expressions show that
    \begin{align*}
        \frac{1}{N}\mathcal{D}\ell_N^{\operatorname{reg}}(\ProjOne_J u_0)[h]-\mathcal{I}_{u_0,J}(\tilde{u}_N -\ProjOne_J u_0)[h] \hspace{-17em}&\\
        &=\frac{1}{\sigma}\iprod{\eta}{\opDeriv{\ProjOne_J u_0}[h] - \opDeriv{u_0}[h]}_N\\
        &\quad + \frac{1}{\sigma^2}\left(\iprod{\op[u_0]-\op[\ProjOne_J u_0]}{\opDeriv{\ProjOne_J u_0}[h]}_N - \iprod{\op[u_0]-\op[\ProjOne_J u_0]}{\opDeriv{\ProjOne_J u_0}[h]}_{L_p^2(\domain)}\right)\\
        &\quad +\frac{1}{\sigma^2}\iprod{\op[u_0]-\op[\ProjOne_J u_0]}{\opDeriv{\ProjOne_J u_0}[h]}_{L_p^2(\domain)} - \frac{1}{\sigma^2}\iprod{\opDeriv{u_0}[u_0-\ProjOne_J u_0]}{\opDeriv{u_0}[h]}_{L_p^2(\domain)}\\
        &\quad-\frac{1}{N}\iprod{\ProjOne_J u_0}{h}_{\mathbb{H}_N}\\
        &=\colon A[h]+B[h]+C[h]+D[h].
    \end{align*}
    Let $(\tilde{\psi}_k)_{k=1,\dots, \operatorname{dim}(V_J)}$ be an $\iprod{}{}_{H^m(\domain)}$-orthonormal basis of $V_J$.
    For $A$ we use the Jackson estimate \eqref{eq:Jackson} and \eqref{eq:linearisation2} to find
    \begin{align*}
        N\EV[\prob[u_0][N]{}]{\norm{A}_{V_J^\ast}^2}&= \frac{1}{\sigma^2}\sum_{k=1}^{\operatorname{dim}(V_J)}\norm{(\opDeriv{\ProjOne_J u_0}-\opDeriv{u_0})\tilde{\psi}_k}_{L_p^2(\domain)}^2\lesssim \operatorname{dim}(V_J) \norm{(\identity-\ProjOne_J)u_0}_{H^m(\domain)}^2\\
        &\lesssim \operatorname{dim}(V_J)2^{-2J\beta}.
    \end{align*}
    For $B$ we find
    \begin{align*}
        N\EV[\prob[u_0][N]{}]{\norm{B}_{V_J^\ast}^2}&=N\sum_{k=1}^{\operatorname{dim}(V_J)}\mathrm{E}_{\prob[u_0][N]{}}\Big[ \Big|\iprod{\op[u_0]-\op[\ProjOne_J u_0]}{\opDeriv{\ProjOne_J u_0}[\tilde{\psi}_k]}_N\\
        &\quad - \iprod{\op[u_0]-\op[\ProjOne_J u_0]}{\opDeriv{\ProjOne_J u_0}[\tilde{\psi}_k]}_{L_p^2(\domain)}\Big|^2\Big]\\
        &\lesssim \operatorname{dim}(V_J)\EV[\prob[u_0][N]{}]{(\op[u_0]-\op[\ProjOne_J u_0])(X_1)^2\opDeriv{\ProjOne_J u_0}[\tilde{\psi}_k](X_1)^2}\\
        &\lesssim \operatorname{dim}(V_J) \norm{\op[u_0]-\op[\ProjOne_J u_0]}_{L^\infty(\domain)}^2\norm{\opDeriv{\ProjOne_J u_0}[\tilde{\psi}_k]}_{L_p^2(\domain)}^2\\
        &\lesssim \operatorname{dim}(V_J) \norm{(\identity-\ProjOne_J)u_0}_{W^{m,\infty}(\domain)}^2\Big(\norm{\opDeriv{u_0}[\tilde{\psi}_k]}_{L_p^2(\domain)}\\
        &\quad +\norm{(\opDeriv{\ProjOne_J u_0} - \opDeriv{u_0})\tilde{\psi}_k}_{L_p^2(\domain)}\Big)^2\\
        &\lesssim \operatorname{dim}(V_J)2^{-2J(\beta-\eta)}(\norm{\tilde{\psi}_k}_{H^m(\domain)} + \norm{(\identity-\ProjOne_J)u_0}_{H^m(\domain)}\norm{\tilde{\psi}_k}_{H^m(\domain)})^2\\
        &\lesssim \operatorname{dim}(V_J)2^{-2J(\beta-\eta)}
    \end{align*}
    for any $\eta>d/2$, where we used the Sobolev embedding, \eqref{eq:Lipschitz_Bounds_infty},  \eqref{eq:Graphnormequivalence} and \eqref{eq:linearisation2}. For $C$ we find 
    \begin{equation*}
        \abs{C[h]}\lesssim \norm{(\identity-\ProjOne_J)u_0}_{H^m(\domain)}^2 \norm{h}_{H^m(\domain)}\lesssim 2^{-2J\beta}\norm{h}_{H^m(\domain)},
    \end{equation*}
    where we used \eqref{eq:linearisation_2}, \eqref{eq:Graphnormequivalence} and \eqref{eq:linearisation2}.
   We find $\norm{C}_{V_J^\ast}\lesssim 2^{-2J\beta}$. For $D$ we compute
   \begin{equation*}
       \abs{D[h]}\le \frac{1}{N}\norm{\ProjOne_J u_0}_{\mathbb{H}_N}\norm{h}_{\mathbb{H}_N}\lesssim \frac{1}{N}2^{J(\beta_0 +\epsilon +(\beta_0+\epsilon-\beta)_+)}
   \end{equation*} for any $\epsilon>0$
   by \eqref{eq:CharacterisationSobolevSpaces} and the Bernstein inequality \eqref{eq:Bernstein}. In combination the previous bounds show that
   \begin{align*}
       \sqrt{N}\norm{\frac{1}{N}\mathcal{D}\ell_N^{\operatorname{reg}}(\ProjOne_J u_0) - \mathcal{I}_{u_0,J}(\tilde{u}_N-\ProjOne_Ju_0)}_{V_J^\ast}\hspace{-18em}&\hspace{18em}\\
       &\le \sqrt{N}(\norm{A}_{V_J^\ast}+\norm{B}_{V_J^\ast}+\norm{C}_{V_J^\ast}+\norm{D}_{V_J^\ast})\\
       &=\mathcal{O}_{\prob[u_0][N]{}}\Big(\sqrt{\operatorname{dim}(V_J)}(2^{-J\beta} + 2^{-J(\beta-\eta)} + \sqrt{N}2^{-2J\beta}+N^{-1/2}2^{J(\beta_0+\epsilon+(\beta_0+\epsilon-\beta)_+)})\Big)\\
       &=\smallo_{\prob[u_0][N]{}}(1)
   \end{align*}
   if $\epsilon <\beta + d/4-\beta_0$ and $d/2<\eta<\min(\beta-d/2,3d/4)$. This shows the first claim. For the second claim we use the Jackson estimates \eqref{eq:Jackson} and \eqref{eq:Jackson_orthogonalprojection} to obtain
   \begin{align*}
       \sqrt{N}\norm{(\ProjTwo_J-\ProjOne_J)u_0}_{H^m(\domain)}&\le \sqrt{N}\Big(\norm{(\identity - \ProjTwo_J)u_0}_{H^m(\domain)}+\norm{(\identity - \ProjOne_J)u_0}_{H^m(\domain)}\Big)\\
       &\lesssim \sqrt{N}2^{-J\beta}\lesssim \sqrt{\operatorname{dim}(V_J)}.
   \end{align*}
\end{proof}

\begin{proof}[Proof of Theorem \ref{thm:BvM_MAP}]
Since in the proof of Theorem \ref{thm:BvM_RandomDesign} we established the BvM theorem for centring at $\tilde{u}_N$, it suffices to control the difference between $\tilde{u}_N$ and $\hat{u}_N^{\operatorname{MAP}}$.
By Lemmas \ref{lem:MAP_Convergence} and \ref{lem:IntegratedIntermediatepoint_Invertibility} with $\prob[u_0][N]{}$-probability tending to one the segment $(u_t^{\operatorname{MAP}} = \ProjOne_J u_0 + t(\hat{u}_N^{\operatorname{MAP}} - \ProjOne_Ju_0)\colon 0\le t\le 1)$ lies in $\mathcal{U}$ from Assumption \ref{assump:operator_linearisation} and the operator $\bar{\mathcal{J}}_N$ from \eqref{eq:IntegratedMidPointOperator} is invertible. The fundamental theorem of calculus (Proposition A.2.3) applied to $t\mapsto \frac{1}{N}\mathcal{D}\ell_N^{\operatorname{reg}}(u_t^{\operatorname{MAP}})$ and the condition $\mathcal{D}\ell_N^{\operatorname{reg}}(\hat{u}_N^{\operatorname{MAP}})=0$ imply
\begin{align*}
    \frac{1}{N}\mathcal{D}\ell_N^{\operatorname{reg}}(\ProjOne_J u_0)&= \frac{1}{N}\mathcal{D}\ell_N^{\operatorname{reg}}(\ProjOne_J u_0) - \frac{1}{N}\mathcal{D}\ell_N^{\operatorname{reg}}(\hat{u}_N^{\operatorname{MAP}})\\
    &=-\frac{1}{N}\int_0^1 \mathcal{D}^2\ell_N^{\operatorname{reg}}(u_t^{\operatorname{MAP}})[\hat{u}_N^{\operatorname{MAP}}-\ProjOne_J u_0]\,\D t=\bar{\mathcal{J}}_N (\hat{u}_N^{\operatorname{MAP}} - \ProjOne_J u_0).
\end{align*}
We then find
\begin{equation*}
    \hat{u}_N^{\operatorname{MAP}} - \tilde{u}_N = \bar{\mathcal{J}}_N^{-1}\Big(\frac{1}{N}\mathcal{D}\ell_N^{\operatorname{reg}}(\ProjOne_J u_0) - \mathcal{I}_{u_0,J}(\tilde{u}_N-\ProjOne_Ju_0) + (\mathcal{I}_{u_0,J} - \bar{\mathcal{J}}_N)(\tilde{u}_N - \ProjOne_Ju_0)\Big).
\end{equation*}
With $\prob[u_0][N]{}$-probability tending to one find
\begin{align*}
    &\sqrt{N}\norm{\hat{u}_N^{\operatorname{MAP}} - \tilde{u}_N}_{H^m(\domain)}\\
    &\qquad \lesssim \norm{\bar{\mathcal{J}}_N^{-1}}_{\mathcal{L}(V_J^\ast,V_J)}\sqrt{N}\norm{\frac{1}{N}\mathcal{D}\ell_N^{\operatorname{reg}}(\ProjOne_J u_0) - \mathcal{I}_{u_0,J}(\tilde{u}_N-\ProjOne_Ju_0)}_{V_J^\ast}\\
    &\qquad \qquad  + \norm{\bar{\mathcal{J}}_N^{-1}}_{\mathcal{L}(V_J^\ast,V_J)}\norm{\mathcal{I}_{u_0,J} - \bar{\mathcal{J}}_N}_{\mathcal{L}(V_J,V_J^\ast)}\sqrt{N}\norm{\tilde{u}_N - \ProjOne_Ju_0}_{H^m(\domain)}.
\end{align*}
By Lemmas \ref{lem:IntegratedIntermediatepoint_Invertibility}
and \ref{lem:ScoreExpansion} we find
\begin{equation}
    \sqrt{N}\norm{\hat{u}_N^{\operatorname{MAP}}-\tilde{u}_N}_{H^m(\domain)}= \smallo_{\prob[u_0][N]{}}(1)\label{eq:MAP_efficientEstimaotr}
\end{equation}
as $N\to\infty$.
\end{proof}

\begin{proof}[Proof of Lemma \ref{lem:Convergence_MAP_Asymptotic_Normality}]
    In the proof of Lemma \ref{lem:Convergence_PosteriorMean} we showed that $\sqrt{N}(\tilde{u}_N - \ProjTwo_J u_0)\xrightarrow{d}\gaussian$ in $H^\gamma(\domain)$ under $\prob[u_0][N]{}$. The claim follows from \eqref{eq:MAP_efficientEstimaotr} combined with Slutsky's Lemma.
\end{proof}

\begin{proof}[Proof of Theorem \ref{thm:LaplaceApproximation}]
    First note that combining the bounds \eqref{eq:aux_TrueInformationvsAverages} and \eqref{eq:UniformHessianControl} evaluated at $t=1$ implies that
    \begin{equation}
        \norm{\mathcal{J}_N(\hat{u}_N^{\operatorname{MAP}}) - \mathcal{I}_{u_0,J}}_{\mathcal{L}(V_J,V_J^\ast)}=\smallo_{\prob[u_0][N]{}}(1)\label{eq:Convergence_Empirical_Hessian}
    \end{equation}
    as $N\to\infty$. Since $\mathcal{I}_{u_0,J}$ is uniformly coercive (in $J\in\mathbb{N}$), we deduce that $\mathcal{J}_N(\hat{u}_N^{\operatorname{MAP}})$ is also invertible with $\prob[u_0][N]{}$-probability tending to one. On this event, define $\Gamma_{N,J}= N(0,\mathcal{J}_N(\hat{u}_N^{\operatorname{MAP}})^{-1})$ and for $K\le J$ let $\gaussian^K = (\ProjOne_K)_{\#}\gaussian$. Then we find
    \begin{align*}
        \mathcal{W}_{1,H^\gamma(\domain)}(\Gamma_{N,J},\gaussian) &\le \mathcal{W}_{1,H^\gamma(\domain)}(\Gamma_{N,J},(\ProjOne_K)_{\#}\Gamma_{N,J}) + \mathcal{W}_{1,H^\gamma(\domain)}((\ProjOne_K)_{\#}\Gamma_{N,J},\gaussian^K)\\
        &\quad + \mathcal{W}_{1,H^\gamma(\domain)}(\gaussian^K,\gaussian)\\
        &=\colon A_{N,K} + B_{N,K} + C_{N,K}.
    \end{align*}
    We proceed by bounding the three summands individually. For $A_{N,K}$ note that \eqref{eq:Convergence_Empirical_Hessian} combined with the uniform coercivity of $\mathcal{I}_{u_0,J}$ implies that $\mathcal{J}_N(\hat{u}_N^{\operatorname{MAP}})^{-1}\le 2\mathcal{I}_{u_0,J}^{-1}$ in Loewner order with $\prob[u_0][N]{}$-probability. Consequently, we find
    \begin{align*}
        A_{N,K}^2&\le \EV[X_{N,J}|\data\sim\Gamma_{N,J}]{\norm{(\identity - \ProjOne_K)X_{N,J}}_{H^\gamma(\domain)}^2\given\data}\\
        &\le 2 \EV[Z_J\sim N(0,\mathcal{I}_{u_0,J}^{-1})]{\norm{(\identity - \ProjOne_K)Z_J}_{H^\gamma(\domain)}^2}\\
        &\lesssim \sum_{K<\abs{\mu}\le J}2^{2\gamma\abs{\mu}}\EV[Z_J\sim N(0,\mathcal{I}_{u_0,J}^{-1})]{\iprod{Z_J}{\psi_\mu}_{L^2(\domain)}^2}\\
        &= \sum_{K<\abs{\mu}\le J}2^{2\gamma\abs{\mu}}\norm{\mathcal{I}_{u_0,J}^{-1}\psi_\mu}_{\operatorname{LAN}}^2\\
        &\lesssim \sum_{K<\abs{\mu}\le J}2^{2\gamma\abs{\mu}}\norm{\psi_\mu}_{H^{-m}(\domain)}^2\\
        &\lesssim 2^{-2K(m-\gamma-d/2-\epsilon)}
    \end{align*}
    for any $0<\epsilon<m-\gamma-d/2$. Consequently, we have
    \begin{equation*}
    \lim_{K\to\infty}\limsup_{N\to\infty}\prob[u_0][N]{A_{N,K}>\delta}=0
    \end{equation*}
    for any $\delta>0$.
    For $B_{N,K}$ first note the that covariances of the Gaussian measures involved are $\Sigma_{N,K}=\ProjOne_K\mathcal{J}_N(\hat{u}_N^{\operatorname{MAP}})^{-1}\ProjOne_K^\ast$ and $\Sigma_K = \ProjOne_K\mathcal{I}_{u_0}^{-1}\ProjOne_K^\ast$, respectively. Then
    \begin{align*}
        \Sigma_{N,K}-\Sigma_K &= \ProjOne_K(\mathcal{J}_N(\hat{u}_N^{\operatorname{MAP}})^{-1}-\mathcal{I}_{u_0,J}^{-1})\ProjOne_K^\ast + \ProjOne_K(\mathcal{I}_{u_0,J}^{-1}-\mathcal{I}_{u_0}^{-1})\ProjOne_K\\
        &=\ProjOne_K\mathcal{J}_N(\hat{u}_N^{\operatorname{MAP}})^{-1}(\mathcal{I}_{u_0,J}-\mathcal{J}_N(\hat{u}_N^{\operatorname{MAP}}))\mathcal{I}_{u_0,J}^{-1}\ProjOne_K^\ast + \ProjOne_K(\mathcal{I}_{u_0,J}^{-1}-\mathcal{I}_{u_0}^{-1})\ProjOne_K.
    \end{align*}
    The uniform coercivity of $\mathcal{I}_{u_0}$ and \eqref{eq:Convergence_Empirical_Hessian} imply that $\norm{\mathcal{J}_N(\hat{u}_N^{\operatorname{MAP}})^{-1}-\mathcal{I}_{u_0,J})}_{\mathcal{L}(V_J^\ast,V_J)}\xrightarrow{\prob[u_0][N]{}}0$ as $N\to\infty$. Since $\mathcal{I}_{u_0,J}^{-1} = \ProjTwo_J\mathcal{I}_{u_0}^{-1}$ on $V_K^\ast$ and $\ProjTwo_J\to \identity$ strongly on $\mathcal{V}$, we deduce that $\Sigma_{N,K}\xrightarrow{\prob[u_0][N]{}}\Sigma_K$ as $N\to\infty$ for every fixed $K$. Since the Gaussian measures $(\ProjOne_K)_{\#}\Gamma_{N,J}$ and $\gaussian^K$ are finite-dimensional, the convergence of the covariance matrices implies convergence in Wasserstein-1 distance.\\    
    The summand $C_{N,K}$ is controlled in Lemma \ref{lem:Convergence_Gaussian_Measures}.\\
    Define the transformation $\tau_N(u) = \sqrt{N}(u-\hat{u}_N^{\operatorname{MAP}})$, $u\in V_J$, then $(\tau_N)_{\#}\hat{\Pi}_N^{\operatorname{Lap}}(\MTemptyplaceholder|\data) = \Gamma_{N,J}$ and
    \begin{align*}
        \sqrt{N}\mathcal{W}_{1,H^\gamma(\domain)}\Big(\Pi_N(\MTemptyplaceholder|\data),\hat{\Pi}_N^{\operatorname{Lap}}(\MTemptyplaceholder|\data)\Big) &= \mathcal{W}_{1,H^\gamma(\domain)}((\tau_N)_{\#}\Pi_N(\MTemptyplaceholder|\data),\Gamma_{N,J})\\
        &\xrightarrow{\prob[u_0][N]{}}0,
    \end{align*}
    as $N\to\infty$.
\end{proof}

\section{Further technical results}

We observe the following bound between the Hellinger distance, Kullback-Leibler divergence and the $L_p^2(\domain)  $ distance.  For two measures $\mu,\nu$ on the same probability space with dominating measures $\lambda$ we define their Hellinger distance
\begin{equation*}
    h^2(\mu,\nu)\coloneqq \int\left(\sqrt{\frac{\D\mu}{\D\lambda}}-\sqrt{\frac{\D \nu}{\D\lambda}}\right)^2\,\D\lambda
\end{equation*}
and their Kullback-Leibler divergence
\begin{equation*}
    \operatorname{KL}(\mu,\nu) \coloneqq 
    \begin{cases}
        \int \frac{\D\mu}{\D \nu}\log\left(\frac{\D\mu}{\D\nu}\right)\,\D\nu,&\mu\ll\nu,\\
        \infty,&\text{else.}
    \end{cases}
\end{equation*}
\begin{lem}\label{lem:RandomDesign_Hellinger/KL}
Suppose that $u,v\in \mathcal{V}$ with $\op[u],\op[v]\in C(\domain)$ and let 
\[S\coloneqq S(u,v) \coloneqq\| \op[u]-\op[v]\|_{L^\infty(\domain)}.\]
Then the Hellinger distance $h(\prob[u]{}, \prob[v]{})$ satisfies
    \[ \frac{2(1 - e^{-\frac{S^2}{8\sigma^2}})}{S^2} \norm{\op[u]-\op[v]}_{L_p^2(\domain)}^2 \leq h^2(\prob[u]{}, \prob[v]{}) \leq \frac{1}{4\sigma^2} \norm{\op[u]-\op[v]}_{L_p^2(\domain)}^2. \]
Moreover, we have the following expression for the KL-divergence:
\[ \operatorname{KL}(\prob[u]{},\prob[v]{}) = \frac{1}{2\sigma^2} \norm{\op[u]-\op[v]}_{L_p^2(\domain)}^2, \]
and
\begin{equation*}
    \EV[\prob[u][]{}][][\Big]{ \Big(\log \Big(\frac{\D\prob[u]{}(X_1,Y_1)}{\D\prob[v]{}(X_1,Y_1)}\Big)\Big)^2}
    %&\le \frac{1}{2\sigma^4}\norm{\op[u]-\op[v]}_{L^4(\domain)}^4 + \frac{2}{\sigma^2}\norm{\op[u]-\op[v]}_{L_p^2(\domain)}^2\\
    \le \frac{S^2+4\sigma^2}{4\sigma^4} \norm{\op[u]-\op[v]}_{L_p^2(\domain)}^2.
\end{equation*}
\end{lem}
\begin{proof}
    Compare Lemmas 22 and 23 of \cite{GN20}.
\end{proof}

We need the following lemma for the concentration of the empirical norm.
\begin{lem}[Concentration of empirical norm]\label{lem:EmpiricalNorm} Let $\mathcal{O} \subseteq \mathbb R^d$ be some open subset with finite Lebesgue measure, and let $X_i \overset{i.i.d.}{\sim} p(x)\,\D x$. Let $\mathcal{W}$ be a class of real-valued functions on $\mathcal{O}$ with uniform upper bound
\[
\sup_{u \in \mathcal{W}} \|u\|_{L^\infty(\mathcal{O})} =: M < \infty,
\]
where $\|u\|_{L^\infty(\mathcal{O})} = \sup_{x \in \mathcal{O}} |u(x)|$. There exists a universal constant $C > 0$ such that for any $(\delta_N : N \geq 1)$ satisfying
\begin{equation}
N \delta_N^2 \geq C M^2 \log(N(\mathcal{W}, \|\cdot\|_\infty, \delta_N)),\label{eq:aux_Entropycondition}
\end{equation}
and any $R \geq C(\delta_N + M/\sqrt N)$, it holds that
\[
\mathbb P\left(\|u\|_{L_p^2(\mathcal{O})} \geq 2\|u\|_N \text{ for some } u \in \mathcal{W} \text{ with } \|u\|_{L_p^2(\mathcal{O})} \geq R\right)
\leq 2\exp\left(-\frac{NR^2}{CM^2}\right).
\]
\end{lem}
\begin{proof}
    This is a direct specialisation of Lemma A.5 in \cite{reinhardtStatisticalLearningTheory2024} to $\mathcal{W}$ with uniform envelope $M = \sup_{u\in \mathcal{W}} \norm{u}_{L^\infty(\mathcal{O})}$.
\end{proof}

We proceed with studying the $\iprod{}{}_{\operatorname{LAN}}$-orthogonal projection $\ProjTwo_J\colon \mathcal{V}\to V_J$.

\begin{lem}\label{lem:Properties_LAN_Norm}
The orthogonal projection $\ProjTwo_{J}$ of $(\mathcal{V},\iprod{}{}_{\operatorname{LAN}})$ onto $V_{J}$ satisfies the following properties:
    \begin{enumerate}[(i)]
        \item\label{num:ProjectionBoundedness} Boundedness: For $u\in \mathcal{V}$ we have $\norm{\ProjTwo_{J} u}_{H^m(\domain)}\le \frac{C}{c}\norm{u}_{H^m(\domain)}$ for $0<c\le C<\infty$ from \eqref{eq:Normequivalence}.
        \item \label{num:ProjectionDensity} Density in $(\mathcal{V},\norm{}_{H^m(\domain)})$: For any $u\in \mathcal{V}$ we have 
        \begin{equation*}
            \norm{(\identity - \ProjTwo_J)u}_{H^m(\domain)}\to 0,
        \end{equation*}
        as $J\to\infty$.
        \item\label{num:ProjectionJackson} Jackson estimate: For $m\le s< S$ and $u\in \mathcal{V}\cap H^s(\domain)$ we have
        \begin{equation}
            \norm{(\identity-\ProjTwo_{J})u}_{H^m(\domain)}\lesssim 2^{-J(s-m)}\norm{u}_{H^s(\domain)}.\label{eq:Jackson_orthogonalprojection}
        \end{equation}
        \item\label{num:ProjectionBernstein} Bernstein estimate: For $m\le s< t< S$ and $u\in \mathcal{V}\cap H^s(\domain)$ we have
        \begin{equation}
            \norm{\ProjTwo_{J}u}_{H^t(\domain)} \lesssim 2^{J(t-s)}\norm{u}_{H^s(\domain)}.\label{eq:Bernstein_orthogonalprojection}
        \end{equation}
    \end{enumerate}
\end{lem}
\begin{proof}
    Since $V_J$ is a closed linear subspace of $(\mathcal{V},\iprod{}{}_{\operatorname{LAN}})$, there exists a unique $\iprod{}{}_{\operatorname{LAN}}$-orthogonal projection $\ProjTwo_J\colon (\mathcal{V},\iprod{}{}_{\operatorname{LAN}})\to V_J$. The Property \eqref{num:ProjectionBoundedness} follows from the norm equivalence \eqref{eq:Normequivalence} and boundedness of $\ProjTwo_J$ in $\norm{}_{\operatorname{LAN}}$-norm. To show \eqref{num:ProjectionDensity} take any $u\in \mathcal{V}$ and observe for any $v\in V_J$ that
    \begin{align*}
        \norm{(\identity- \ProjTwo_J)u}_{H^m(\domain)}&\le \frac{1}{c}\norm{(\identity- \ProjTwo_J)u}_{\operatorname{LAN}} \le \frac{1}{c}\left(\norm{u-v}_{\operatorname{LAN}} + \norm{\ProjTwo_J(u-v)}_{\operatorname{LAN}}\right)\\
        &\le \frac{2}{c}\norm{u-v}_{\operatorname{LAN}}\le \frac{2C}{c}\norm{u-v}_{H^m(\domain)}\le \frac{2C}{c}\norm{(\identity- \ProjOne_J)u}_{H^m(\domain)}.
    \end{align*}
    To show that this tends to zero as $J\to\infty$, we use the uniform $\norm{}_{H^m(\domain)}$-stability of $\ProjOne_J$ from \eqref{eq:H^m_stability} and density of the approximation spaces in $(\mathcal{V},\norm{}_{H^m(\domain)})$:
    \begin{align}
    \begin{split}
        \norm{(\identity-\ProjOne_J)u}_{H^m(\domain)}&\le \inf_{v\in V_J}\left(\norm{u-v}_{H^m(\domain)} + \norm{\ProjOne_J(u-v)}_{H^m(\domain)}\right)\\
        &\lesssim \inf_{v\in V_J} \norm{u-v}_{H^m(\domain)}\to 0
    \end{split}\label{eq:aux_DensityOriginalprojections}
    \end{align}
    as $J\to\infty$.
    The Jackson estimate \eqref{num:ProjectionJackson} follows from the best-approximation property of $\ProjTwo_J$, the norm equivalence \eqref{eq:Normequivalence} and the Jackson estimate \eqref{eq:Jackson}: For any $u\in \mathcal{V}\cap H^s(\domain)$ we have
    \begin{align*}
        \norm{(\identity-\ProjTwo_J)u}_{H^m(\domain)}&\le C \norm{(\identity-\ProjTwo_J)u}_{\operatorname{LAN}} = \inf_{v\in V_J}C\norm{u-v}_{\operatorname{LAN}}\le C\norm{u-\ProjOne_J u}_{\operatorname{LAN}}\\
        &\le \frac{C}{c}\norm{(\identity -\ProjOne_J) u}_{H^m(\domain)}\lesssim 2^{-J(s-m)}\norm{u}_{H^s(\domain)}.
    \end{align*}
    It remains to show the Bernstein estimate \eqref{num:ProjectionBernstein}.
    Since $\ProjTwo_J$ is the identity on $V_J$, we find by the wavelet Bernstein estimate \eqref{eq:Bernstein}, the boundedness \eqref{num:ProjectionBoundedness} and the wavelet Jackson estimate \eqref{eq:Jackson} that
    \begin{align*}
        \norm{\ProjTwo_J u}_{H^s(\domain)}&\le \norm{\ProjOne_J u}_{H^s(\domain)} + \norm{\ProjTwo_J(\identity - \ProjOne_J)u}_{H^s(\domain)}\\
        &\lesssim \norm{u}_{H^s(\domain)} + 2^{J(s-m)}\norm{\ProjTwo_J (\identity - \ProjOne_J)u}_{H^m(\domain)}\\
        &\lesssim \norm{u}_{H^s(\domain)} + 2^{J(s-m)}\norm{(\identity - \ProjOne_J)u}_{H^m(\domain)}\\
        &\lesssim \norm{u}_{H^s(\domain)} + 2^{J(s-m)}2^{J(m-s)}\norm{u}_{H^s(\domain)}\\
        &\lesssim \norm{u}_{H^s(\domain)}.
    \end{align*}
    Since $\ProjTwo_J u\in V_J$ for any $u\in\mathcal{V}$ we find with the Bernstein estimate \eqref{eq:Bernstein} that
    \begin{equation*}
        \norm{\ProjTwo_J u}_{H^t(\domain)}\lesssim 2^{J(t-s)}\norm{\ProjTwo_J u}_{H^s(\domain)}\lesssim 2^{J(t-s)}\norm{u}_{H^s(\domain)}.\qedhere
    \end{equation*}
    %\begin{align*}%for s\le 2\le t
    %     \norm{\ProjTwo_{J}u}_{H^t(\domain)}&\lesssim 2^{J(t-2)}\norm{\ProjTwo_{J}u}_{H^2(\domain)}\le \frac{C}{c}2^{J(t-2)}\norm{u}_{H^2(\domain)}\\
    %     &\lesssim 2^{J(t-s)}\norm{u}_{H^s(\domain)}
    %\end{align*}
\end{proof}

\end{document}